\documentclass{article}
\usepackage{amsthm, appendix}
\theoremstyle{plain}
\usepackage{tikz}
\usetikzlibrary{calc}
\newtheorem{theorem}{Theorem}[section]
\usepackage{pgfplots}
\pgfplotsset{compat=1.18}
\usepgfplotslibrary{groupplots}
\usepackage{graphicx} % Required for inserting images
\usepackage{subfigure}
\usepackage{subcaption}
\usepackage{graphicx} % for resizebox
\usepackage{algpseudocode}
\usepackage{amsmath,amsthm,amssymb,amsfonts,mathtools}
\usepackage[margin=1in]{geometry}
\usepackage{enumitem}
\usepackage{bbm}
\usepackage[numbers]{natbib}

\theoremstyle{plain}
\newtheorem{lemma}[theorem]{Lemma}
\newtheorem{proposition}[theorem]{Proposition}

\newtheorem{definition}[theorem]{Definition}
\newtheorem{remark}[theorem]{Remark}
\newtheorem{assumption}[theorem]{Assumption}
\theoremstyle{plain}

\newcommand{\wrap}{\operatorname{wrap}_\pi}
\newcommand{\R}{\mathbb{R}}

\newcommand{\E}{\mathbb{E}}
\newcommand{\spm}{S_\pm^{(t)}}
\newcommand{\spp}{S_+^{(t)}}
\newcommand{\sm}{S_-^{(t)}}
\usepackage{algorithm,caption}
\usepackage{tikz}	
\usepackage{xcolor}
\usepackage[numbers]{natbib}
\usepackage{tabularx}
\usepackage{tcolorbox,tikz, lipsum}

\tcbuselibrary{skins}
\tcbuselibrary{breakable}

\usepackage[textsize=footnotesize]{todonotes}

\title{Topological trapping in circular midpoint opinion dynamics}

\title{Topological trapping in circular midpoint opinion dynamics}

\author{
Annika Brockhaus$^{1}$\thanks{\texttt{a.brockhaus@math.leidenuniv.nl}}
\and
Wioletta M.\ Ruszel$^{2}$\thanks{\texttt{w.m.ruszel@uu.nl}}
\and
Cristian Spitoni$^{2}$\thanks{\texttt{c.spitoni@uu.nl}}
}

\date{}

\makeatletter
\renewcommand{\ALG@name}{Stochastic dynamics} %Change the name Algorithm to Algoritme
\makeatother
\begin{document}

\maketitle
\begin{center}
{\small
$^{1}$Mathematical Institute, Leiden University, Leiden, The Netherlands\\
$^{2}$Mathematical Institute, Utrecht University, Utrecht, The Netherlands
}
\end{center}
\begin{abstract}
We study a discrete-time asynchronous midpoint dynamics on the circle in which,
at each step, a uniformly chosen neighboring pair moves to the midpoint along the
shortest arc. Although the update rule is locally contractive, we show that the global
relaxation mechanism depends sharply on the boundary topology.

Under open boundary conditions the system
converges almost surely to consensus through pure contraction. Under periodic
boundary conditions the graph contains a single cycle, and the wrapped edge
increments define an integer-valued winding number. While consensus remains the
unique absorbing state for every fixed system size, we show that topology
profoundly reshapes the transient dynamics.

We prove that branch--crossings are the only mechanism capable of modifying the
winding number and compute explicitly their probability for disordered initial
data. Local averaging rapidly suppresses large gradients and drives the system
into a no-branch-crossing regime where the winding number freezes. Inside a
fixed winding sector we construct an adaptive co-moving frame in which the
dynamics becomes an exact Euclidean midpoint process and establish strict contraction in a co–moving frame toward a twisted linear profile
determined by the winding number.

Our results isolate a minimal
mechanism by which a single cycle induces sector locking and escape,
even though the final equilibrium remains unchanged.
\end{abstract}

\section{Introduction}

Models of opinion dynamics and synchronization provide paradigmatic
examples of how collective order emerges from local interactions \cite{piko1985}.
In bounded–confidence models such as the
Deffuant--Weisbuch model~\cite{Deffuant2000,Weisbuch2002}
and the Hegselmann--Krause model~\cite{hegselmann2015,Lorenz2007},
real-valued opinions evolve through averaging rules that produce
consensus, clustering, or fragmentation depending on the interaction
structure.
In a complementary direction, coupled oscillator models,
most notably the Kuramoto model~\cite{Kuramoto1975,Strogatz2000,Acebron2005}
and its network extensions~\cite{Wiley2006,Marvel2009,Dorfler2014},
describe phase variables on the circle and reveal how geometry and
topology shape synchronization patterns.

On graphs containing cycles the configuration space naturally
decomposes into topological sectors labeled by an integer winding
number.
In many models of coupled phase oscillators on rings, such as the
Kuramoto model and its nonlocal variants, configurations with a
constant phase difference between neighboring oscillators form exact
equilibrium solutions, usually referred to as \emph{twisted states}
\cite{Wiley2006,GirnykHaslerMaistrenko2012,
DelabaysJacquodTyloo2017}.
Twisted states also arise as equilibrium solutions in deterministic
Kuramoto models on structured graphs, where their stability depends on
the network topology 
\cite{OmelchenkoWolfrumLaing2014,MedvedevWright2017,medvedev2015stability}. In stochastic Kuramoto-type models, twisted states may exhibit genuine
metastability under small noise, with trajectories remaining trapped
near a given winding sector for exponentially long times before
transitioning to another one \cite{Berglund2025}.

The situation in the present model is fundamentally different.
Because each update replaces a neighboring pair by its midpoint along
the shortest arc, a perfectly linear phase profile is not invariant
under the dynamics. Twisted configurations therefore do not correspond
to stationary states.
Instead they arise as long-lived transient structures generated by the
topological constraint imposed by the winding number.
The system contracts rapidly toward such twisted profiles but
eventually escapes through rare branch--crossing events, after which
the winding number decreases until consensus is reached.

The present work studies a minimal discrete-time model that lies at the
interface between opinion dynamics and phase synchronization.
We consider $N$ agents with opinions $\theta(i)\in[-\pi,\pi)$ evolving
by asynchronous short-arc midpoint averaging:
at each step a neighboring pair $(i,i{+}1)$ is selected uniformly at
random and both agents move to the midpoint along the shortest arc.
The rule is locally contractive and closely related to linear gossip
algorithms~\cite{Jadbabaie2003}, yet the circular geometry introduces a
nonlinear wrapping constraint that couples local averaging to global
topology.

Our focus is on how boundary topology interacts with this local
contractive mechanism.
On a finite chain the midpoint rule induces global contraction and the
system converges almost surely to consensus.
On the ring, however, the presence of a single cycle introduces a
topological invariant, the winding number $W$, measuring the total
signed increment around the graph.
Although consensus remains the unique absorbing state for every fixed
$N$, this topological constraint fundamentally alters the transient
dynamics.

We show that the winding number can change only through
\emph{branch--crossing} events, in which a midpoint update pushes a
local difference outside the principal branch.
Starting from disordered initial data, such events occur with positive
probability.
However, local averaging rapidly suppresses large gradients and drives
the system into a no--branch--crossing corridor.
From that time onward the winding number freezes and the dynamics
evolves within a fixed topological sector.

Our main result is that this sector locking traps the dynamics within a fixed winding sector for long time intervals.
Inside each winding sector the configuration decomposes into a linear
trend of slope $\beta = 2\pi W/N$ and fluctuations.  We construct an adaptive co-moving frame in which the detrended
dynamics becomes an exact Euclidean midpoint process and prove
strict contraction. In this frame the dynamics relaxes toward a
trapping manifold determined by the winding number.
Long-lived trapping thus emerges without competing equilibria: topology does not modify the final absorbing state, but reorganizes the relaxation mechanism by confining the dynamics within winding sectors for extended periods.

By comparing open and periodic boundary conditions within the same
contractive discrete model, we isolate a minimal mechanism by which a
single cycle induces topological trapping and slow relaxation.

The paper is organized as follows.
Section~\ref{sec:model} introduces the model and the notation.
Section~\ref{sec:consensus} establishes almost-sure convergence to consensus under both open and periodic boundary conditions.
Section~\ref{sec:meta} presents the main results on topological trapping in winding sectors.
Section~\ref{sec:meta1} identifies branch--crossings as the only mechanism that can modify the winding number and computes their probability for disordered initial data. In Section~\ref{sec:lift} we introduce a lifted representation of circular configurations on the universal covering space $\mathbb{R}$, where wrapped increments become ordinary differences and midpoint updates reduce to arithmetic midpoints.
Section~\ref{sec:euri} develops a heuristic picture of the long-lived winding configurations observed in the dynamics, while Section~\ref{sec:meta0} proves relaxation toward these configurations using an adaptive co-moving frame and a contraction argument.
Finally, Section~\ref{sec:escape-heur} describes a mechanism by which the dynamics can leave a winding sector, and Section~\ref{sec:future} concludes.
%
%
%----------------------------------------
% SECTION 2: MODEL DEFINITION, DYNAMICS
%----------------------------------------
%
%
\section{Model Definition and Dynamics}
\label{sec:model}
\subsection{Notation}

Let us define a graph $G=(V,\mathcal{E})$ as follows. For $N\ge2$ let $V=V_N=\{1,2,\dots,N\}$ denote the set of vertices.
We consider two possible edge sets $\mathcal{E}$.
Under open (empty) boundary conditions the set of edges is the set
$\mathcal E_{\emptyset}=\{\{i,i+1\}:1\le i\le N-1\}$,
while under periodic boundary conditions we have the ring $\mathcal E_{p}
=\mathcal E_{\emptyset}\cup\{\{N,1\}\}$. We will call $\mathcal{E}_{\emptyset}$ the \emph{open path} and $\mathcal{E}_p$ the \emph{ring}.
Throughout, $\mathcal E$ denotes either $\mathcal E_{\emptyset}$ or $\mathcal E_p$. Trivially, we have that $|\mathcal{E}_{\emptyset}|=N-1$ and $|\mathcal{E}_p|=N$.

For periodic boundary conditions, indices are interpreted cyclically: 
for $i \in \mathbb{Z}$ we define
$[i]_N := 1 + \big((i-1) \bmod N\big)
$,
so that $[i]_N \in \{1,\dots,N\}$. 
In particular, $N+1 \equiv 1$ and $0 \equiv N$. 
Whenever expressions such as $i+1$ or $i-1$ appear under periodic boundary conditions, 
they are implicitly replaced by $[i\pm1]_N$.

Each vertex $i\in V_N$ carries an opinion $\theta(i)\in\mathbb S^1$.
We identify $\mathbb S^1$ with the interval $[-\pi,\pi)$
and represent configurations as vectors
\[
\theta=(\theta(i))_{i\in V}\in\Omega_N:=[-\pi,\pi)^{V_N}.
\]
The configuration space $\Omega_N$ is endowed with its Borel $\sigma$-algebra,
denoted by $\mathcal B(\Omega_N)$.
For $a\in\mathbb R$, define the wrap operator
\begin{equation}
\label{e:wrap}
\wrap(a)
:=
a-2\pi\Big\lfloor\frac{a+\pi}{2\pi}\Big\rfloor,
\end{equation}
so that $\wrap(a)\in[-\pi,\pi)$.
All angular differences are computed using this operator,
that is, as shortest signed differences in $[-\pi,\pi)$.
We have the following properties of the wrap operator:
\begin{equation}
\label{eq:wrap1}
\wrap(\wrap(x)\pm y)=\wrap(x\pm y),
\end{equation}
for all $x,y\in[-\pi,\pi)$, since $\wrap(x)=x-2\pi \ell$, for some $\ell\in\mathbb{Z}$.
We fix the orientation of edges by declaring $(i,i+1)$
to be oriented from $i$ to $i+1$ for $i=1,\dots,N-1$,
and, in the periodic case, $(N,1)$ to be oriented from $N$ to $1$.
For an oriented edge $e=(i,j)\in\mathcal E$ we define the wrapped increment:
\begin{equation}
\label{e:deltae}
\delta_\theta(e)
:=
\wrap\!\big(\theta(j)-\theta(i)\big)
\in[-\pi,\pi).
\end{equation}
Since the interaction graph of $\mathcal{E}$ is a path or a ring, each oriented edge is uniquely
determined by its left endpoint. Hence, with a slight abuse of notation,
we index edge increments by vertices and write for $i=1,\ldots, N$,
\begin{equation}
\label{s:deltai}
\delta(i) := \delta_\theta((i,i+1)),
\end{equation}
with the convention that $(N,N+1)=(N,[N+1]_N)=(N,1)$ in the periodic case.
The wrapped increments $(\delta(i))_i$ encode the local angular
differences along the oriented edges of the graph. A natural global quantity associated with the configuration is the total angular increment
obtained by summing its components.
We formalize this distinction in the following definition.
\begin{definition} [Total increment and winding number]
\label{def:winding}
Define the total increment and winding number respectively by:
\begin{equation}
\label{e:w}
m(\theta):=\sum_{e\in\mathcal E}\delta_\theta(e),
\qquad
W(\theta):=\frac{1}{2\pi}m(\theta).
\end{equation}
\end{definition}
\begin{remark}
\label{rem:wind}
 If $\mathcal E=\mathcal E_p$, then $W(\theta)\in\mathbb Z$.
If $\mathcal E=\mathcal E_{\emptyset}$, then $W(\theta)\in\mathbb R$ in general.   
In fact, for periodic boundary conditions, for each $i$ there exists $k_i\in\mathbb{Z}$ such that
\[
\delta(i)=\theta(i+1)-\theta(i)-2\pi k_i,
\]
since $\wrap(x)=x-2\pi k$ for a suitable integer $k\in\mathbb{Z}$.  Summing over 
$i=1,\dots,N$ and using periodicity $\theta(N+1)=\theta(1)$ gives
\[
\sum_{i=1}^{N}\delta(i)
=
\sum_{i=1}^{N}\big(\theta(i+1)-\theta(i)\big)
-
2\pi\sum_{i=1}^{N}k_i
=
-2\pi\sum_{i=1}^{N}k_i.
\]
Hence,
\[
W(\theta)=\frac{1}{2\pi}\sum_{i=1}^{N}\delta(i)\in\mathbb Z.
\]
\end{remark}

% SUBSECTION: stochastic dynamics
\subsection{Stochastic dynamics}
\label{ref:acca}
We will use  a \emph{random-sequential} averaging process which can be viewed as a sort of \emph{Asynchronous Continuous-State Cellular Automaton} (\texttt{ACCA}): a  discrete time,  continuous space Markov process  $(\theta_t)_{t \geq 0}$ with deterministic local evolution given by local averages.
More precisely, given the edge set $\mathcal{E}$, with  $\mathcal{E}\in\{\mathcal{E}_\emptyset,\mathcal{E}_p\}$, the time evolution up to time $T\in\mathbb{N}_0$ is described by the following algorithm:
\begin{algorithm}[H] 
\caption*{Asynchronous Continuous-state Cellular Automaton (\texttt{ACCA})}
    
\begin{enumerate}
    \item \textbf{Initialization:} At $t=0$, start from an initial configuration $\theta_0 \in \Omega_N$.
    %such that  $\theta_0(i)\overset{i.i.d.}{\sim}\textnormal{Unif}([-\pi,\pi])$ for each $i\in V_N$
    
    \item \textbf{Select edge $e$:} At the updated time $t \gets t + 1$\ select an edge $e = (k, k+1) \in \mathcal{E}$ uniformly at random, i.e., $e\overset{}{\sim}\textnormal{Unif}(\mathcal{E})$.
    \item \textbf{Update opinions of edge} $e$: Update only the opinions at $e=(k,k+1)$, by setting:
\begin{equation}
\left\{
\begin{array}{ll}
&\theta_{t+1}(k)=  
\wrap\!\big(\theta_t(k) + \tfrac12 \delta_t(k)\big),\\
&\theta_{t+1}(k+1)=
\wrap\!\big(\theta_t(k+1) - \tfrac12 \delta_t(k)\big),
\end{array}
\right.
\label{eq:update2}
\end{equation}
with $\delta_t(k):=\delta_{\theta_t}(k) =\wrap(\theta_t(k{+}1)-\theta_t(k))$. Notice that
$\theta_{t+1}(k)=\theta_{t+1}(k+1)$ (see Remark~\ref{rem:midpoint-correct}).
\\
If $\delta_t(k)=-\pi$ the two vertices are antipodal (see Figure~\ref{fig:acca}) and there are
two possible shortest midpoints on the circle. In this case we choose
one of them with probability $1/2$ each.

    \item \textbf{Repeat steps 2,3 } until $t=T$.
\end{enumerate}
    \end{algorithm}
The \texttt{ACCA} dynamics defines a time-homogeneous discrete-time Markov process
$(\theta_t)_{t\ge0}$ on the measurable space
$(\Omega_N,\mathcal B(\Omega_N))$.
When $|\delta_\theta(e)|\neq \pi$, the midpoint update is deterministic
and defines for $e=(k,k+1)\in\mathcal{E}$ a measurable map $F_e:\Omega_N\to\Omega_N$ given by
\[
(F_e(\theta))(i)=
\begin{cases}
\wrap\!\big(\theta(k)+\tfrac12\,\delta_\theta(e)\big), & i=k,\\[4pt]
\wrap\!\big(\theta(k+1)-\tfrac12\,\delta_\theta(e)\big), & i=k+1,\\[4pt]
\theta(i), & \text{otherwise.}
\end{cases}
\]

In the antipodal case $|\delta_\theta(e)|=\pi$, the midpoint on the circle
is not unique and two symmetric midpoint configurations arise.
We denote the corresponding deterministic maps by
$F_e^{+}$ and $F_e^{-}$. The edge-update kernel $K_e:\Omega_N\times\mathcal B(\Omega_N)\to[0,1]$
is therefore defined by
\[
K_e(\theta,A)=
\begin{cases}
\mathbf 1_A\!\big(F_e(\theta)\big), & |\delta_\theta(e)|\neq \pi,\\[6pt]
\dfrac12\,\mathbf 1_A\!\big(F_e^{+}(\theta)\big)
+\dfrac12\,\mathbf 1_A\!\big(F_e^{-}(\theta)\big), & |\delta_\theta(e)|=\pi,
\end{cases}
\]
for all $\theta\in\Omega_N$ and $A\in\mathcal B(\Omega_N)$.
Finally, since at each step the edge is chosen uniformly at random,
the transition kernel of the Markov process is
\[
K(\theta,A)=\frac1{|\mathcal{E}|}\sum_{e\in \mathcal{E}} K_e(\theta,A),
\qquad
\theta\in\Omega_N,\;A\in\mathcal B(\Omega_N).
\]

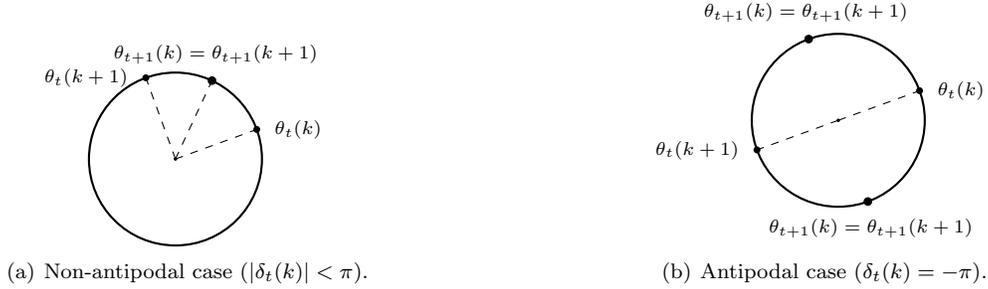
\begin{figure}[htbp]
\centering

\subfigure[Non-antipodal case ($|\delta_t(k)|<\pi$).]{%
\begin{minipage}[t]{0.48\textwidth}
\centering
\begin{tikzpicture}[scale=1.15, every node/.style={font=\scriptsize}]
% Circle
\draw[thick] (0,0) circle (1);
\fill (0,0) circle (0.02);

% angles
\def\a{20}
\def\b{110}
\pgfmathsetmacro{\m}{(\a+\b)/2}

\coordinate (A) at (\a:1);
\coordinate (B) at (\b:1);
\coordinate (M) at (\m:1);

% dashed radii
\draw[dashed] (0,0) -- (A);
\draw[dashed] (0,0) -- (B);
\draw[dashed] (0,0) -- (M);

% points
\fill (A) circle (0.04);
\fill (B) circle (0.04);
\fill (M) circle (0.05);

% labels
\node[right] at ($(A)+(0.10,0)$) {$\theta_t(k)$};
\node[left]  at ($(B)+(-0.10,0)$) {$\theta_t(k+1)$};
\node[above] at (\m:1.10) {$\theta_{t+1}(k)=\theta_{t+1}(k+1)$};

\end{tikzpicture}
\end{minipage}
}\hfill
\subfigure[Antipodal case ($\delta_t(k)=-\pi$).]{%
\begin{minipage}[t]{0.48\textwidth}
\centering
\begin{tikzpicture}[scale=1.15, every node/.style={font=\scriptsize}]
% Circle
\draw[thick] (0,0) circle (1);
\fill (0,0) circle (0.02);

% antipodal angles
\def\a{20}
\pgfmathsetmacro{\b}{\a+180}

% two possible midpoints
\pgfmathsetmacro{\mA}{\a+90}
\pgfmathsetmacro{\mB}{\a-90}

\coordinate (A)  at (\a:1);
\coordinate (B)  at (\b:1);
\coordinate (MA) at (\mA:1);
\coordinate (MB) at (\mB:1);

% radii
\draw[dashed] (0,0) -- (A);
\draw[dashed] (0,0) -- (B);

% points
\fill (A) circle (0.04);
\fill (B) circle (0.04);
\fill (MA) circle (0.05);
\fill (MB) circle (0.05);

% labels
\node[right] at ($(A)+(0.10,0)$) {$\theta_t(k)$};
\node[left]  at ($(B)+(-0.10,0)$) {$\theta_t(k+1)$};

\node[above] at (\mA:1.10) {$\theta_{t+1}(k)=\theta_{t+1}(k+1)$};
\node[below] at (\mB:1.10) {$\theta_{t+1}(k)=\theta_{t+1}(k+1)$};
\end{tikzpicture}
\end{minipage}
}

\caption{Midpoint update under the \texttt{ACCA} dynamics.
(a) If $|\wrap(\theta_t(k+1)-\theta_t(k))|<\pi$, the shortest arc is unique and both angles move to its midpoint.
(b) If $|\wrap(\theta_t(k+1)-\theta_t(k))|=\pi$, the two angles are antipodal and there are two shortest semicircles; the update chooses one midpoint uniformly with probability $1/2$.}
\label{fig:acca}
\end{figure}

 \begin{remark}
\label{rem:midpoint-correct}
Under the update rule \eqref{eq:update2} we have that $\theta_{t+1}(k)=\theta_{t+1}(k+1)$, see Figure~\ref{fig:acca}.
In fact, by definition of the wrap operator, there exists an integer
$m\in \mathbb{Z}$ such that
$\delta_t(k)=\theta_t(k{+}1)-\theta_t(k)-2\pi m$.
Therefore
\[
\theta_t(k)+\tfrac12\,\delta_t(k)
=
\frac{\theta_t(k)+\theta_t(k{+}1)}{2}-\pi m,
\]
\[
\theta_t(k{+}1)-\tfrac12\,\delta_t(k)
=
\frac{\theta_t(k)+\theta_t(k{+}1)}{2}+\pi m.
\]
The two quantities differ by $2\pi m$, and  since $\mathrm{wrap}_\pi(x+2\pi m)=\mathrm{wrap}_\pi(x)$ for every $m\in\mathbb Z$,
we have
$\theta_{t+1}(k)=\theta_{t+1}(k+1)$.
It is important to stress that, in general,
\[
\theta_{t+1}(k)
\neq
\wrap\!\Big(\frac{\theta_t(k)+\theta_t(k{+}1)}{2}\Big),
\]
unless $m=0$ (see Figure~\ref{fig:midpoint-counterexample}), that is, unless the unwrapped difference
$\theta_t(k{+}1)-\theta_t(k)$ already belongs to $[-\pi,\pi)$.
The update therefore corresponds to the midpoint
\emph{along the shortest arc} on $\mathbb S^1$, rather than
the wrapped arithmetic mean.
\end{remark}

\begin{remark}
If $\delta_t(k)\neq -\pi$, the shortest arc on $\mathbb{S}^1$ is unique and so is
its midpoint. When $\delta_t(k)=-\pi$, the two angles are antipodal and
there exist two shortest semicircles of length $\pi$, whose midpoints are
themselves antipodal.

For instance, if $\theta_t(k)=\pi/2$ and $\theta_t(k+1)=-\pi/2$,
then the two shortest arcs have midpoints $0$ and $-\pi$.
Under the convention $\delta_t(k)\in[-\pi,\pi)$ one has $\delta_t(k)=-\pi$,
and the deterministic update yields
$\theta_{t+1}(k)=\theta_{t+1}(k+1)=-\pi$.
We choose  to randomize instead with equal probability between the two midpoints of the shortest arcs,
in order to restore the symmetry of the model.
\end{remark}

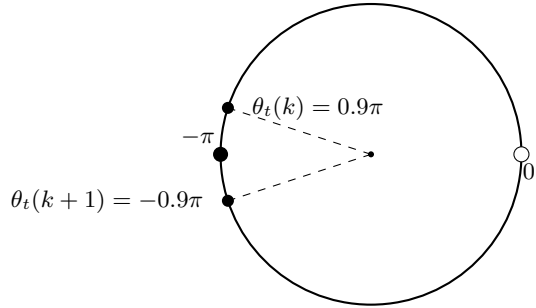
\begin{figure}[H]
\centering
\begin{tikzpicture}[scale=2, every node/.style={font=\small}]

% Circle
\draw[thick] (0,0) circle (1);
\fill (0,0) circle (0.02);

% Counterexample angles
\def\a{162}     % 0.9π ≈ 162°
\def\b{-162}    % -0.9π ≈ -162°

% Wrapped midpoint along shortest arc
\pgfmathsetmacro{\mshort}{180}   % -π and π coincide

% Arithmetic mean
\pgfmathsetmacro{\mmean}{0}

% Coordinates
\coordinate (A) at (\a:1);
\coordinate (B) at (\b:1);
\coordinate (Mshort) at (180:1);
\coordinate (Mmean) at (0:1);

% Radii
\draw[dashed] (0,0) -- (A);
\draw[dashed] (0,0) -- (B);

% Points
\fill (A) circle (0.04);
\fill (B) circle (0.04);

\node[right] at ($(A)+(0.1,0)$) {$\theta_t(k)=0.9\pi$};
\node[left]  at ($(B)+(-0.1,0)$) {$\theta_t(k+1)=-0.9\pi$};

% Correct midpoint
\fill (Mshort) circle (0.05);
\node[above] at (180:1.15) {$-\pi$};

% Wrong arithmetic mean
\fill[white] (Mmean) circle (0.05);
\draw (Mmean) circle (0.05);
\node[below] at (0:1.05) {$0$};

\end{tikzpicture}
\caption{Example illustrating that the midpoint along the shortest arc
does not coincide with the wrapped arithmetic mean.
Here $\theta_t(k)=0.9\pi$ and $\theta_t(k+1)=-0.9\pi$.
The update yields $-\pi$, while the wrapped arithmetic mean equals $0$.}
\label{fig:midpoint-counterexample}
\end{figure}
%
%
%----------------------------------------
% SECTION 3: CONVERGENCE TOWARDS CONSENSUS
%----------------------------------------
%
%
\section{Convergence towards consensus}
\label{sec:consensus}

The \texttt{ACCA} dynamics is locally contractive: each update replaces a
neighboring pair by its short-arc midpoint and therefore reduces
disagreement.
On any finite connected graph this mechanism enforces global alignment,
so consensus is the unique absorbing state.

In this section we establish almost-sure convergence to consensus for the
discrete-time midpoint dynamics on both the path (open boundary
conditions) and the ring (periodic boundary conditions).
These results are discrete-time counterparts of the continuous-time
compass model results of \cite{gantert2020}.
Since each \texttt{ACCA} update applies exactly the same midpoint map
(with $\mu=1/2$), the Lyapunov and flattening arguments of
\cite[Lemma~3.1 and Proposition~3.3]{gantert2020} extend directly.

The only role of the Poisson clocks in the continuous-time model is to
ensure that prescribed finite patterns of edge updates occur with
positive probability.
In the present discrete-time dynamics the selected edges form an
i.i.d.\ sequence with uniform law on $\mathcal E$.
Consequently every finite sequence of edge selections has strictly
positive probability, and the arguments of \cite{gantert2020} carry over
verbatim.

For a configuration $\theta\in\Omega_N$ on a finite graph
$G=(V_N,\mathcal E)$ with wrapped edge increments $\delta_t(e)$,
we say that \emph{weak consensus} holds if $\delta_t(e)\to0$ for all
$e\in\mathcal E$, and that \emph{strong consensus} holds if
$\theta_t(v)\to L$ for all $v\in V_N$ and some random
$L\in\mathbb S^1$.
On finite connected graphs weak consensus implies strong consensus.

\begin{lemma}[Consensus with open boundary conditions]
\label{lem:path-consensus}
For every fixed $N$, the discrete-time \texttt{ACCA} dynamics satisfies
\[
\lim_{t\to\infty}
\sum_{e\in\mathcal E_{\emptyset}}
|\delta_t(e)| = 0
\qquad\mathrm{a.s.}
\]
In particular, strong consensus holds almost surely.
\end{lemma}

\begin{proof}[Proof sketch]
This is the discrete-time analogue of
\cite[Lemma~3.1]{gantert2020}.
Let $V_N$ be the path graph and define
\[
\mathcal W_t :=
\sum_{e\in\mathcal E_{\emptyset}} |\delta_t(e)|.
\]

A midpoint update affects only the selected edge and its neighbors.
Since the maximal degree of the path is $2$, the local computation in
\cite[(10)--(12)]{gantert2020} shows that one update cannot increase
$\mathcal W_t$.
Hence $(\mathcal W_t)_{t\ge0}$ is nonincreasing and therefore converges
almost surely.

To identify the limit, consider the leftmost edge $e_1=(1,2)$.
Whenever $e_1$ is updated,
\[
\delta_{t+1}(e_1)=(1-2\mu)\delta_t(e_1).
\]
Since $\mu=1/2$, this yields $\delta_{t+1}(e_1)=0$.
Because edges are chosen i.i.d.\ and uniformly, $e_1$ is selected
infinitely often almost surely, and therefore
$\delta_t(e_1)\to0$.

Proceed inductively along the path.
Suppose $\delta_t(e_j)\to0$ almost surely for all $j\le r$.
If $\delta_t(e_{r+1})$ did not converge to zero, there would exist
$\varepsilon>0$ and infinitely many times with
$|\delta_t(e_{r+1})|>\varepsilon$.
At any subsequent update of $e_{r+1}$ the neighboring increment
$e_r$ would receive a contribution of order $\mu\varepsilon$,
contradicting the convergence $\delta_t(e_r)\to0$.
Hence $\delta_t(e_{r+1})\to0$ almost surely. By induction,
$\delta_t(e)\to0$ for all $e\in\mathcal E_{\emptyset}$,
and therefore 
$\sum_{e\in\mathcal E_{\emptyset}} |\delta_t(e)| \to 0$
a.s..

On a finite connected graph weak consensus implies strong consensus.
\end{proof}

\begin{proposition}[Consensus with periodic boundary conditions]
\label{prop:ring-consensus}
For every fixed $N$, the discrete-time \texttt{ACCA} dynamics satisfies
\[
\lim_{t\to\infty}
\max_{e\in\mathcal E_p}
|\delta_t(e)| = 0
\qquad\mathrm{a.s.},
\]
and hence strong consensus holds almost surely.
\end{proposition}

\begin{proof}[Proof sketch]
We adapt the argument of
\cite[Proposition~3.3]{gantert2020} to the discrete-time
\texttt{ACCA} dynamics. Let $e_\star=(N,1)$ denote the closing edge and define as before:
\[
\mathcal W_t :=
\sum_{e\in\mathcal E_p} |\delta_t(e)|.
\]
Similarly to the path, a midpoint update cannot increase $\mathcal W_t$,
so $(\mathcal W_t)_{t\ge0}$ is non-increasing and converges almost surely. Fix $\varepsilon>0$.
By Lemma~\ref{lem:path-consensus} there exists a finite sequence of
interior edges
\[
\mathbf w=(e^{(1)},\dots,e^{(m)})
\in\mathcal E_{\emptyset}^m
\]
such that applying these $m$ updates successively on the path produces
a configuration satisfying
\[
\sum_{e\in\mathcal E_{\emptyset}} |\delta(e)| \le \varepsilon,
\]
independently of the initial configuration. Let $A_\ell$ denote the event that during the $\ell$-th block of
$m$ steps the selected edges coincide with the word $\mathbf w$.
Since edges are chosen i.i.d.\ uniformly from $\mathcal E_p$,
\[
\mathbb P(A_\ell)=N^{-m}>0,
\]
and the events $(A_\ell)_{\ell\ge0}$ are independent across blocks. Whenever $A_\ell$ occurs, the closing edge $e_\star$ is not updated
during that block, so the dynamics coincides with the path dynamics.
Thus
\[
\sum_{e\in\mathcal E_{\emptyset}} |\delta_t(e)| \le \varepsilon
\]
at the end of that block.
By the triangle inequality,
\[
|\delta_t(e_\star)|
\le
\sum_{e\in\mathcal E_{\emptyset}} |\delta_t(e)|
\le
\varepsilon,
\]
and therefore
\[
\max_{e\in\mathcal E_p} |\delta_t(e)| \le \varepsilon .
\]
Since the events $(A_\ell)_{\ell}$ occur infinitely often almost surely and
$\mathcal W_t$ is nonincreasing, we obtain
\[
\limsup_{t\to\infty}
\max_{e\in\mathcal E_p} |\delta_t(e)|
\le \varepsilon
\qquad\text{a.s.}
\]
Because $\varepsilon>0$ is arbitrary,
\[
\max_{e\in\mathcal E_p} |\delta_t(e)| \to 0
\qquad\text{a.s.}
\]
which implies strong consensus.
\end{proof}

The results above establish almost-sure convergence to consensus for
every fixed $N$. The simulations below focus on the open-boundary case
and illustrate the relaxation mechanism on the path, highlighting the
strong slowdown of convergence as the system size increases.

Figures~\ref{fig:200_empty}--\ref{fig:sim5_empty_cyl} illustrate indeed the contraction
mechanism and the growth of relaxation times with system size. Under open
boundary conditions, the midpoint dynamics drives the system directly toward
consensus (Lemma~\ref{lem:path-consensus}). However, the relaxation becomes
increasingly slow as $N$ grows, consistent with the diffusive scaling of
repeated local averaging. In particular, for $N=1000$ the configuration
remains far from flat even after $10^{10}$ updates (Figure~\ref{fig:sim5_empty}),
with long-lived  \emph{ladder} structures that decay only on extremely long
time scales.

Because angular coordinates $\theta_t(i)\in[-\pi,\pi)$ are discontinuous at
$\pm\pi$, nearby points on $\mathbb{S}^1$ may appear far apart when they lie
on opposite sides of the cut. To avoid this artifact we also represent the
configuration on a cylinder, where the phase is plotted as a continuous
angular coordinate and the index $i$ forms the axial direction. In this
representation consensus appears as a vertical line (see panel (d) of Figure~\ref{fig:sim5_empty_cyl}), while winding manifests
as a helical structure.

The slow contraction observed on the open path provides a baseline time
scale for the dynamics. In the periodic case studied in Section~\ref{sec:meta}, the presence of
a cycle fundamentally modifies the transient behavior: once branch--crossings
cease, the winding number freezes and the system evolves near a twisted
configuration before eventual convergence to consensus.

% N=200

\begin{figure}[H]
    \centering
    \scalebox{0.75}{%
        \begin{minipage}{\textwidth}
            \centering
            
            \subfigure[]{\includegraphics[width=0.44\textwidth]{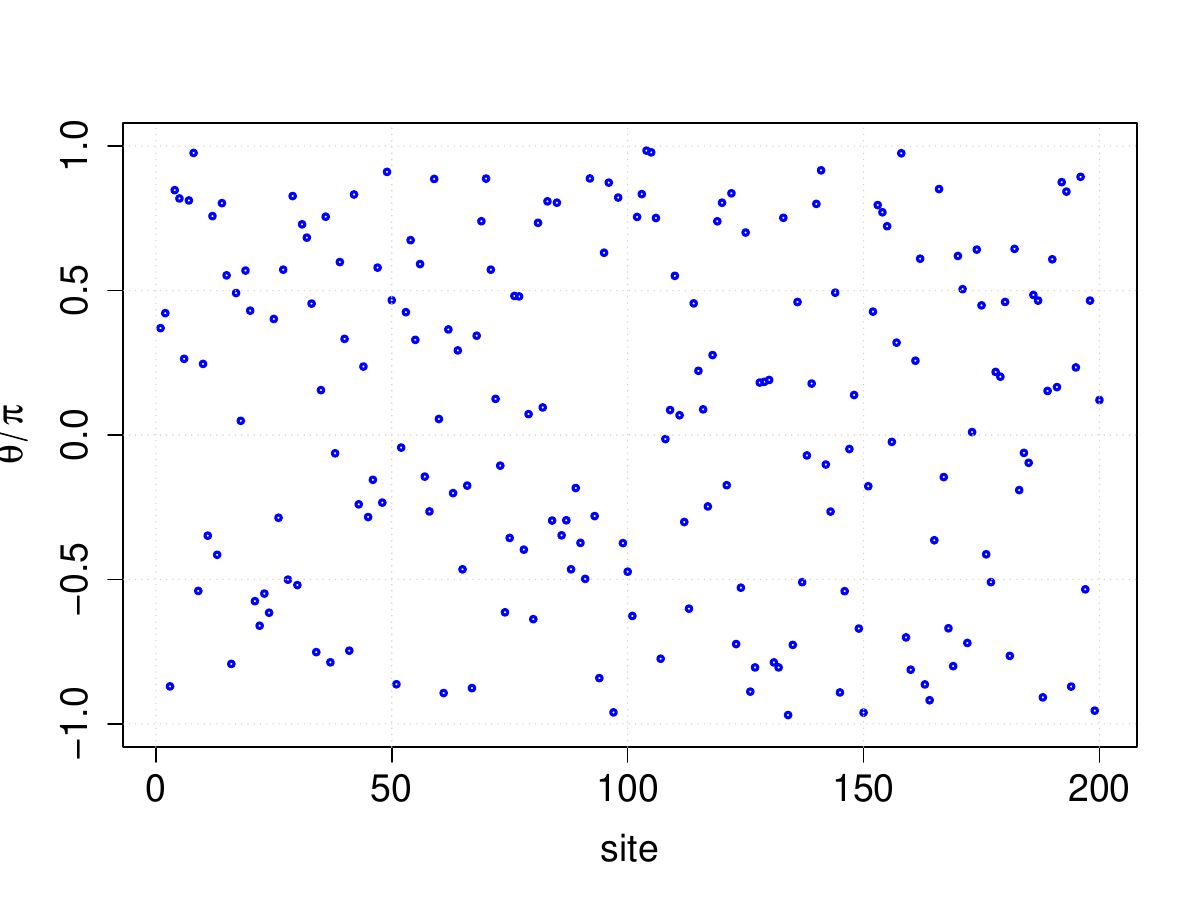}} 
            \subfigure[]{\includegraphics[width=0.44\textwidth]{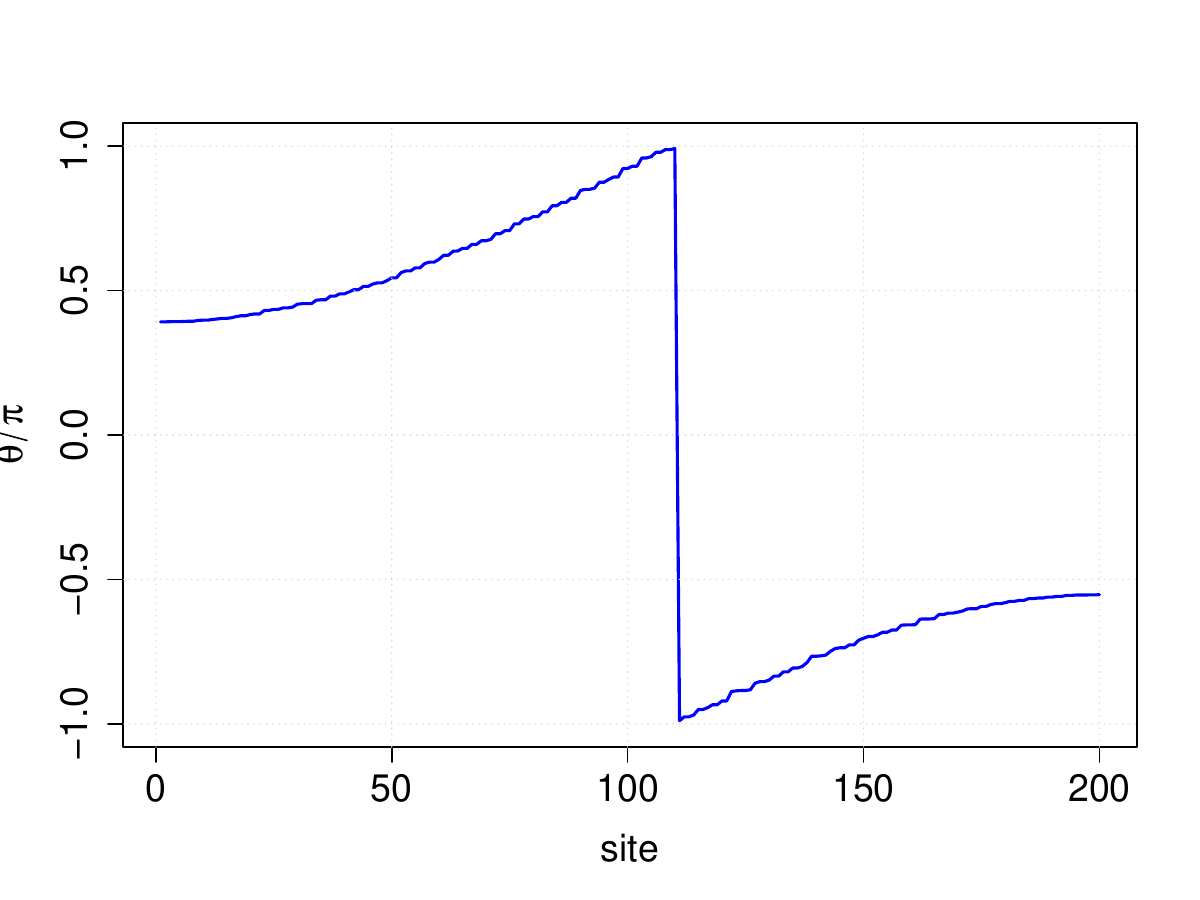}} 
            \subfigure[]{\includegraphics[width=0.44\textwidth]{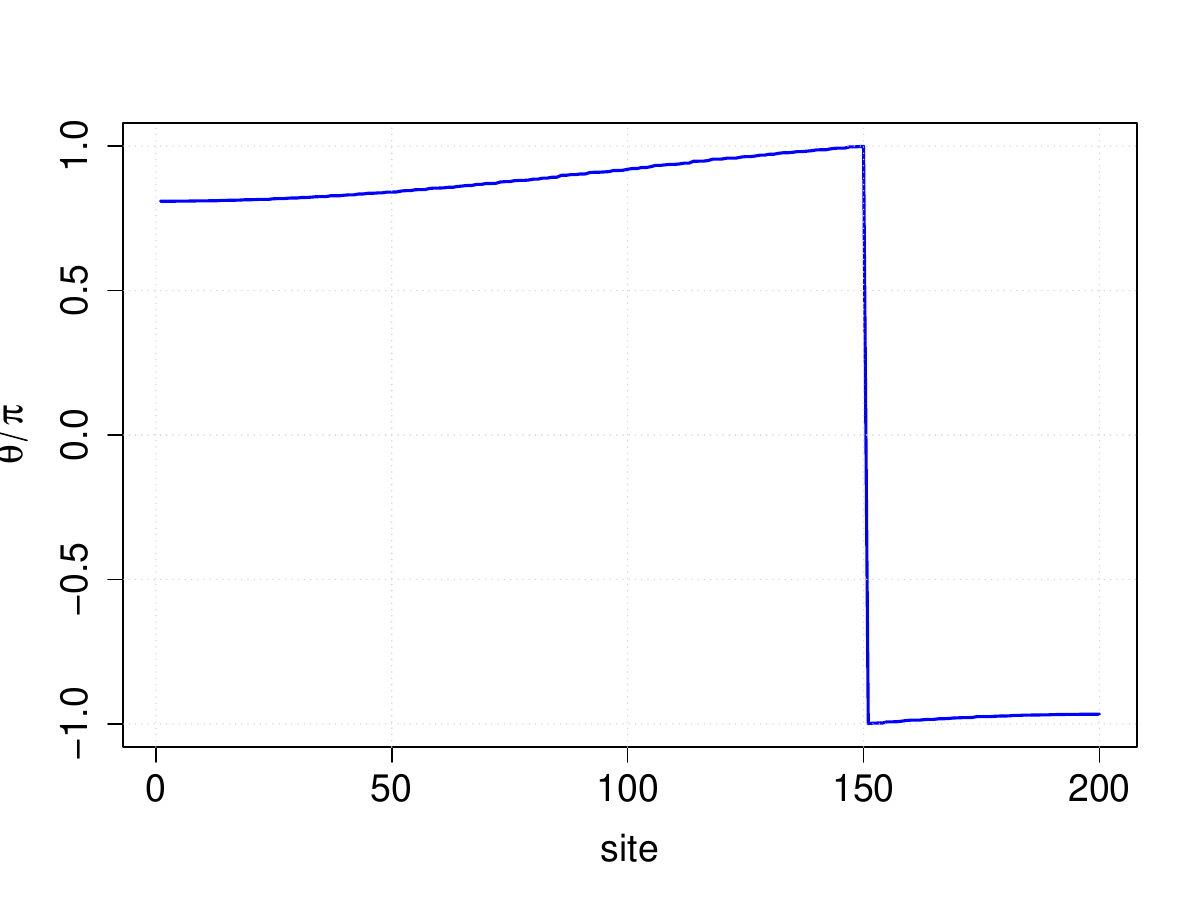}}
            \subfigure[]{\includegraphics[width=0.44\textwidth]{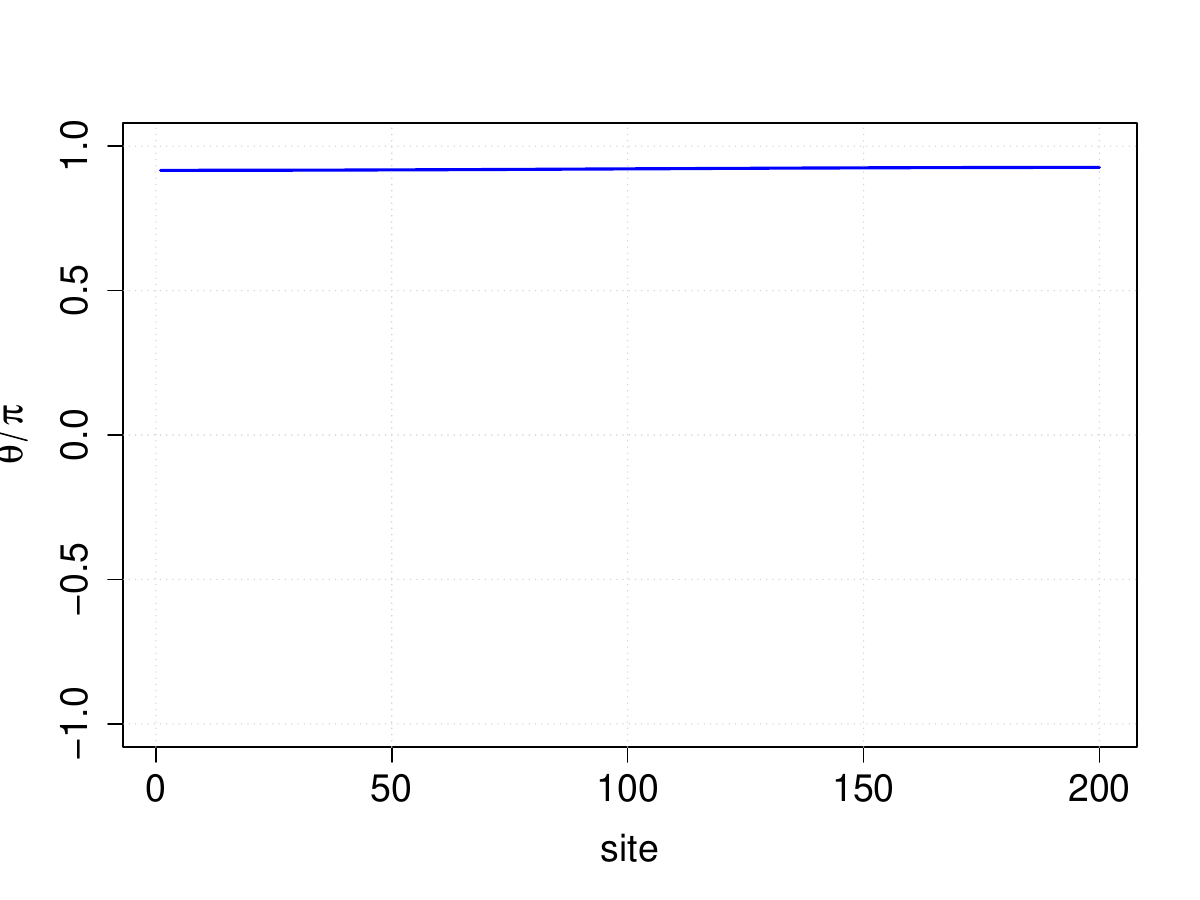}}
            
        \end{minipage}
    }
    \caption{Simulation for a system with $N=200$ with empty boundary conditions (a) is the starting opinion distribution (b) configuration at $t=2.5\cdot 10^7$ with temporary ladder structures (c) at $t=5\cdot 10^7$ only one global ladder is remaining (d) at $t=10^{8}$ the ladder is reduced}
    \label{fig:200_empty}
\end{figure}

\begin{figure}[H]
    \centering
    \scalebox{0.6}{%
        \begin{minipage}{\textwidth}
            \centering
            \subfigure[]{\includegraphics[width=0.49\textwidth]{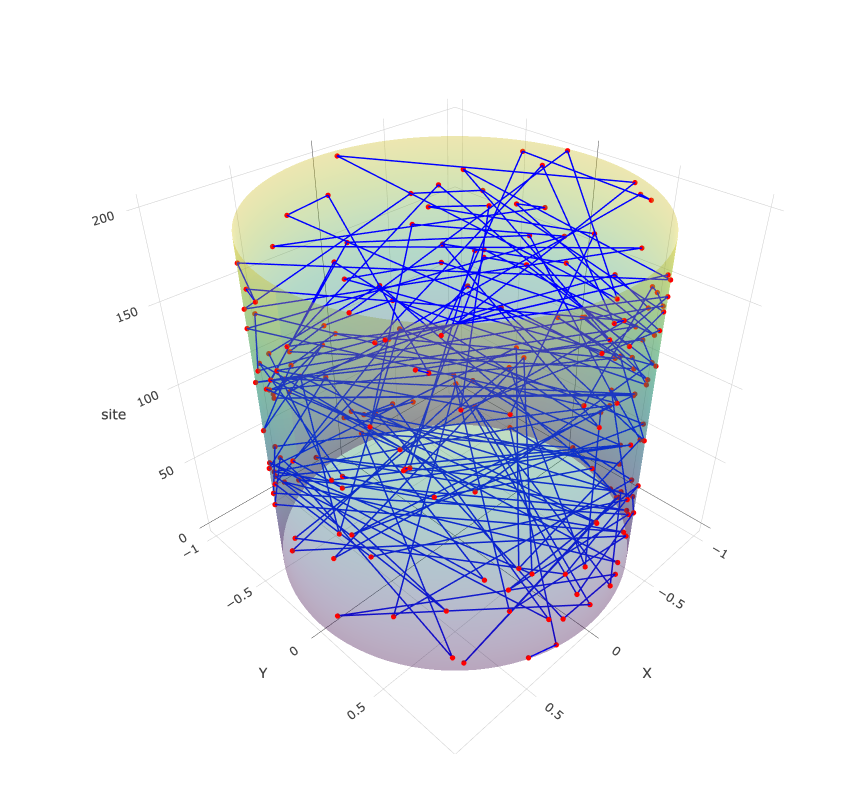}} 
            \subfigure[]{\includegraphics[width=0.49\textwidth]{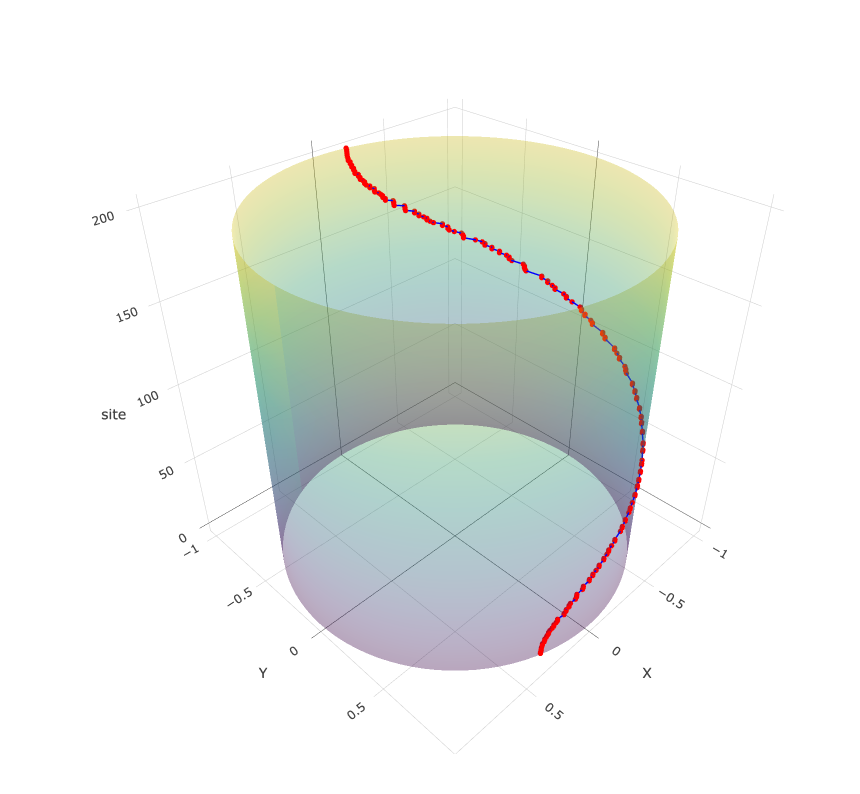}} 
            \subfigure[]{\includegraphics[width=0.49\textwidth]{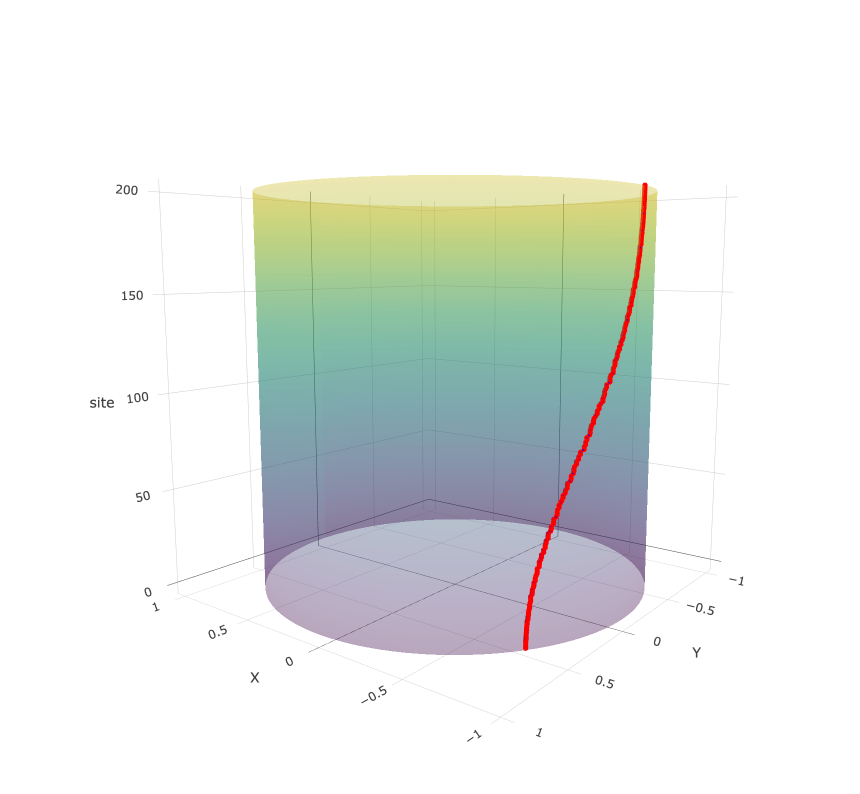}}
            \subfigure[]{\includegraphics[width=0.49\textwidth]{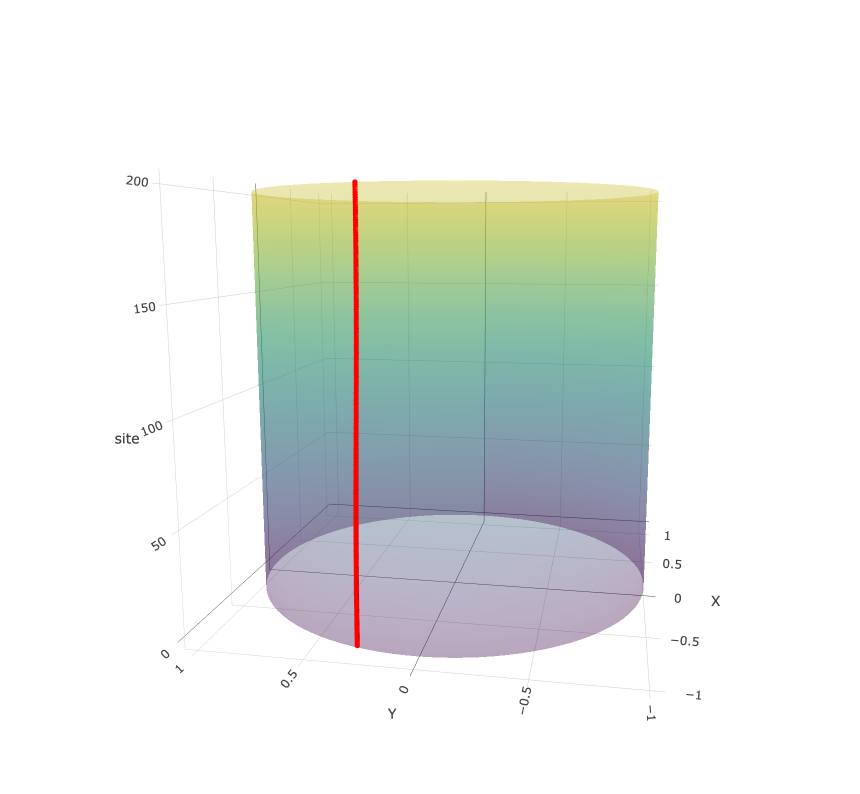}}
        \end{minipage}
    }
    \caption{Simulation for a system with $N=200$ with empty boundary condition as in Figure~\ref{fig:200_empty} in cylindrical coordinates}
    \label{fig:200_empty_cyl}
\end{figure}

% N=1000
\begin{figure}[H]
    \centering
    \scalebox{0.75}{%
        \begin{minipage}{\textwidth}
            \centering
            \subfigure[]{\includegraphics[width=0.44\textwidth]{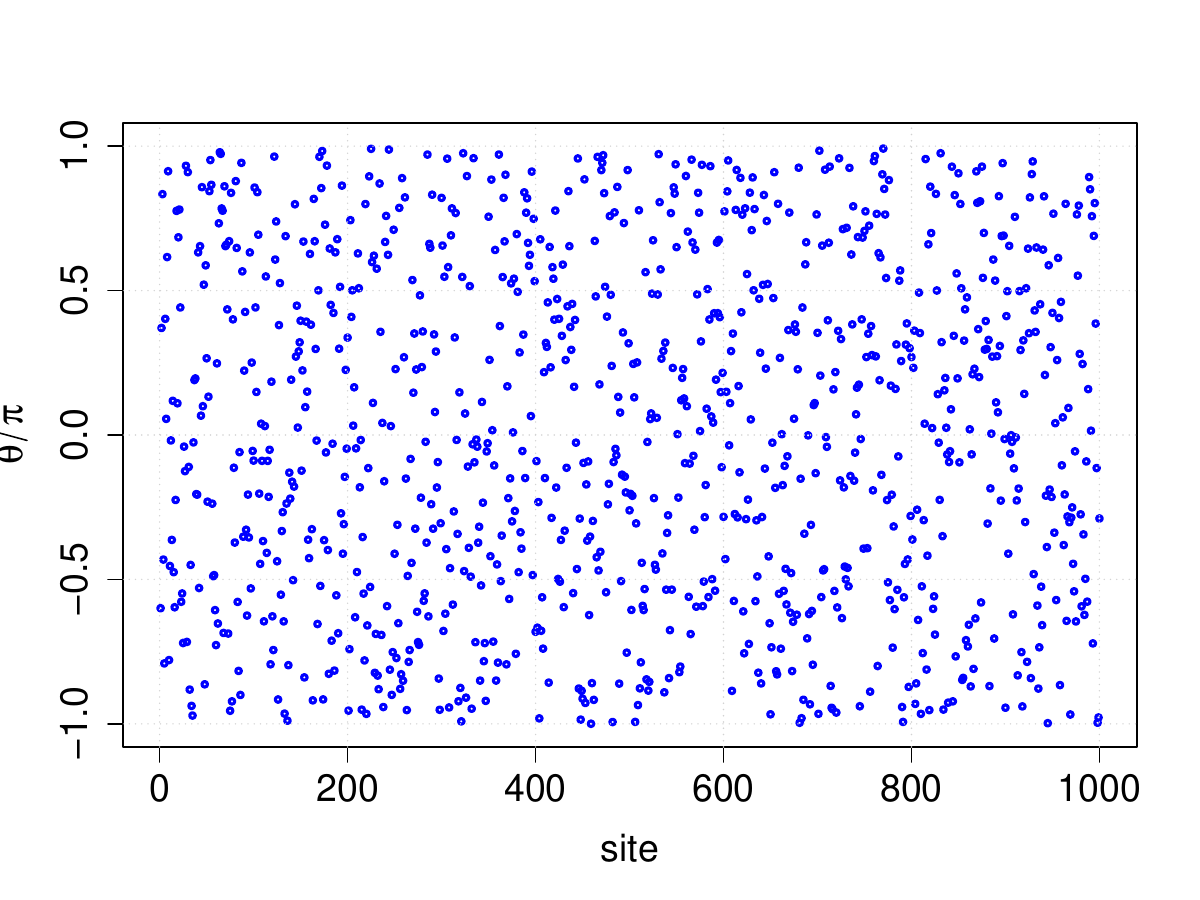}} 
            \subfigure[]{\includegraphics[width=0.44\textwidth]{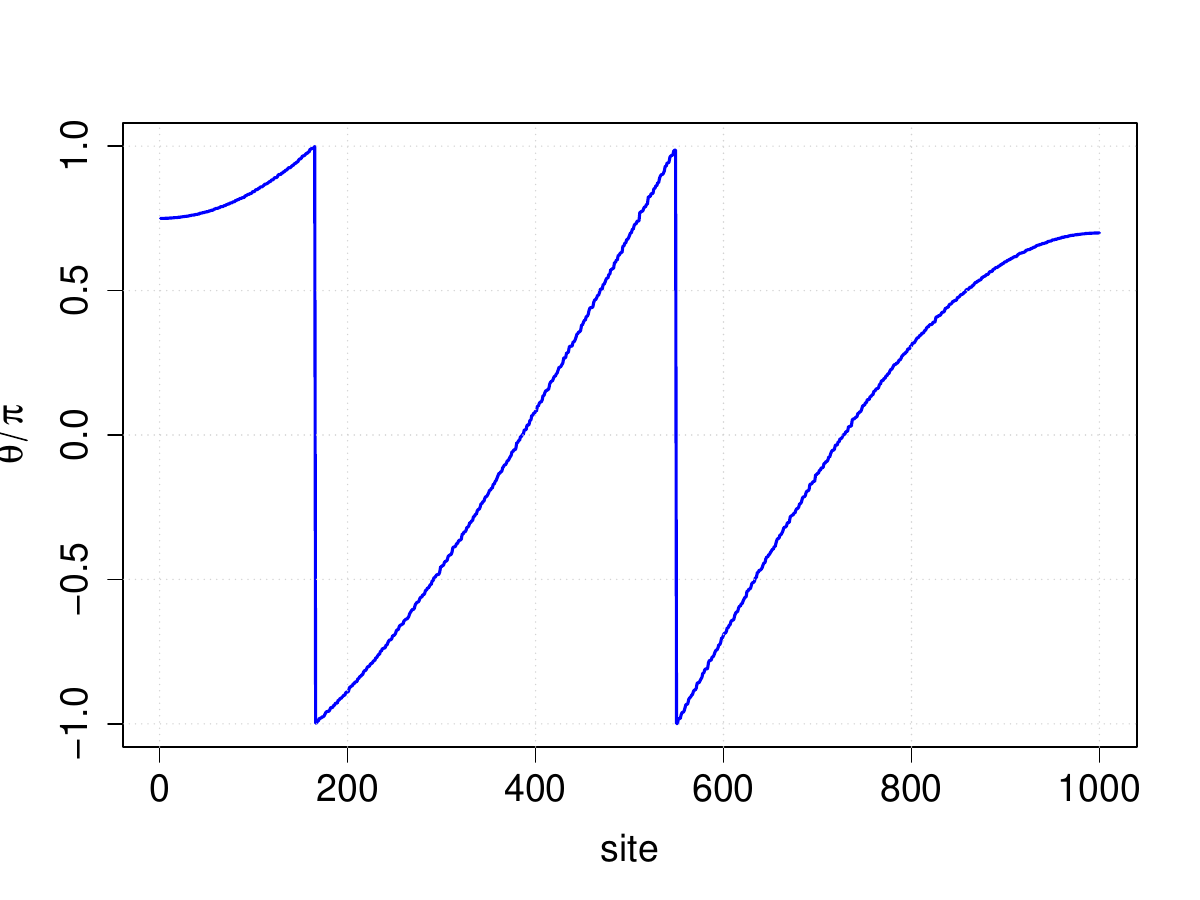}} 
            \subfigure[]{\includegraphics[width=0.44\textwidth]{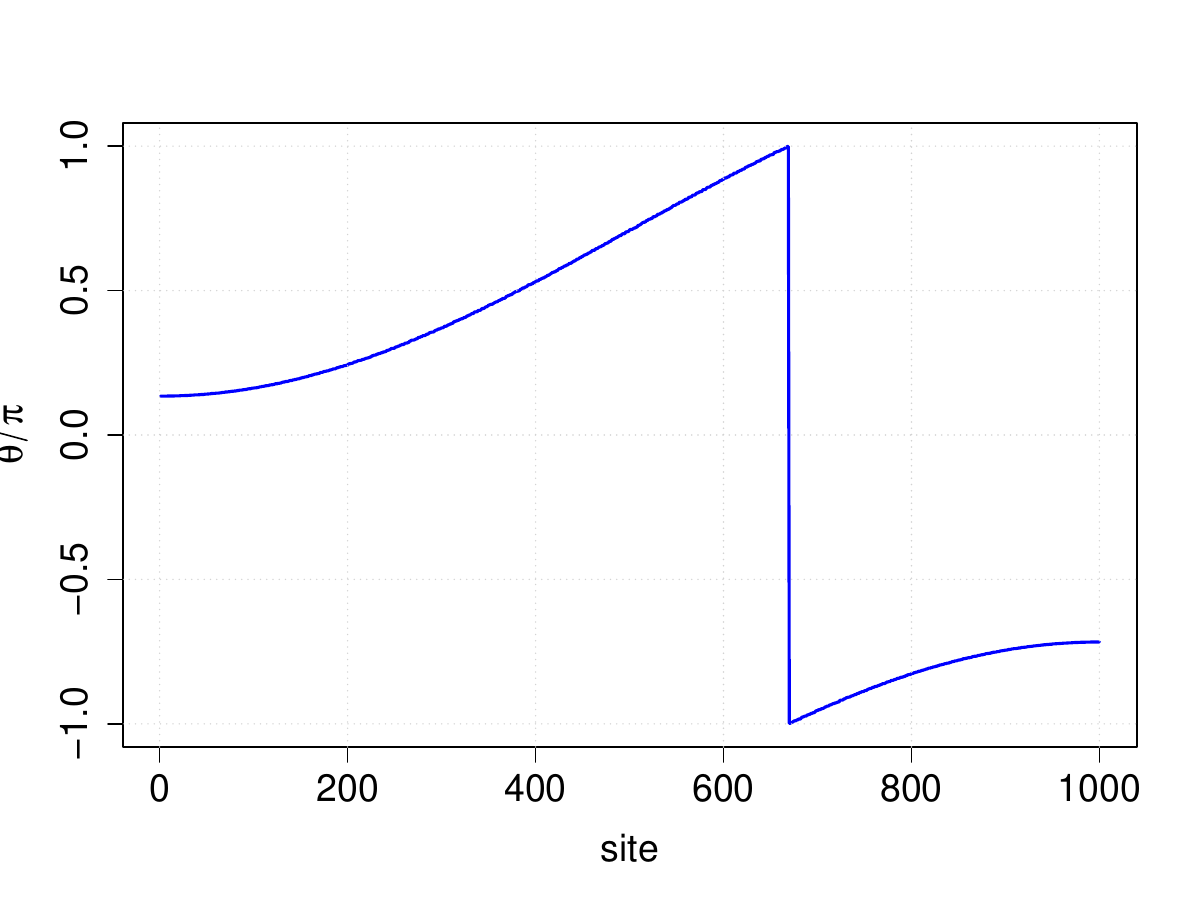}}
            \subfigure[]{\includegraphics[width=0.44\textwidth]{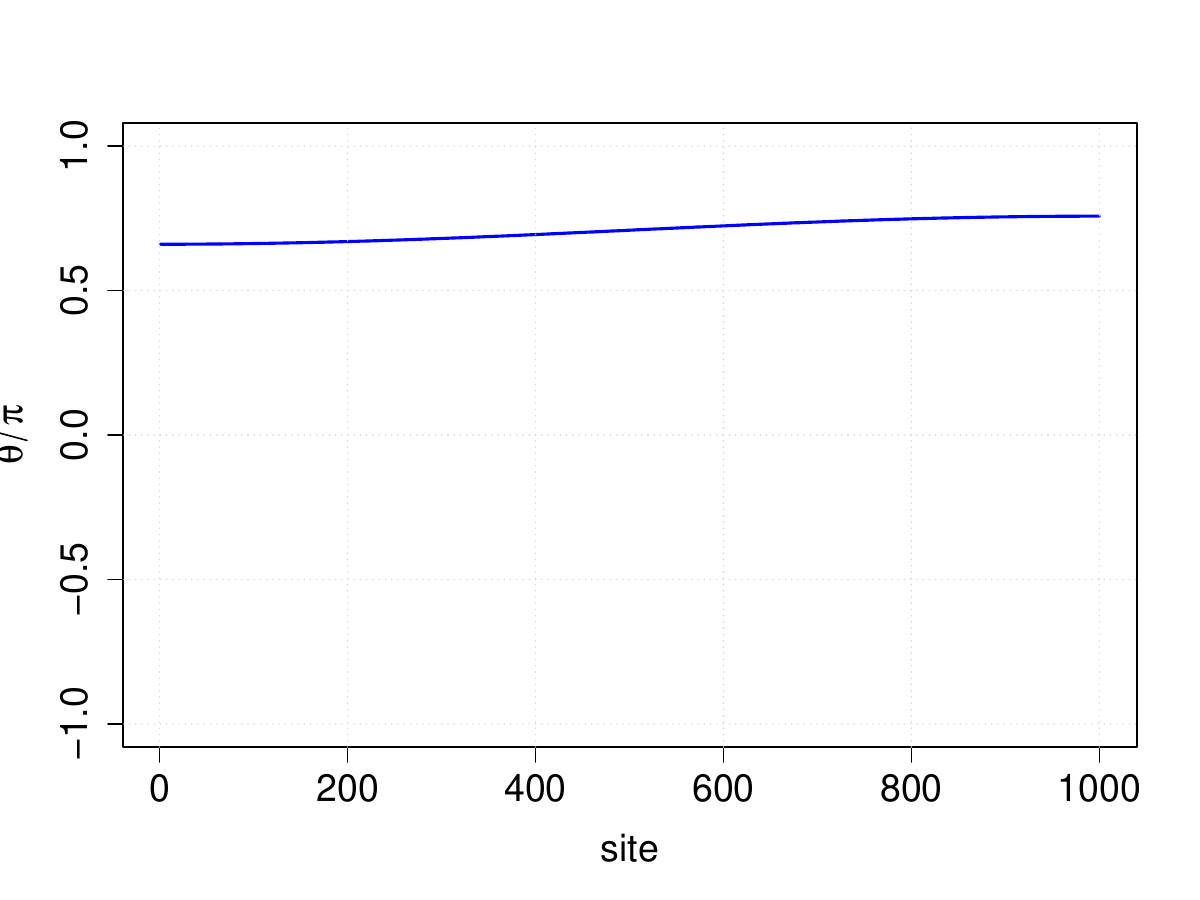}}
        \end{minipage}
    }
    \caption{Simulation for a system with $N=1000$ with empty boundary conditions (a) is the starting opinion distribution (b) configuration at $t=2.5\cdot 10^9$ with temporary ladder structures (c) at $t=5\cdot 10^9$ only one global ladder is remaining (d) at $t=10^{10}$ the ladder is reduced}
    \label{fig:sim5_empty}
\end{figure}
\begin{figure}[H]
    \centering
    \resizebox{0.6\textwidth}{!}{%
        \begin{minipage}{\textwidth}
            \centering
            \subfigure[]{\includegraphics[width=0.49\textwidth]{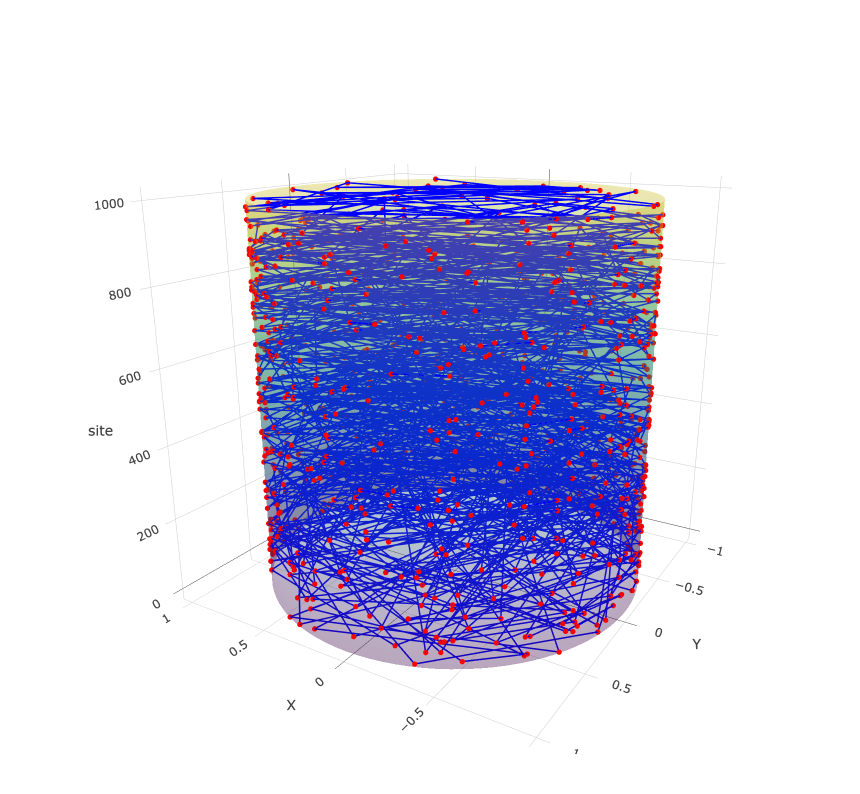}} 
            \subfigure[]{\includegraphics[width=0.49\textwidth]{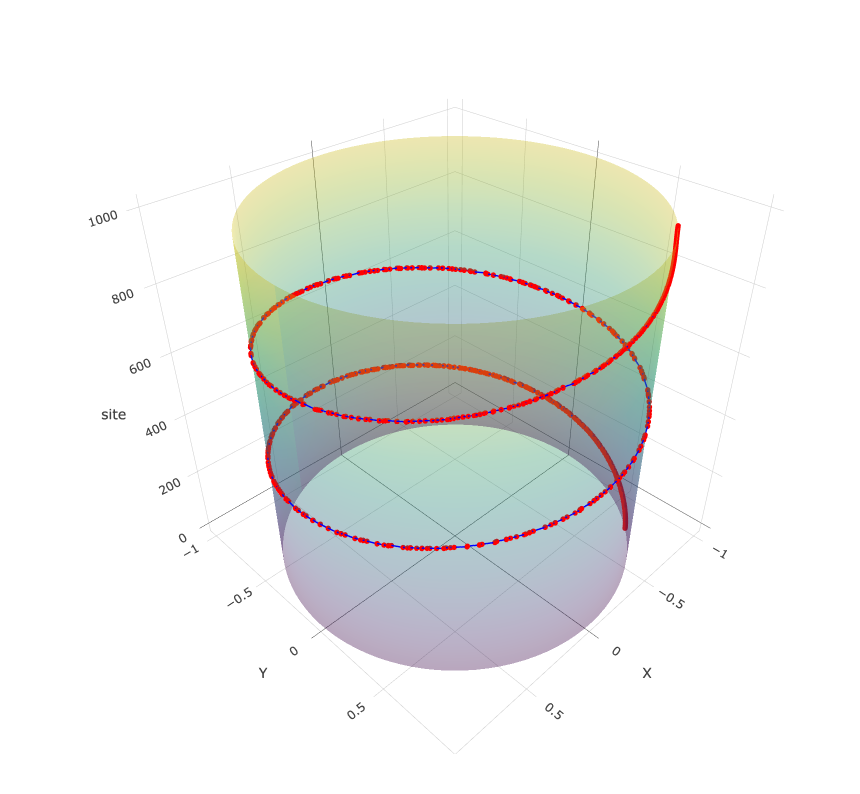}} 
            \subfigure[]{\includegraphics[width=0.49\textwidth]{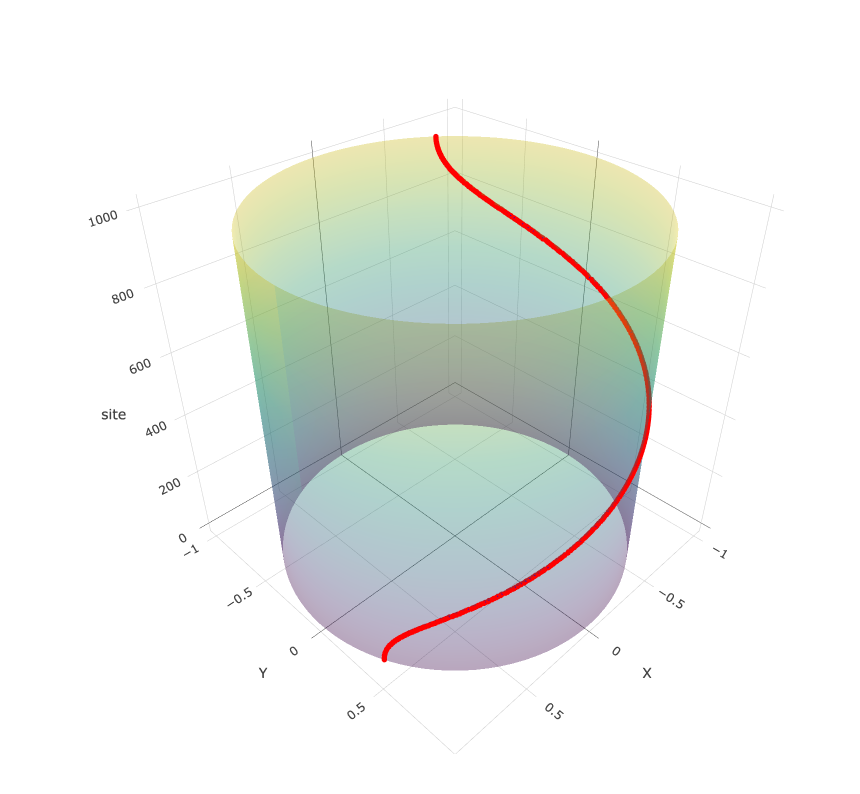}}
            \subfigure[]{\includegraphics[width=0.49\textwidth]{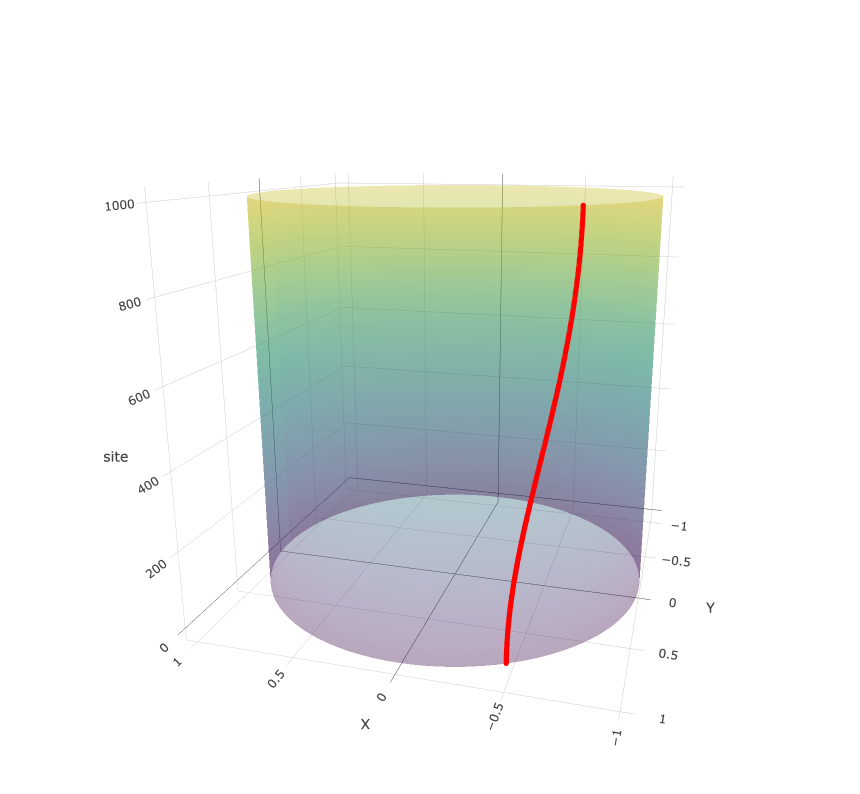}}
        \end{minipage}
    }
    \caption{Simulation for a system with $N=1000$ as in Figure~\ref{fig:sim5_empty} cylindrical coordinates}
    \label{fig:sim5_empty_cyl}
\end{figure}
%
%
%
%---------------------------------------------------
% SECTION 4: TRAPPING AROUND A WINDING CONFIGURATION
%----------------------------------------------------
%
%
%
%
\section{Trapping around a winding configuration}
\label{sec:meta}
Under periodic boundary conditions the circular geometry introduces a
global topological structure. The wrapped edge increments define an
integer winding number $W_t$ (see Equation \eqref{eq:winding}). Although $W_t$ may fluctuate at early
times (see Figure~\ref{fig:W_per}), we show that it can change only through \emph{branch--crossing}
events (Theorem~\ref{thm:W-conditional}). Once such events cease, the winding number freezes and the
dynamics becomes confined to a fixed winding sector. Within a given
sector the configuration relaxes toward a twisted profile determined
by the frozen winding number (Theorem~\ref{thm:adaptive-detrended} and Theorem~\ref{thm:closeness-winding}), producing long-lived transient states
before eventual convergence to consensus via a rare fluctuation, whose mechanism is illustrated in Section~\ref{sec:escape-heur}.
%
% SUBSECTION: Branching-crossing and winding number conservation
%
 \subsection{Branching-crossing and winding number conservation}
\label{sec:meta1}

Under periodic boundary conditions the wrapped edge increments define
a global quantity measuring the total rotation around the ring.  The
\emph{winding number} at time $t$ is
\begin{equation}\label{eq:winding}
W_t := W(\theta_t) = \frac{1}{2\pi}\sum_{i=1}^N \delta_t(i).
\end{equation}

Although $W_t$ is always an integer (see Remark~\ref{rem:wind}), it is not automatically conserved
by the dynamics. The following result shows that the winding number
can change only through \emph{branch--crossing} events, namely updates
in which a neighboring increment leaves the principal interval
$(-\pi,\pi)$. For generic initial data such events may occur during an
initial transient regime, when the increments $\delta_t(i)$ are still
large. However, the \texttt{ACCA} dynamics is locally contractive and
progressively smooths phase differences, making branch--crossings
increasingly unlikely.

\begin{theorem}[Branch-crossing and winding number]
\label{thm:W-conditional}
Let $\theta_t\in[-\pi,\pi)^N$ evolve by the \texttt{ACCA} dynamics with periodic boundary conditions.
Fix $t$ and suppose that edge $(k,k+1)$ is updated. Define the two sums:
\begin{equation}
\label{e:S}
\spm(k):=\delta_t(k\!\pm\!1)+\tfrac12\,\delta_t(k).
\end{equation}
Then
\[
W_{t+1}=W_t-(m_-+m_+),\qquad
\wrap(\spm(k))=\spm(k)-2\pi m_\pm,
\]
with suitable  $m_\pm\in\{-1,0,1\}$. In particular,
\[
W_{t+1}=W_t\quad\Longleftrightarrow\quad \spm(k) \in(-\pi,\pi).
\]
\end{theorem}
Hence, the winding number is conserved for \emph{all} times if and only if every update satisfies
$\spm\in(-\pi,\pi)$ (i.e., no branch-crossing occurs).
\begin{proof}
Let $d:=\delta_t(k)$. 
By Remark~\ref{rem:midpoint-correct}, the two updated angles coincide after wrapping, so that
$\theta_{t+1}(k)=\theta_{t+1}(k+1)$ and $ \delta_{t+1}(k)=\wrap(\theta_{t+1}(k+1)-\theta_{t+1}(k))=\wrap(0)=0$.

If $i\notin\{k-1,k,k+1\}$, then neither endpoint of the edge $(i,i+1)$ is updated, so that
$\delta_{t+1}(i)=\delta_t(i)$. It remains to compute the two neighbouring differences
$\delta_{t+1}(k-1)$ and $\delta_{t+1}(k+1)$.
For the left neighbour edge,
\[
\delta_{t+1}(k-1)=\wrap\!\big(\theta_{t+1}(k)-\theta_{t+1}(k-1)\big)
=\wrap\!\big(\theta_{t+1}(k)-\theta_t(k-1)\big).
\]
Using $\theta_{t+1}(k)=\wrap(\theta_t(k)+\tfrac12 d)$ and the identity \eqref{eq:wrap1},
we obtain:
\begin{equation}
\label{eq:deltat}
\delta_{t+1}(k-1)
=\wrap\!\Big(\theta_t(k)+\tfrac12 d-\theta_t(k-1)\Big).
\end{equation}
Since $\delta_t(k-1)=\wrap(\theta_t(k)-\theta_t(k-1))$ and by identity \eqref{eq:wrap1}, the relation \eqref{eq:deltat} becomes:
\[
\delta_{t+1}(k-1)=\wrap\!\big(\delta_t(k-1)+\tfrac12\,\delta_t(k)\big)=\wrap(\sm(k)).
\]
An analogous computation for $\delta_{t+1}(k+1)$ gives:
\begin{eqnarray*}
\delta_{t+1}(k+1)
&=&\wrap\!\big(\theta_t(k+2)-\theta_{t+1}(k+1)\big)
=\wrap\!\Big(\theta_t(k+2)-\theta_t(k+1)+\tfrac12 d\Big)\\
&=&\wrap\!\big(\delta_t(k+1)+\tfrac12\,\delta_t(k)\big)=\wrap(\spp(k)).
\end{eqnarray*}
Thus, we get:
\[
\delta_{t+1}(k)=0,\qquad
\delta_{t+1}(k\pm1)=\wrap(\spm(k)),\qquad
\delta_{t+1}(i)=\delta_t(i), \forall i\notin\{k-1,k,k+1\}.
\]
Summing over $i\in\{1,\dots,N\}$ the updated edge increments, we get: 
\[
\sum_{i=1}^N \delta_{t+1}(i)
=
\sum_{i=1}^N \delta_t(i)
-\delta_t(k-1)-\delta_t(k)-\delta_t(k+1)
+\wrap(\spm(k))+0+\wrap(\spp(k)).
\]
so that, by the definition \eqref{e:S}:
\begin{equation}
\label{e:deltaplus}
\sum_{i=1}^N \delta_{t+1}(i)
=
\sum_{i=1}^N \delta_t(i)
+\big(\wrap(\sm(k))-\sm(k)\big)+\big(\wrap(\spp(k))-\spp(k)\big).
\end{equation}
By definition, $\wrap(\spm(k))=\spm(k)-2\pi m_\pm$, for  $m_\pm\in\mathbb{Z}$.
Moreover, because each $\delta_t(\cdot)\in[-\pi,\pi)$, we have
$\spm(k)\in[-3\pi/2,3\pi/2)$, so bringing $\spm(k)$ back into $[-\pi,\pi)$ requires at most one shift by $\pm2\pi$; hence necessarily $m_\pm\in\{-1,0,1\}$.
Substituting into \eqref{e:deltaplus}, we have: 
\[
\sum_{i=1}^N \delta_{t+1}(i)
=
\sum_{i=1}^N \delta_t(i)-2\pi(m_-+m_+),
\]
and dividing by $2\pi$ gives $W_{t+1}=W_t-(m_-+m_+)$.
Finally, $W_{t+1}=W_t$ holds if and only if $m_-=m_+=0$, which is equivalent to $\wrap(\spm(k))=\spm(k)$.
Since $\wrap(x)=x$ exactly when $x\in[-\pi,\pi)$ (and the antipodal values $\pm\pi$ correspond to the branch-crossing ambiguity),
this is equivalent to $\spm(k)\in(-\pi,\pi)$, concluding the proof.
\end{proof}
The preceding theorem identifies branch-crossings as the only source
of variation of the winding number.
Whenever no update produces $\spm\notin(-\pi,\pi)$, the integers $m_\pm$
vanish and the topological charge $W_t$ is invariant.
Since the analysis below concerns time intervals during which no branch-crossing occurs, we restrict attention to trajectories satisfying this condition.
We formalize this in the following assumptions.

We introduce stopping times describing the first time at which the
dynamics leaves the regular regime. Let
$\mathcal F_t:=\sigma(\theta_0,e_{k_0},\dots,e_{k_{t-1})}$, where $e_j\in \mathcal{E}$, be the natural filtration generated by the initial configuration and the
sequence of updated edges.
Define the $(\mathcal F_t)_t$--stopping times
\begin{equation}
\label{eq:tau-branch}
\tau_{\mathrm{br}}
:=
\inf\Big\{t\ge0:\exists\,k\in V_N \text{ such that }
S^{(t)}_-(k)\notin(-\pi,\pi)\ \text{or}\ S^{(t)}_+(k)\notin(-\pi,\pi)\Big\},
\end{equation}
and
\begin{equation}
\label{eq:tau-ant}
\tau_{\mathrm{ant}}
:=
\inf\Big\{t\ge0:\exists\,k\in V_N \text{ such that }
\delta_t(k)=-\pi\Big\}.
\end{equation}

The stopping time $\tau_{\mathrm{br}}$ represents the first instant at
which a \emph{branch--crossing} becomes possible, while
$\tau_{\mathrm{ant}}$ is the first time at which an antipodal edge
appears. We then define the stopping time
\begin{equation}
\label{eq:tau}
\tau := \tau_{\mathrm{br}}\wedge\tau_{\mathrm{ant}} .
\end{equation}

Before time $\tau$ all local increments remain in the principal interval
and every midpoint update is uniquely defined. In particular one has
\[
S^{(s)}_\pm(k)\in(-\pi,\pi)
\quad\text{and}\quad
\delta_s(k)\neq -\pi
\qquad
\text{for all } s<\tau \text{ and all edges } (k,k+1).
\]
It is therefore convenient to formulate several local identities and lift
properties up to the stopping time $\tau$.

\begin{assumption}[Analysis before the first singular event]
\label{asmp:no-antipodes}
Let $(\theta_t)_{t\ge0}$ be the \texttt{ACCA} process and let $\tau$ be the
stopping time defined in \eqref{eq:tau}.
We assume that statements involving the configuration at time $t$
are understood to hold on the event
\[
\{\tau \ge t\}.
\]
Equivalently, we restrict attention to trajectories for which no
branch--crossing or antipodal update has occurred before time $t$.
\end{assumption}

\begin{assumption}[No branch--crossing regime]
\label{asmp:tau-infty}
In addition to Assumption~\ref{asmp:no-antipodes}, we further restrict attention to trajectories of
$(\theta_t)_t$ for which
\[
\{\tau=\infty\}.
\]
Equivalently,
\[
S^{(t)}_\pm(k)\in(-\pi,\pi)
\quad\text{and}\quad
\delta_t(k)\neq -\pi
\qquad
\text{for all } t\ge0 \text{ and all edges } (k,k+1).
\]
\end{assumption}
\begin{remark}[A sufficient condition preventing branch--crossings]
\label{rem:corridor}
A simple deterministic condition ensures that no branch--crossing can occur at time $t$.
Suppose that at time $t$ we have:
\begin{equation}
\max_{1\le i\le N} |\delta_t(i)| < \frac{2\pi}{3}.
\end{equation}
By the triangle inequality, for every edge $(k,k+1)$ one has:
\[
|S_\pm^{(t)}(k)|
\le |\delta_t(k\pm1)|+\tfrac12|\delta_t(k)|
\le \max_i|\delta_t(i)|+\tfrac12\max_i|\delta_t(i)|
=\tfrac32\max_i|\delta_t(i)|
<\pi.
\]
Hence $S_\pm^{(t)}(k)\in(-\pi,\pi)$ for every $k$, so no update at time $t$
can produce a branch--crossing.
\end{remark}

\begin{figure}[H]
    \centering
\includegraphics[height=7cm]{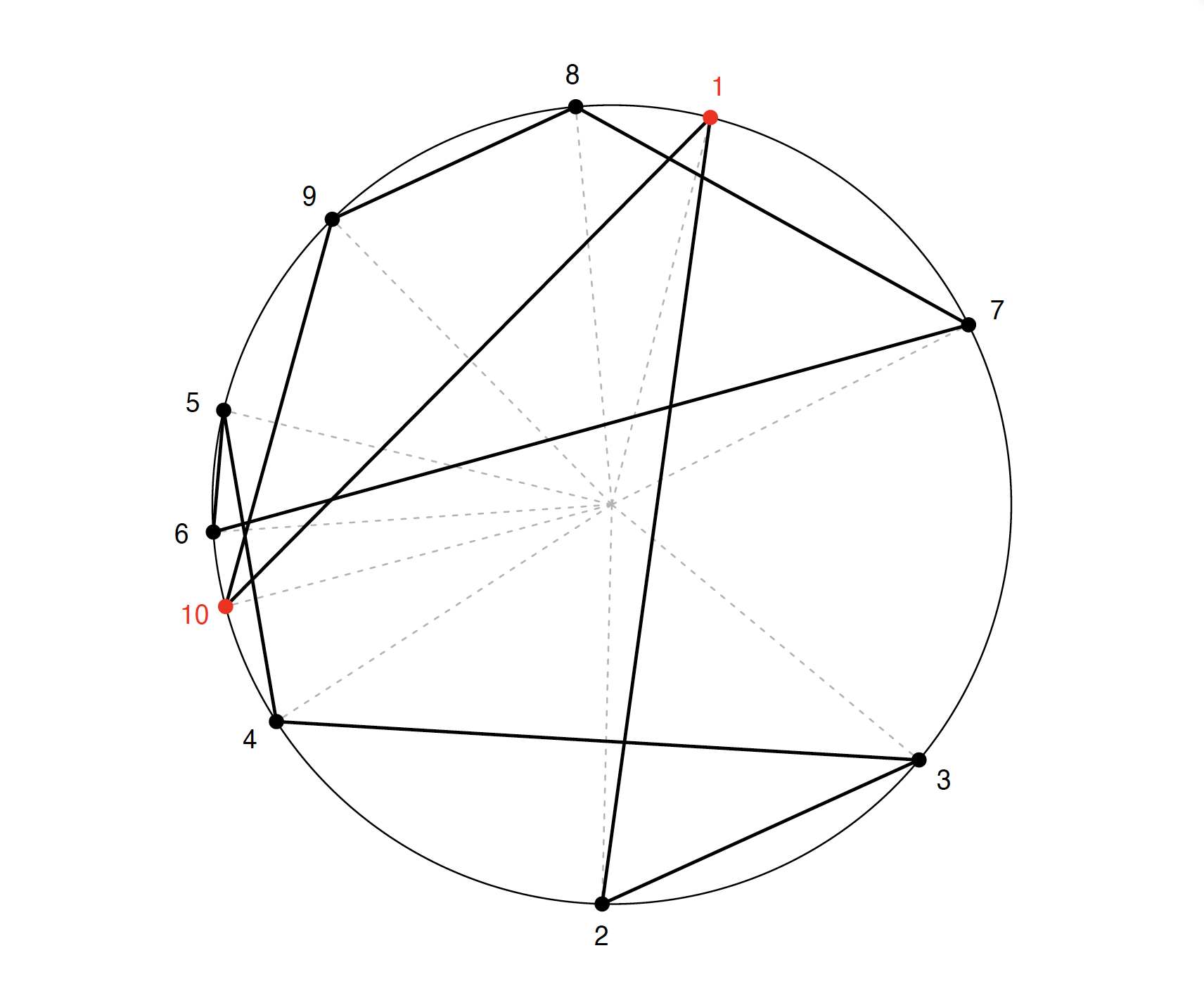}
    \caption{Example of a random configuration of $N=10$ angles on the circle, illustrating explicitly the mechanism of branch-crossing.
Each point represents an angle $\theta(i)\in[-\pi,\pi)$, and dashed
rays indicate the corresponding central angles. Only nearest–neighbour
edges $(k,k+1)$ of the ring graph are drawn.
A vertex is colored in red if at least one of the quantities
$S_-(k)$ or
$S_+(k)$ lies outside the interval
$(-\pi,\pi)$.
For the vertex $1$, we have $S_-(1)=-3.54; S_+(1)=0.58$, while for the vertex $10$ we have $S_-(10)=0.02; S_+(10)=3.96$. Hence the edge $(10,1)$ does not fulfill Assumption~\ref{asmp:no-antipodes}.
}
    \label{fig:crossing}
\end{figure}

Assumption~\ref{asmp:no-antipodes} requires that no update produces a
branch-crossing. A priori, this condition is not automatic: for generic
initial data it may fail already at the very first step. In particular,
if the initial configuration consists of independent uniformly
distributed angles on $[-\pi,\pi)$, then the wrapped differences
$\delta_0(i)$ are typically of order one, and a midpoint update may
push a neighboring difference beyond the interval $(-\pi,\pi)$.
The following lemma quantifies this phenomenon by computing explicitly
the probability that a branch-crossing occurs at the first update.
It shows that Assumption~\ref{asmp:no-antipodes} excludes an event of
strictly positive probability.
\begin{lemma}
\label{lem:first-crossing-probability}
Let $\theta_0\in\Omega_N$ be such that  $\theta_0(i)\overset{i.i.d.}{\sim}\mathrm{Unif}([-\pi,\pi))$, for $i\in V_N$, and fix an edge $(k,k+1)\in \mathcal{E}_p$.  
Then the probability that a midpoint update of $\theta_0$ at edge $(k,k+1)$  produces a \emph{branch-crossing} at time $1$ (i.e., $S_{+}^{(0)}(k)\notin (-\pi,\pi)$ or $S_{-}^{(0)}(k)\notin (-\pi,\pi)$) equals $\frac{11}{48}$.
\end{lemma}
The proof of Lemma~\ref{lem:first-crossing-probability} is given in Appendix~\ref{a:first}.
Figure~\ref{fig:histo-assumption} shows a Monte Carlo verification of
Lemma~\ref{lem:first-crossing-probability}, illustrating that the empirical
frequency of configurations satisfying Assumption~\ref{asmp:no-antipodes} concentrates
around the theoretical value $37/48$.
\begin{figure}[H]
    \centering
    \includegraphics[width=0.7\textwidth]{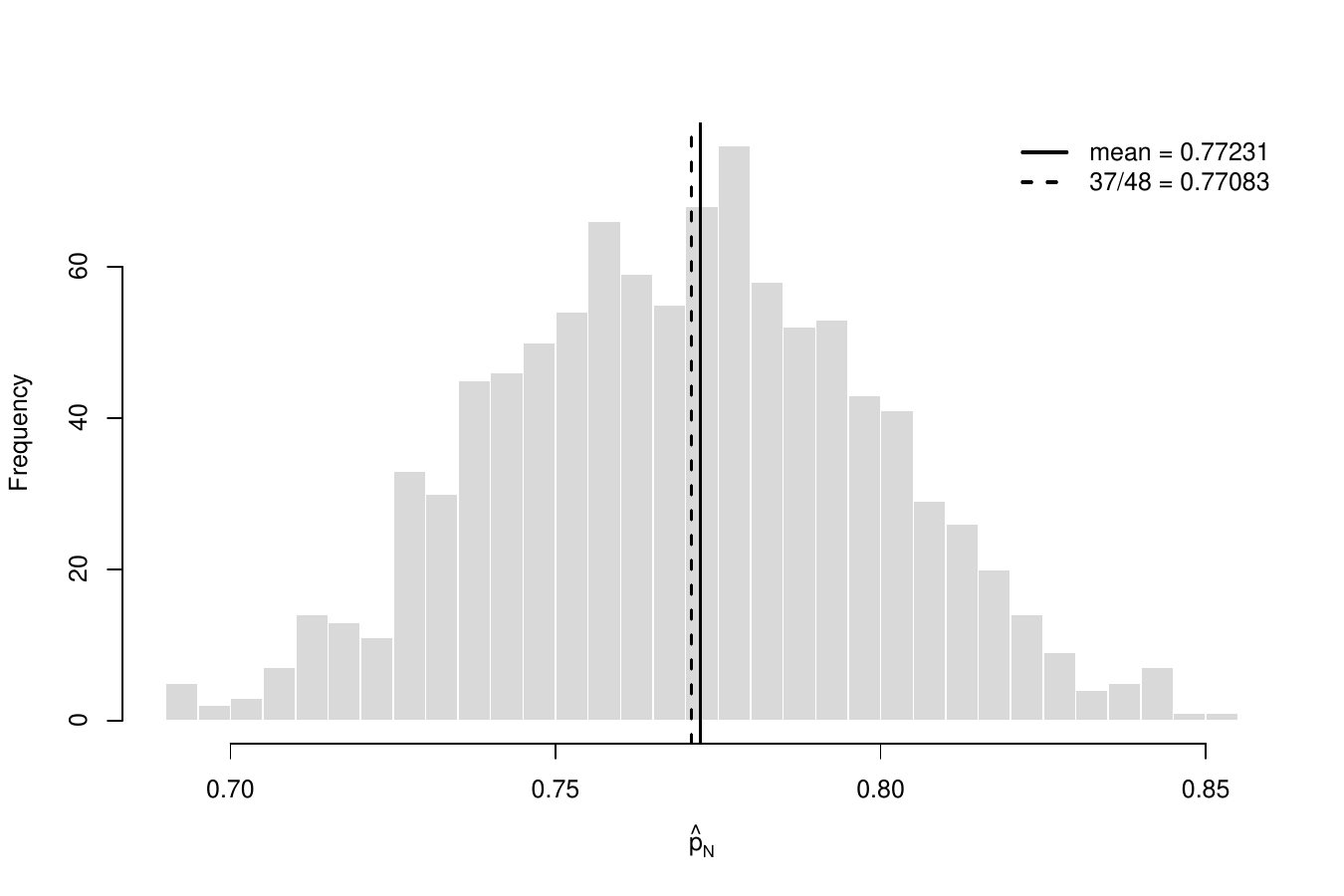}
    \caption{
    Simulation verification of Assumption~\ref{asmp:no-antipodes}.
    For each repetition we generate $N=4000$ i.i.d.\ initial angles 
    $\theta_0(i)\sim \mathrm{Unif}([-\pi,\pi))$ on the ring.
    From this configuration we sample $m=200$ edges uniformly at random
    and check, for each sampled edge $k$, whether Assumption~\ref{asmp:no-antipodes}
     is satisfied.
    We report the empirical fraction $\hat{P}_N$ of edges fulfilling the assumption.
    The histogram shows the distribution of $\hat{P}_N$
    over $R=1000$ independent repetitions.
    The solid vertical line marks the sample mean,
    while the dashed line indicates the theoretical value
    $37/48$ obtained in Lemma~\ref{lem:first-crossing-probability}.
    The numerical agreement confirms the analytical computation.
    }
    \label{fig:histo-assumption}
\end{figure}
%------------------------------------------------------------
%SUBSECTION: lifted representation of circular configurations
%------------------------------------------------------------
\subsection{Lifted representation of configurations on $(V_N,\mathcal{E}_p)$}
\label{sec:lift}

The dynamics studied in this paper take place on the circle
$\mathbb S^1$ which is isomorphic to $[-\pi,\pi)$, where angles are represented in the interval
$[-\pi,\pi)$ and local updates are defined through wrapped
differences between neighboring opinions.
While this representation is natural for the original process,
the presence of the wrapping operator complicates the analysis,
since midpoint updates are not linear in the angular coordinates.

To overcome this difficulty, we introduce a lifted representation
of configurations on the universal covering space $\R$ of the
circle.
In this representation, circular increments become ordinary
differences in $\R$, and midpoint updates correspond to
arithmetic midpoints of lifted endpoints (See Lemma~\ref{lem:lift-midpoint}).
This linear structure will play a key role in the analysis of the
dynamics and in the construction of the detrended variables used
in Section~\ref{sec:meta0}.

The lifting procedure associates to each circular configuration
a real-valued profile whose increments coincide with the wrapped
angular increments.

\begin{definition}[Lift of a configuration]
\label{def:lift}
Let $\theta\in\Omega_N$ be a configuration defined on $(V_N, \mathcal{E}_p)$ with $(\delta(i))_i$ its wrapped
increments defined in \eqref{e:deltae}.
A \emph{lift} of $\theta$ is any vector
\[
\eta=(\eta(1),\dots,\eta(N+1))\in\R^{N+1}
\]
such that
\begin{equation}
\label{eq:lift-inc}
\eta(i+1)-\eta(i)=\delta(i),
\qquad i=1,\dots,N,
\end{equation}
and
\begin{equation}
\label{eq:lift-anchor}
\wrap(\eta(1))=\theta(1).
\end{equation}
\end{definition}

\begin{remark}[Closing relation]
\label{rem:closing}
Summing \eqref{eq:lift-inc} yields
\begin{equation}
\label{eq:closing}
\eta(N+1)=\eta(1)+\sum_{i=1}^N\delta(i)
        =\eta(1)+2\pi W(\theta),
\end{equation}
with $W(\theta)$ defined in \eqref{e:w}.
Thus the lifted path accumulates exactly the winding number
of the configuration.
\end{remark}

\begin{remark}[Geometric interpretation]
Components $1$ and $N+1$ represent the same physical agent after
one complete traversal of the ring in the covering space.
Indeed
\[
\wrap(\eta(N+1))=\wrap(\eta(1)+2\pi W(\theta))=\theta(1).
\]
Hence, the closing edge $(N,1)$ of the ring corresponds in the
lifted space to the segment joining $\eta(N)$ and $\eta(N+1)$.
In particular,
\[
\delta(N)=\eta(N+1)-\eta(N).
\]
\end{remark}

Let $(\theta_t)_{t\ge0}$ be a trajectory of the \texttt{ACCA}
dynamics and denote by $\delta_t(i)$ the wrapped increments at
time $t$.
Recall the stopping time $\tau$ defined in \eqref{eq:tau}.
On the event $\{t<\tau\}$ no branch--crossing occurs and midpoint
updates are well defined.

\begin{definition}[Lifted trajectory]
\label{def:lifted-process}
Fix an initial lift $\eta_0$ of $\theta_0$.
For $t<\tau$ construct $\eta_{t+1}$ recursively as follows. Suppose the edge $(k_t,k_t+1)$ is updated at time $t$.
Define the lifted midpoint
\[
m_t:=\frac{\eta_t(k_t)+\eta_t(k_t+1)}{2}.
\]
Set
\[
\eta_{t+1}(k_t)=\eta_{t+1}(k_t+1):=m_t .
\]
The remaining coordinates are determined by the increment
relations
\begin{equation}
\label{eq:lift-rec}
\eta_{t+1}(i+1)-\eta_{t+1}(i)=\delta_{t+1}(i),
\qquad i=1,\dots,N .
\end{equation}
The resulting process $(\eta_t)_{t\ge0}$ is called the
\emph{lifted representation} of the trajectory $(\theta_t)_{t\geq0}$.
\end{definition}
We have then the following projection property:
\begin{lemma}[Projection property of a lift]
\label{lem:lift-projection}
Let $\theta\in\Omega_N$ and let $\eta\in\R^{N+1}$ be its corresponding lift introduced in Definition~\ref{def:lift}. Then
\[
\wrap(\eta(i))=\theta(i),
\qquad i=1,\dots,N+1,
\]
where for $i=N+1$ the right-hand side is interpreted as $\theta(1)$.
\end{lemma}

\begin{proof}
We argue by induction on $i$. For $i=1$ the claim is exactly the anchor condition.
Assume now that for some $i\in\{1,\dots,N\}$ we have $
\wrap(\eta(i))=\theta(i)$.
By \eqref{eq:lift-inc}, we have:
\[
\wrap(\eta(i+1))
=
\wrap\!\big(\eta(i)+\delta(i)\big).
\]
Therefore, by the induction hypothesis:
\begin{equation}
\label{eq:proj1}
\wrap(\eta(i+1))
=
\wrap\!\big(\theta(i)+2\pi m_i+\delta(i)\big)
=
\wrap\!\big(\theta(i)+\delta(i)\big).
\end{equation}
By Remark~\ref{rem:wind},
there exists $\ell_i\in\mathbb Z$ such that 
$\delta(i)=\theta(i+1)-\theta(i)+2\pi \ell_i$,
so that
\[
\theta(i)+\delta(i)
=
\theta(i+1)+2\pi \ell_i.
\]
Hence, \eqref{eq:proj1} becomes:
\[
\wrap(\eta(i+1))
=
\wrap\!\big(\theta(i+1)+2\pi \ell_i\big)
=
\theta(i+1).
\]
This proves the induction step. Hence,
$\wrap(\eta(i))=\theta(i)$ for any $i\in\{i=1,\dots,N\}$.
Finally, for $i=N+1$ we use the closing relation \eqref{eq:closing} of Remark~\ref{rem:closing}, so that:
\[
\wrap(\eta(N+1))=\wrap(\eta(1))=\theta(1).
\]
This completes the proof.
\end{proof}
The previous definition specifies how a lifted trajectory should evolve
once the value of the updated edge is fixed. We now verify that this
construction is well posed. In particular, we show in Proposition~\ref{prop:lifted-process} that starting from an
initial lift there exists a unique lifted representation of the
trajectory that remains consistent with the circular dynamics. The next
proposition establishes the existence and uniqueness of the lifted
process. The following Lemma~\ref{lem:lift-midpoint} then records the key structural property of
this representation: in lifted coordinates the \texttt{ACCA} update
coincides exactly with the arithmetic midpoint operation on the updated
edge.
\begin{proposition}[Existence and uniqueness]
\label{prop:lifted-process}
Given an initial lift $\eta_0$, there exists a unique lifted
representation $(\eta_t)_{t\geq 0}$ associated with the trajectory
$(\theta_t)_{t\geq 0}$ on the event $\{t<\tau\}$.
\end{proposition}

\begin{proof}
We construct the sequence recursively in time.

At time $t=0$ the lift $\eta_0$ is given.
Assume that $\eta_t$ has already been constructed and is a lift of
$\theta_t$ in the sense of Definition~\ref{def:lift}.
Let $(k_t,k_t+1)$ be the edge updated at time $t$, where by convention
$k_t+1=N+1$ if $k_t=N$, and define
\[
m_t:=\frac{\eta_t(k_t)+\eta_t(k_t+1)}{2}.
\]
Set
\[
\eta_{t+1}(k_t)=\eta_{t+1}(k_t+1):=m_t .
\]
The remaining coordinates are uniquely determined by the increment
relations $
\eta_{t+1}(i+1)-\eta_{t+1}(i)=\delta_{t+1}(i)$,
for all $i\in\{1,\dots,N\}$,
which can be solved recursively starting from the known value
$\eta_{t+1}(k_t)$. This defines a unique vector $\eta_{t+1}\in\R^{N+1}$. By construction, these relations hold for all $i\neq k_t$, while
\[
\eta_{t+1}(k_t+1)-\eta_{t+1}(k_t)=0=\delta_{t+1}(k_t),
\]
because the midpoint update makes the two endpoints coincide.
Hence $
\eta_{t+1}(i+1)-\eta_{t+1}(i)=\delta_{t+1}(i)$, for all $i\in\{1,\dots,N\}$. It remains to verify the projection condition.
Since $\eta_t$ is a lift of $\theta_t$ by the recursion hypothesis, Lemma~\ref{lem:lift-projection}
gives
\[
\wrap(\eta_t(k_t))=\theta_t(k_t),
\qquad
\wrap(\eta_t(k_t+1))=\theta_t(k_t+1).
\]
Using $\eta_t(k_t+1)-\eta_t(k_t)=\delta_t(k_t)$ we obtain
\[
m_t=\eta_t(k_t)+\tfrac12\delta_t(k_t)
    =\eta_t(k_t+1)-\tfrac12\delta_t(k_t).
\]
Therefore
\[
\wrap(m_t)
=
\wrap\!\bigl(\theta_t(k_t)+\tfrac12\delta_t(k_t)\bigr)
=
\theta_{t+1}(k_t),
\]
and similarly
\[
\wrap(m_t)=\theta_{t+1}(k_t+1).
\]
Hence
\[
\wrap(\eta_{t+1}(k_t))=\theta_{t+1}(k_t),
\qquad
\wrap(\eta_{t+1}(k_t+1))=\theta_{t+1}(k_t+1).
\]
For any $i$ the increment relation implies
\[
\eta_{t+1}(i+1)=\eta_{t+1}(i)+\delta_{t+1}(i).
\]
Since
\[
\delta_{t+1}(i)=\wrap(\theta_{t+1}(i+1)-\theta_{t+1}(i)),
\]
by \eqref{eq:wrap1}, it follows that
\[
\wrap(\eta_{t+1}(i+1))=\theta_{t+1}(i+1)
\]
whenever $\wrap(\eta_{t+1}(i))=\theta_{t+1}(i)$.
Starting from the updated edge, this identity propagates along the
ring and yields
\[
\wrap(\eta_{t+1}(i))=\theta_{t+1}(i),
\qquad i=1,\dots,N+1,
\]
where $\theta_{t+1}(N+1)$ is interpreted as $\theta_{t+1}(1)$. Thus $\eta_{t+1}$ satisfies both the increment relations and the
projection condition of Definition~\ref{def:lift}, so it is a lift of
$\theta_{t+1}$. This proves existence.

Uniqueness follows because the value at the updated edge is fixed by
$m_t$, and the increment relations determine all remaining coordinates
uniquely.
\end{proof}
\begin{remark}[Interpretation of the closing edge in the lift]
\label{rem:closing-edge-lift}
The additional coordinate $\eta_t(N+1)$ represents the same physical
vertex as site $1$ after one complete traversal of the ring in the
covering space. By the closing relation of Definition~\ref{def:lift},
\[
\eta_t(N+1)=\eta_t(1)+2\pi W_t.
\]
Hence, the edge $(N,1)$ of the original ring is represented in the
lifted space by the segment joining $\eta_t(N)$ and $\eta_t(N+1)$.
Consequently, when the closing edge $(N,1)$ is updated by the \texttt{ACCA}
dynamics, the lifted midpoint is computed between the two lifted
neighbors $\eta_t(N)$ and $\eta_t(N+1)$. The update therefore takes
the form
\[
\eta_{t+1}(N)=\eta_{t+1}(N+1)
=\frac{\eta_t(N)+\eta_t(N+1)}{2}=\frac{\eta_t(N)+\eta_t(1)+2\pi W_t}{2}.
\]
After this midpoint assignment, the remaining coordinates are
reconstructed using the increment relations
\[
\eta_{t+1}(i+1)-\eta_{t+1}(i)=\delta_{t+1}(i).
\]
In particular, the coordinate $\eta_{t+1}(1)$ is determined
automatically by the reconstruction together with the anchor
condition $\wrap(\eta_{t+1}(1))=\theta_{t+1}(1)$.
\end{remark}
\begin{lemma}[Exact midpoint identity in the lift]
\label{lem:lift-midpoint}
Let $(\eta_t)_t$ be the lifted representation of $(\theta_t)_t$.
If the edge $(k_t,k_t+1)$ is updated at time $t$, then
\[
\eta_{t+1}(k_t)=\eta_{t+1}(k_t+1)
=
\frac{\eta_t(k_t)+\eta_t(k_t+1)}{2}.
\]
\end{lemma}

\begin{proof}
The identity follows directly from the definition of the lifted
trajectory.
\end{proof}
%
% SUBSECTION: trapping in winding sectors: a heuristic picture
%
\subsection{Trapping in winding sectors: a heuristic picture}
\label{sec:euri}
Lemma~\ref{lem:first-crossing-probability} shows that branch-crossing
events occur with strictly positive probability when the configuration
is initially disordered. At early times the wrapped increments
$\delta_t(i)$ are typically large, and midpoint updates may therefore
push neighboring differences outside the principal interval
$(-\pi,\pi)$, producing changes in the winding number.
At the same time, however, the \texttt{ACCA} dynamics is locally contractive:
repeated midpoint averaging progressively smooths phase differences
and suppresses large gradients. As the configuration flattens,
branch--crossings become increasingly unlikely. Once all local
increments remain inside $(-\pi,\pi)$, the winding number can no
longer change and the dynamics becomes confined to a fixed winding
sector.

This transition from an active crossing phase to a frozen winding
sector is visible in the numerical experiments shown in
Figures~\ref{fig:sim5_per}--\ref{fig:W_per}. These simulations illustrate the winding-sector trapping mechanism that will be rigorously analyzed in the following subsections.

Figure~\ref{fig:sim5_per} shows the evolution of a system with
$N=1000$ under periodic boundary conditions. Initially the wrapped
differences $\delta_0(i)$ are widely spread and branch--crossings
occur frequently, causing repeated fluctuations of the winding number.
As the dynamics proceeds the profile becomes progressively smoother,
and extended monotone segments appear. Once all increments remain
inside $(-\pi,\pi)$ the winding number stabilizes. The configuration
then approaches a twisted profile with approximately constant slope
$2\pi W/N$, as visible in panel~(d).

The geometric meaning of this state becomes clearer in cylindrical
coordinates (Figure~\ref{fig:sim5_per_cyl}). In this representation the
configuration wraps around the cylinder exactly $W$ times, forming a
helical structure. Local midpoint updates continue to smooth the
configuration, but the global twist persists because the winding
number is fixed.

The evolution of the winding number itself is shown in
Figure~\ref{fig:W_per}. During the initial transient phase $W_t$
changes several times, reflecting the positive crossing probability
computed in Lemma~\ref{lem:first-crossing-probability}. Once the
configuration enters a \emph{no--branch--crossing corridor}, $W_t$
becomes constant and remains so for very long time intervals.

These simulations reveal a clear separation of time scales. Local
averaging rapidly smooths the configuration, while the winding number
changes only during rare crossing events. Once crossings cease, the
system evolves for long times near a twisted configuration determined
by the frozen winding number. Escape from this winding sector
requires a new branch--crossing event, after which the same mechanism
repeats until the winding number eventually reaches zero and global
consensus is achieved. The next subsections formalize this picture.

% N=1000 periodic
\begin{figure}[H]
    \centering
    \subfigure[]{\includegraphics[width=0.44\textwidth]{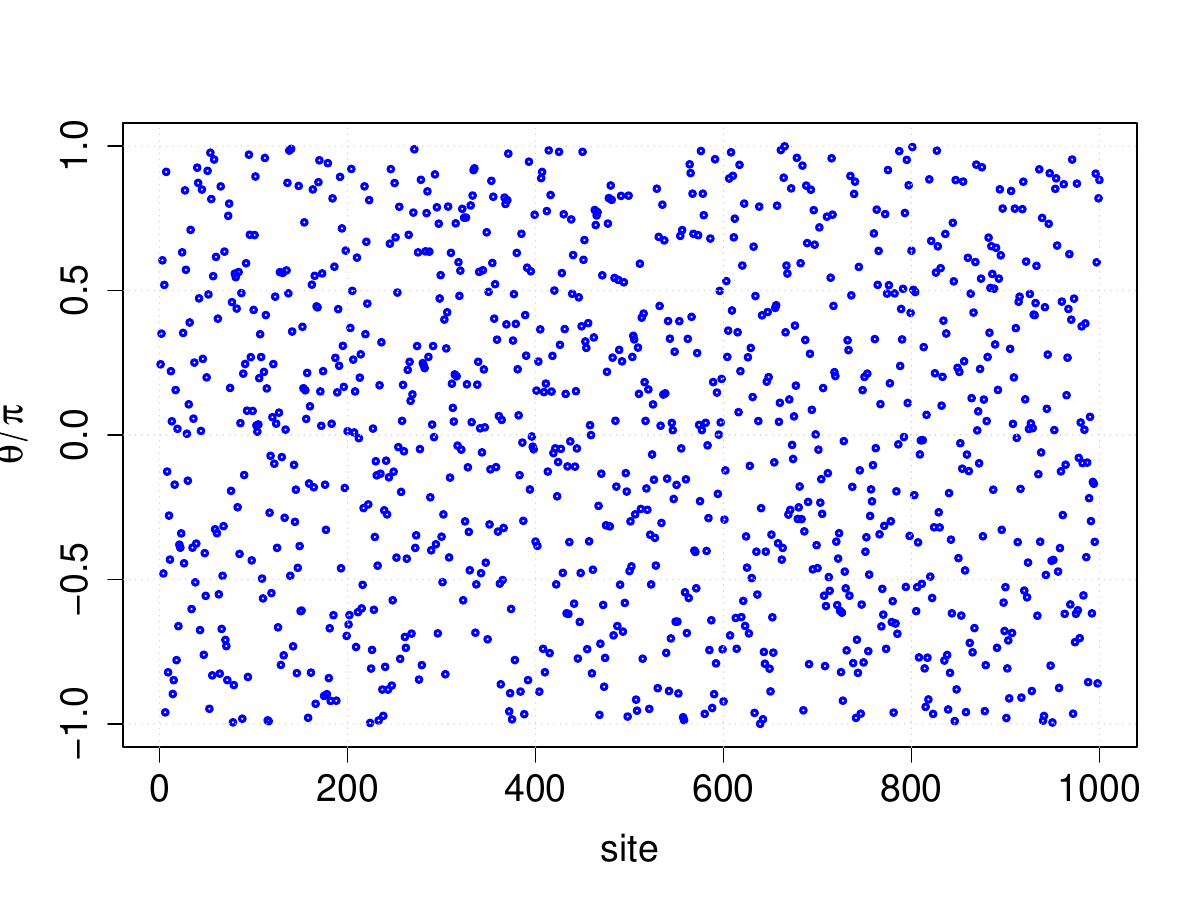}} 
    \subfigure[]{\includegraphics[width=0.44\textwidth]{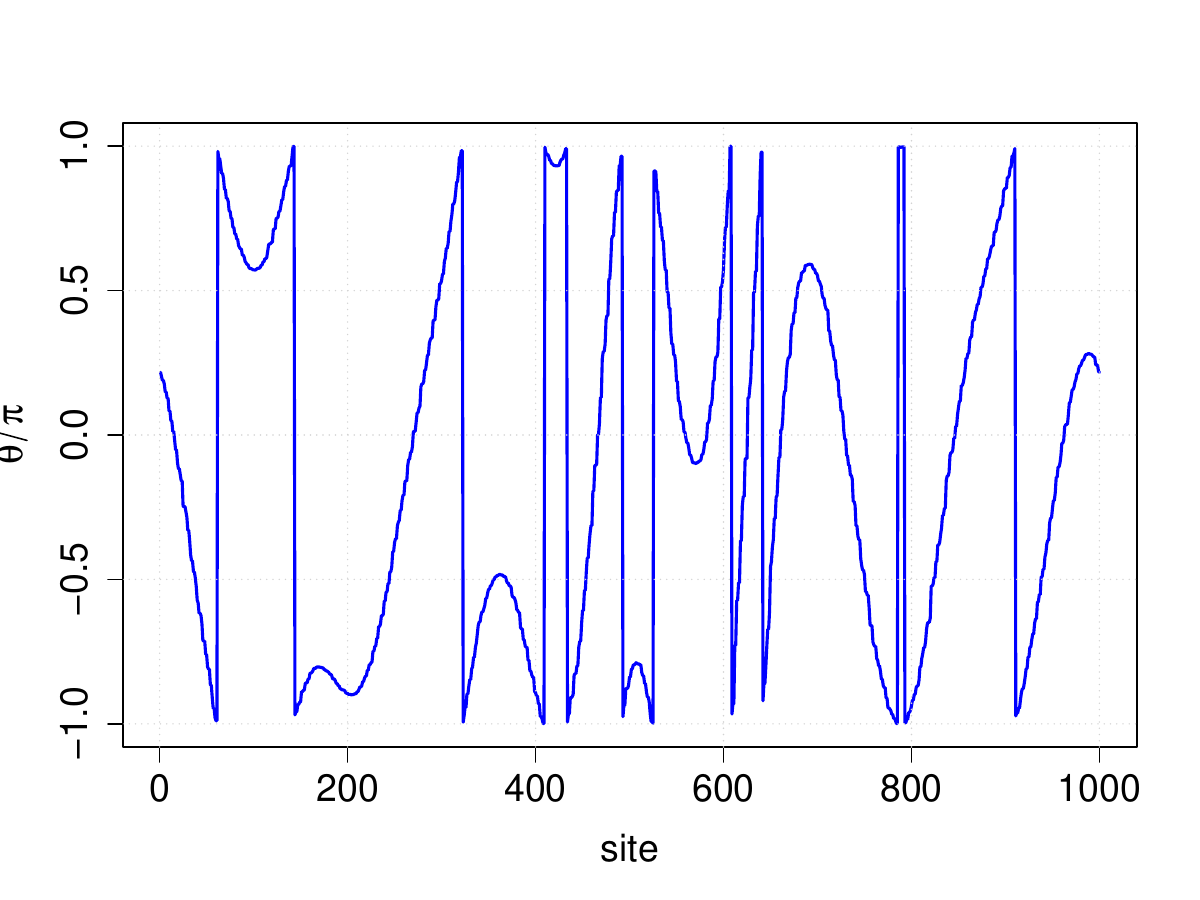}} 
    \subfigure[]{\includegraphics[width=0.44\textwidth]{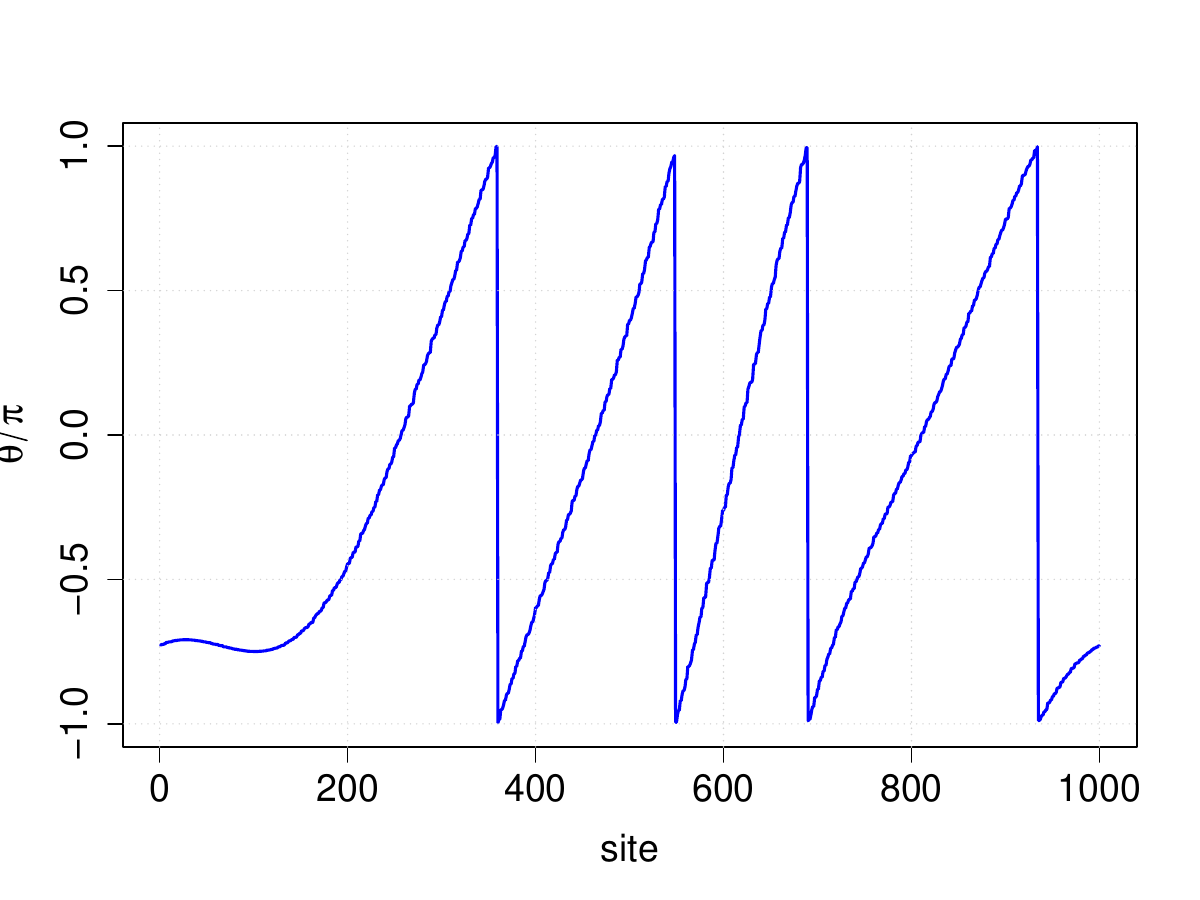}}
    \subfigure[]{\includegraphics[width=0.44\textwidth]{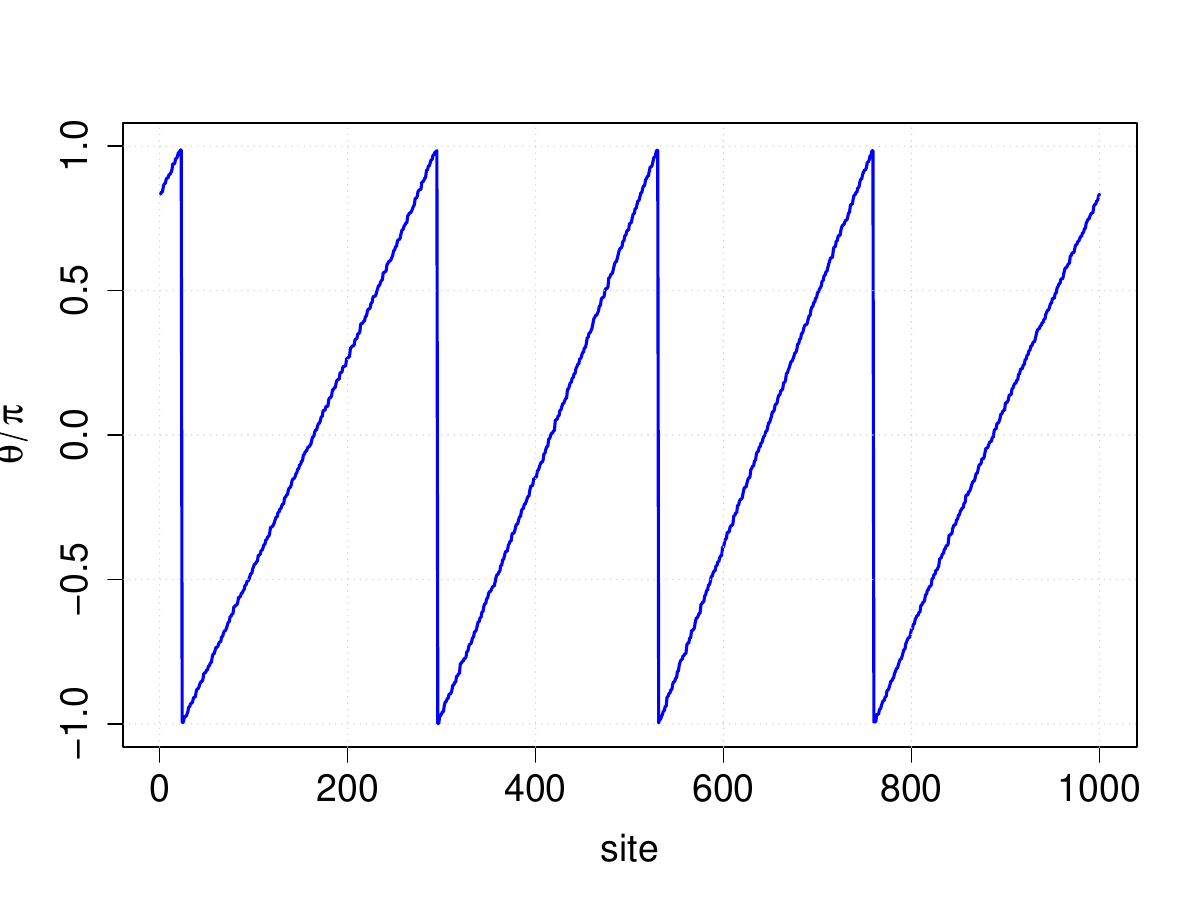}}
    \caption{Simulation for a system with $N=1000$ with periodic boundary conditions (a) is the starting opinion distribution  (b) configuration at $t=10^7$ with temporary winding configurations  (c) at $t= 10^8$  (d) at $t=10^9$ the configuration has $W=4$.}
    \label{fig:sim5_per}
\end{figure}

\begin{figure}[H]
    \centering
    \subfigure[]{\includegraphics[width=0.49\textwidth]{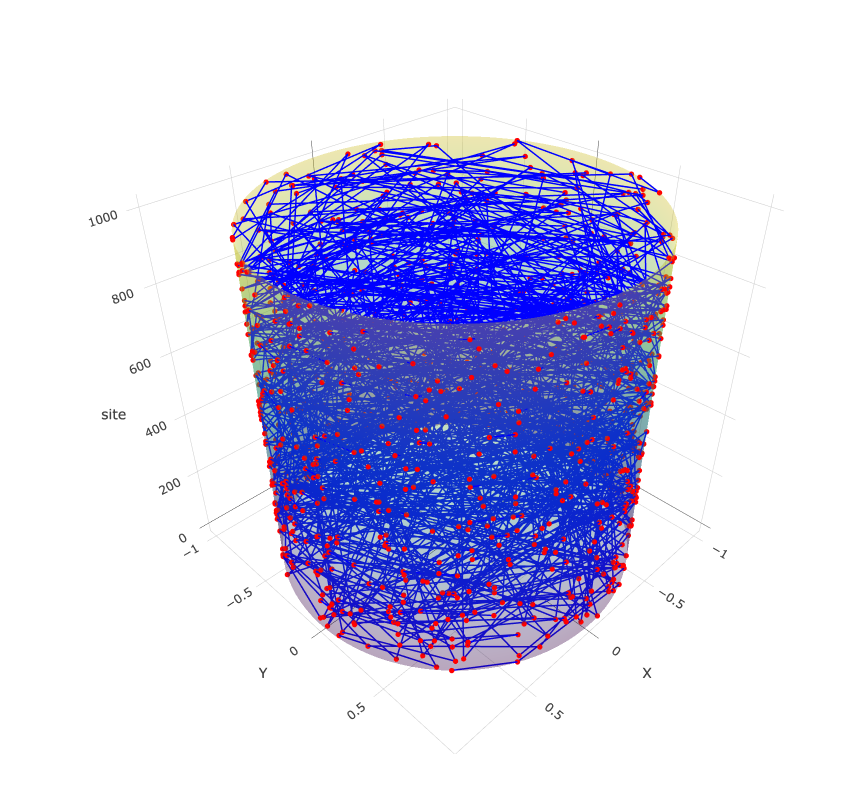}} 
    \subfigure[]{\includegraphics[width=0.49\textwidth]{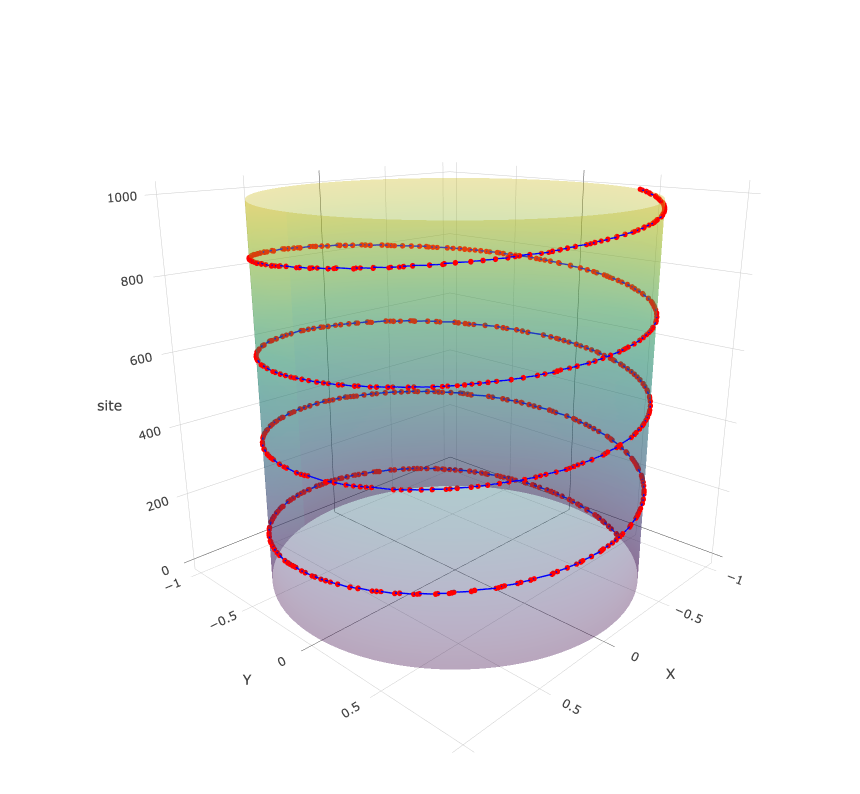}} 
    
    \caption{Simulation for a system with $N=1000$  as in Figure~\ref{fig:sim5_per} cylindrical coordinates  (a) is the starting opinion distribution  (b) configuration at   at $t=10^9$.}
    \label{fig:sim5_per_cyl}
\end{figure}
\begin{figure}[H]
    \centering
\includegraphics[height=8cm]{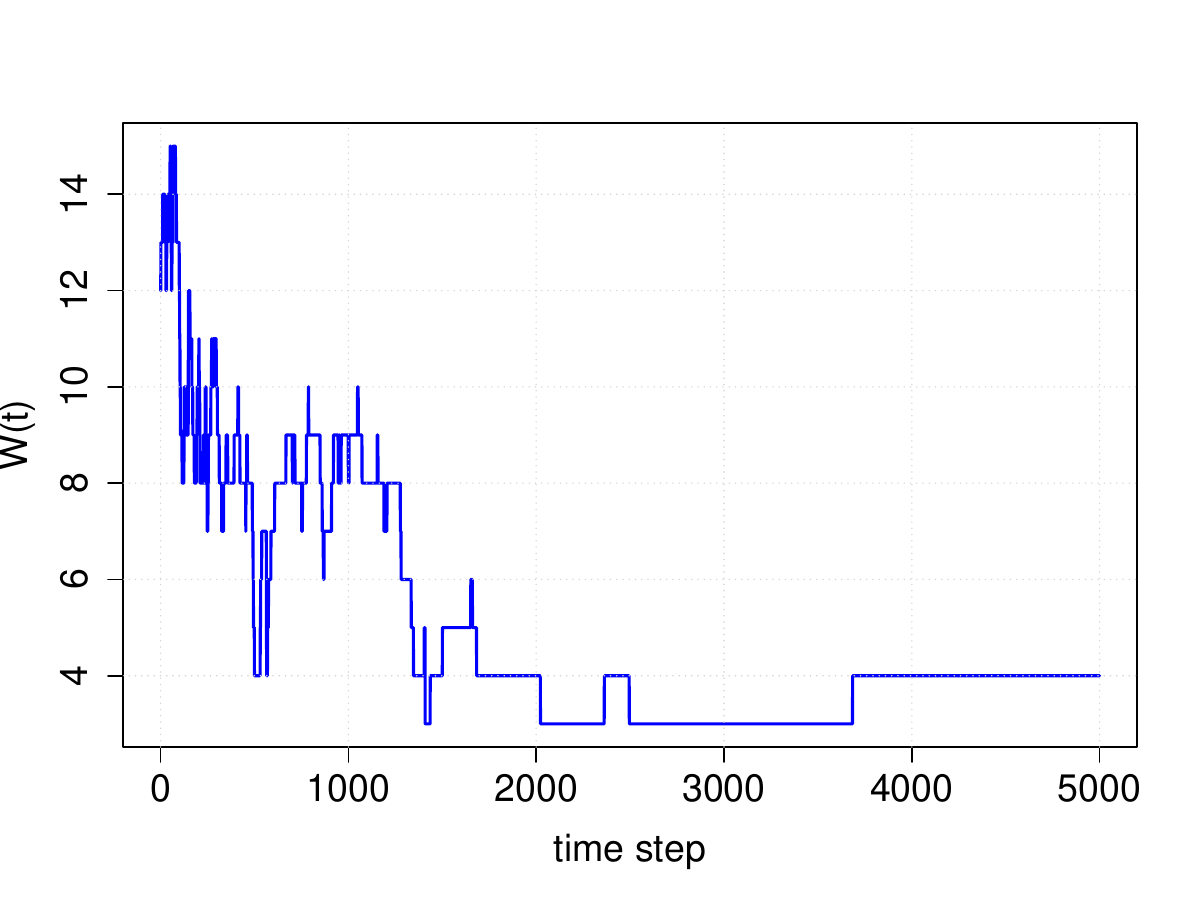}
    \caption{Evolution in time of the winding number of the simulation of  Figure~\ref{fig:sim5_per}. Notice that from $t=3700$ $W_t$ remains constant at value $4.$}
    \label{fig:W_per}
\end{figure}

%
% SUBSECTION: rigorous results on the relaxation around a winding configuration
%
\subsection{Rigorous results on the relaxation around a winding configuration}
\label{sec:meta0}

We now turn this heuristic picture of winding-sector trapping into a rigorous statement.
Throughout this subsection we work under periodic boundary conditions
and assume Assumption~\ref{asmp:tau-infty}, so that no branch--crossing
or antipodal edge ever occurs.
By Theorem~\ref{thm:W-conditional} this implies that the winding number is frozen:
\[
W_t=W_0\qquad\text{for all }t\ge 0.
\]
In particular, the (deterministic) slope
\begin{equation}\label{eq:beta}
\beta := \frac{2\pi W_0}{N}
\end{equation}
is constant. A configuration with winding $W_0$ is therefore naturally
decomposed into a linear profile of slope $\beta$ plus fluctuations.

The analysis is carried out in the lifted variables introduced in
Section~\ref{sec:lift}. In particular, by Lemma~\ref{lem:lift-midpoint} the midpoint update of the
\texttt{ACCA} dynamics corresponds exactly to an arithmetic midpoint update
in the lifted process $\eta_t$. The key mechanism is that local midpoint averaging suppresses fluctuations
but cannot remove the linear mode imposed by the winding number. As a
result, the dynamics contracts only after subtracting a deterministic
drift generated by the winding sector.
However, a perfectly linear function $\eta(i)=\alpha+\beta i$ is not invariant
under \texttt{ACCA} dynamics:
when an edge $(k,k+1)$ is selected, the two sites are replaced by their
arithmetic mean. On a linear profile this midpoint equals
$\alpha+\beta k+\tfrac{\beta}{2}$, which differs from the values
$\alpha+\beta k$ and $\alpha+\beta(k+1)$ by deterministic shifts
$\pm\beta/2$.
Thus each update introduces a local drift relative to the linear trend,
and these shifts accumulate along the update history.

The main goal is to isolate a \emph{co--moving frame} in which this
deterministic drift is removed and the remaining evolution becomes an
exact Euclidean midpoint averaging process.
In that frame we establish strict contraction and almost sure convergence
to consensus.
To achieve this, we introduce a \emph{compensator} process $s_t$ that records the accumulated
drift and allows us to subtract it exactly from the lifted dynamics (Theorem~\ref{thm:adaptive-detrended}). The compensator process therefore tracks the deformation of a linear
profile with slope $\beta$ under the same sequence of midpoint updates.

\begin{definition}[Compensator process $(s_t)_t$]
\label{def:comp}
We define recursively the \emph{compensator} $s_t\in\R^N$.
Set $s_0\equiv 0$.
When the edge $(k,k+1)$ is updated with $1\le k\le N-1$, define
\begin{eqnarray}
s_{t+1}(k)&=&\frac{s_t(k)+s_t(k+1)}{2}+\frac{\beta}{2}, \nonumber \\
s_{t+1}(k+1)&=&\frac{s_t(k)+s_t(k+1)}{2}-\frac{\beta}{2}, \label{eq:s-update1}\\
s_{t+1}(i)&=&s_t(i),\qquad i\notin \{k,k+1\}. \nonumber
\end{eqnarray}
If instead the closing edge $\{N,1\}$ is updated, we define
\begin{eqnarray}
s_{t+1}(N)&=&\frac{s_t(N)+s_t(1)}{2}+\frac{\beta}{2}, \nonumber \\
s_{t+1}(1)&=&\frac{s_t(N)+s_t(1)}{2}-\frac{\beta}{2}, \label{eq:s-update2}\\
s_{t+1}(i)&=&s_t(i),\qquad i\notin\{N,1\}.\nonumber
\end{eqnarray}
\end{definition}
Since the lifted configuration is indexed by $\{1,\dots,N+1\}$,
we extend the compensator to this index set by defining
\[
s_t(N+1):=s_t(1).
\]
Thus, $s_t$ is periodic on the lifted path, consistently with the
identification of vertices $1$ and $N+1$ in the covering space.
We give a basic uniform-in-time second-moment bound for the compensator. %Although the convergence of the detrended dynamics will follow from the exact
%midpoint structure of $\zeta_t$, 
Proposition~\ref{prop:s-uniform-bound-per-site}
shows indeed that the compensator remains uniformly controlled per site and therefore
does not generate unbounded fluctuations in the reconstructed profile. The proof is given in Appendix~\ref{sec:appB}.
\begin{proposition}[Uniform-in-time per-site $L^2$ bound for the compensator]
\label{prop:s-uniform-bound-per-site}
Let $(s_t)_{t\ge0}$ be defined by \eqref{eq:s-update1}--\eqref{eq:s-update2}
with $s_0\equiv 0$, and assume that at each step the updated edge is chosen uniformly
among the $N$ edges of the ring. Then $\sum_{i=1}^N s_t(i)=0$ for all $t$, and
\begin{equation}\label{prop:s-uniform-bound-elementary}
\sup_{t\ge0}\,\E\Big[\frac1N\|s_t\|_2^2\Big] \;\le\; C\,(2\pi W_0)^2,
\end{equation}
for some constant $C>0$ independent of $N$.
Equivalently,
\begin{equation}\label{eq:s-uniform-L2}
\sup_{t\ge0}\,\E\|s_t\|_2^2 \;\le\; C\,N.
\end{equation}
\end{proposition}

\begin{definition}[Detrended process]
\label{def:detrended}
We introduce the detrended process $(\zeta_t(i))_i$:
\begin{equation}
\label{eq:detrended}
\zeta_t(i):=\eta_t(i)-\beta(i-1)-s_t(i),
\qquad i=1,\dots,N+1,
\end{equation}
and
\[
\bar{\zeta}_t:=\frac1N\sum_{i=1}^N \zeta_t(i).
\]
\end{definition}
\begin{remark}
\label{rem:detrended}
Since
\[
\eta_t(N+1)=\eta_t(1)+2\pi W_t,
\qquad
2\pi W_t=\beta N,
\]
and $s_t(N+1)=s_t(1)$, it follows that
\[
\zeta_t(N+1)=\zeta_t(1).
\]
\end{remark}

A perfectly twisted profile of winding $W_0$ has the deterministic slope $\beta$,
i.e.\ it is of the form $\eta(i)=\alpha+\beta i$.
However, such a profile is \emph{not} invariant under asynchronous midpoint updates:
if an interior edge $(k,k+1)$ is updated, then the arithmetic midpoint equals
\[
\frac{\beta k+\beta(k+1)}{2}=\beta k+\frac{\beta}{2},
\]
so, relative to the linear trend $\beta i$, each update introduces a deterministic
local shift of size $\pm\beta/2$ at the updated sites.
Over time these shifts accumulate in a spatially nonuniform way, depending on the random
update history. The compensator $s_t$ is designed to reproduce exactly this accumulated bias.
Subtracting $s_t$ cancels the history-dependent drift and reveals a purely contractive
midpoint dynamics in the variables $\zeta_t$.

\medskip

By using the lifted midpoint identities of Lemma~\ref{lem:lift-midpoint},
we now show that after subtracting the deterministic trend $\beta i$ and the
compensator $s_t$, the dynamics becomes an exact Euclidean midpoint update
in the detrended variables $\zeta_t$.
\begin{lemma}[Exact midpoint structure in detrended variables]
\label{lem:compensator-consistency}
Assume Assumption~\ref{asmp:no-antipodes} so that $W_t=W_0$ for all $t$,
and let $\beta=\frac{2\pi W_0}{N}$.
Define the compensator $(s_t)_t$ process by \eqref{eq:s-update1}--\eqref{eq:s-update2}.
If the edge $(k_t,k_t+1)$ is updated at time $t$, then
\[
\zeta_{t+1}(k_t)=\zeta_{t+1}(k_t+1)
=
\frac{\zeta_t(k_t)+\zeta_t(k_t+1)}{2},
\]
while all other coordinates remain unchanged.
In particular, the $\zeta$-dynamics is an exact Euclidean midpoint update
on the selected edge.
\end{lemma}

\begin{proof}
By Lemma~\ref{lem:lift-midpoint}, the lifted variables satisfy
\[
\eta_{t+1}(k_t)=\eta_{t+1}(k_t+1)
=
\frac{\eta_t(k_t)+\eta_t(k_t+1)}{2}.
\]
Moreover, by construction of the compensator, one has
\[
s_{t+1}(k_t)=\frac{s_t(k_t)+s_t(k_t+1)}{2}+\frac{\beta}{2},
\qquad
s_{t+1}(k_t+1)=\frac{s_t(k_t)+s_t(k_t+1)}{2}-\frac{\beta}{2},
\]
where in the case $k_t=N$ we use the convention $s_t(N+1)=s_t(1)$.
We now compute
\begin{align*}
\zeta_{t+1}(k_t)
&=\eta_{t+1}(k_t)-\beta(k_t-1)-s_{t+1}(k_t)\\
&=\frac{\eta_t(k_t)+\eta_t(k_t+1)}{2}
   -\beta(k_t-1)
   -\frac{s_t(k_t)+s_t(k_t+1)}{2}
   -\frac{\beta}{2}\\
&=\frac{\eta_t(k_t)-\beta(k_t-1)-s_t(k_t)}{2}
  +\frac{\eta_t(k_t+1)-\beta k_t-s_t(k_t+1)}{2}\\
&=\frac{\zeta_t(k_t)+\zeta_t(k_t+1)}{2}.
\end{align*}
Similarly,
\begin{align*}
\zeta_{t+1}(k_t+1)
&=\eta_{t+1}(k_t+1)-\beta k_t-s_{t+1}(k_t+1)\\
&=\frac{\eta_t(k_t)+\eta_t(k_t+1)}{2}
   -\beta k_t
   -\frac{s_t(k_t)+s_t(k_t+1)}{2}
   +\frac{\beta}{2}\\
&=\frac{\eta_t(k_t)-\beta(k_t-1)-s_t(k_t)}{2}
  +\frac{\eta_t(k_t+1)-\beta k_t-s_t(k_t+1)}{2}\\
&=\frac{\zeta_t(k_t)+\zeta_t(k_t+1)}{2}.
\end{align*}
For every $i\notin\{k_t,k_t+1\}$, neither $\eta_t(i)$ nor $s_t(i)$ is modified by
the update, hence
\[
\zeta_{t+1}(i)=\eta_{t+1}(i)-\beta(i-1)-s_{t+1}(i)
=\eta_t(i)-\beta(i-1)-s_t(i)
=\zeta_t(i).
\]
Therefore the detrended variables evolve exactly by Euclidean midpoint
averaging on the selected edge.
\end{proof}

The previous lemma shows that, after subtracting the deterministic trend $\beta i$ and
the compensator $s_t(i)$, the dynamics reduces to Euclidean midpoint averaging.
Equivalently, the fluctuation field $\zeta_t$ evolves \emph{autonomously}:
although $s_t$ is defined from the same update history, all terms involving $\beta$ and $s_t$
cancel exactly in the update of $\zeta_t$, leaving the pure midpoint rule.
This cancellation is the key to the contraction statement below.

\begin{theorem}[Convergence of the configuration toward a detrended limit]
\label{thm:adaptive-detrended}
Under the \texttt{ACCA} dynamics and under Assumption~\ref{asmp:tau-infty},
there exists a (random) constant $\alpha^\star\in\R$ such that
\begin{equation}\label{eq:config-to-detrended-limit}
\lim_{t\to\infty}\max_{1\le i\le N}\big|\zeta_t(i)-\alpha^\star\big|
=0 \qquad \mathrm{a.s.}
\end{equation}
\end{theorem}

\begin{proof}
By Definition~\ref{def:detrended} and Lemma~\ref{lem:compensator-consistency}, each update replaces
$\zeta_t(k)$ and $\zeta_t(k+1)$ by their midpoint and leaves all other coordinates unchanged.
In particular, the mean $\bar\zeta_t$ is preserved. Define the variance functional
\begin{equation}
\label{eq:variance}
\widetilde\Psi_t:=\sum_{i=1}^N\big(\zeta_t(i)-\bar\zeta_t\big)^2.
\end{equation}
If edge $(k,k+1)$ is updated, only the two terms at $k$ and $k+1$ change.
Writing $a:=\zeta_t(k)-\bar\zeta_t$ and $b:=\zeta_t(k+1)-\bar\zeta_t$,
the midpoint update sends both to $(a+b)/2$, hence
\[
\widetilde\Psi_{t+1}-\widetilde\Psi_t
=
2\Big(\frac{a+b}{2}\Big)^2-(a^2+b^2)
=
-\frac12(a-b)^2
=
-\frac12\big(\zeta_t(k+1)-\zeta_t(k)\big)^2\le 0.
\]
Thus $(\widetilde\Psi_t)_t$ is nonincreasing and convergent, and along the realized updates
\begin{equation}\label{eq:sq-sum-on-updates}
\sum_{t=0}^\infty \big(\zeta_t(k_t+1)-\zeta_t(k_t)\big)^2
\le 2\,\widetilde\Psi_0<\infty,
\end{equation}
where $k_t$ denotes the updated edge at time $t$.

We now prove almost-sure consensus for $\zeta_t$.
Let $R_t:=\max_i\zeta_t(i)-\min_i\zeta_t(i)$ denote the diameter.
A midpoint update is a convexification step, hence $R_{t+1}\le R_t$.

Fix $t$ and let $u_-$ and $u_+$ be indices attaining the minimum and maximum at time $t$.
Choose a simple path along the ring from $u_-$ to $u_+$ with vertices
$u_0=u_-,u_1,\dots,u_{m}=u_+$, where $m\le N-1$, and consider the associated edge word
\[
\mathbf{w}_t=(\{u_0,u_1\},\{u_1,u_2\},\dots,\{u_{m-1},u_m\}).
\]
If, starting from time $t$, the sequence of updated edges satisfies
\[
k_t=\{u_0,u_1\},\;
k_{t+1}=\{u_1,u_2\},\;
\dots,\;
k_{t+m-1}=\{u_{m-1},u_m\},
\]
then 
\begin{equation}\label{eq:diam-contract}
R_{t+m}\le \Big(1-2^{-m}\Big)\,R_t \;\le\; \Big(1-2^{-(N-1)}\Big)\,R_t.
\end{equation}
Indeed, along such a sequence the value at $u_m$ is replaced by the average of the two
endpoint values with weights at least $2^{-m}$ and $1-2^{-m}$, hence the maximum is pulled
toward the minimum by at least a fraction $2^{-m}$, while the minimum cannot decrease
under midpoint averaging. This yields \eqref{eq:diam-contract}.

Now partition time into disjoint blocks of length $N-1$ and define $A_\ell$ to be the event
that the $\ell$-th block realizes the specific word $\mathbf{w}_{t_\ell}$ chosen at the
beginning of that block, where $t_\ell=\ell(N-1)$.
Conditionally on $\mathcal F_{t_\ell}$, the probability of $A_\ell$ is at least
\[
\mathbb P(A_\ell\mid \mathcal F_{t_\ell}) \;\ge\; N^{-(N-1)}=:p>0,
\]
since each of the next $N-1$ edge choices is uniform over $N$ edges.
Moreover, the edge choices in disjoint blocks are independent, so the events $(A_\ell)$
satisfy a conditional Borel--Cantelli argument and occur infinitely often almost surely.
Whenever $A_\ell$ occurs, \eqref{eq:diam-contract} implies that $R_t$ is multiplied by a
factor at most $q:=1-2^{-(N-1)}\in(0,1)$.
Since $(R_t)_t$ is nonincreasing, it follows that $R_t\to 0$ almost surely.

Hence $\zeta_t$ converges almost surely to a constant $\alpha^\star$ (necessarily
$\alpha^\star=\lim_{t\to\infty}\bar\zeta_t=\bar\zeta_0$), and therefore
$\max_i|\zeta_t(i)-\alpha^\star|\to 0$ almost surely.
\end{proof}

\begin{remark}\label{rem:beta0}
If $W_0=0$ then $\beta=0$ and the compensator $s_t$ is identically zero.
In that case $\zeta_t=\eta_t$ and Theorem~\ref{thm:adaptive-detrended}
reduces to almost sure consensus for the lifted midpoint dynamics.
\end{remark}
Theorem~\ref{thm:adaptive-detrended} implies that $\zeta_t$ becomes spatially flat, as one can see in the simulation in Figure~\ref{fig:zeta}, which visualizes the structural decomposition underlying
Theorem~\ref{thm:adaptive-detrended}.
In the adaptive variables $\zeta_t$ the midpoint dynamics is strictly contractive and converges to consensus,
whereas in the original angular variables $\theta_t$ the preserved topological twist prevents flattening.
\begin{figure}[H]
\centering
\begin{minipage}{0.49\textwidth}
    \centering
    \includegraphics[width=\linewidth]{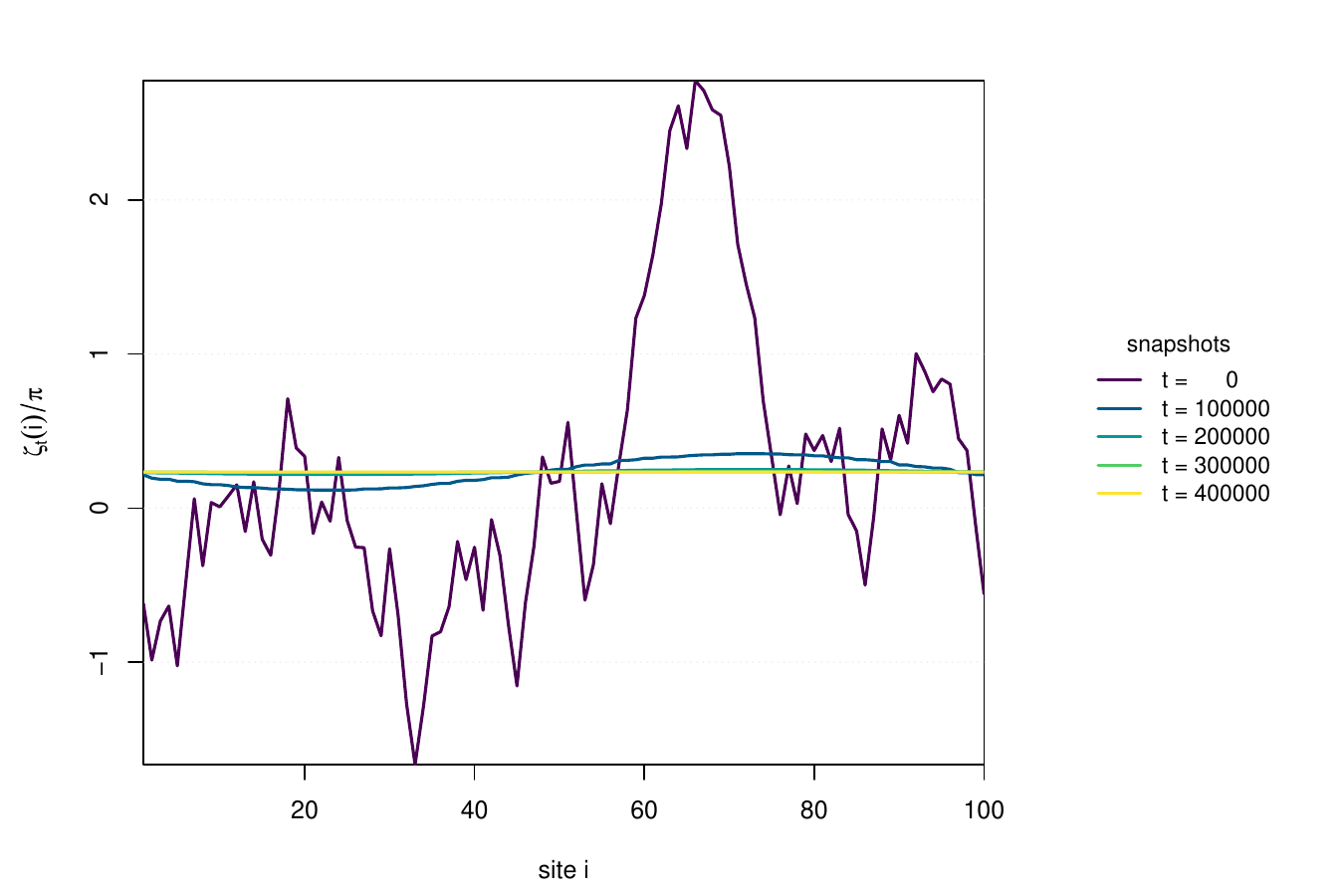}
    \vspace{-2mm}
    
    \small (a)
\end{minipage}\hfill
\begin{minipage}{0.49\textwidth}
    \centering
    \includegraphics[width=\linewidth]{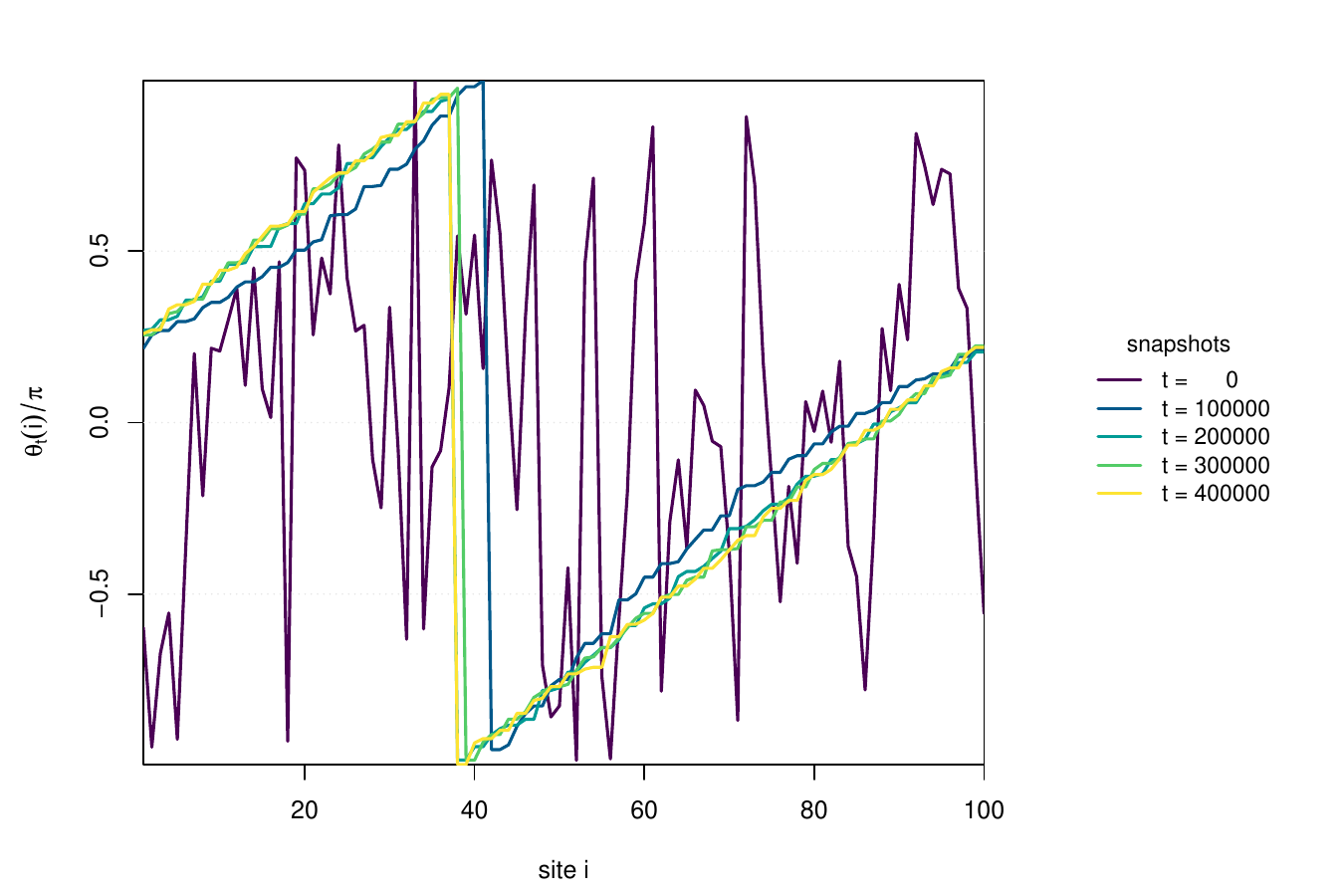}
    \vspace{-2mm}
    
    \small (b)
\end{minipage}

\caption{
Comparison between the detrended dynamics and the original angular dynamics under periodic boundary conditions with conserved winding number $W_0=1$ for a simulation with $N=100$. We started from a random configuration satisfying Assumption~\ref{asmp:no-antipodes}.
(a) Snapshots of the detrended variable $\zeta_t(i)/\pi$ at increasing times,
 illustrating the  contraction of the  midpoint dynamics in the co--moving frame (Theorem~\ref{thm:adaptive-detrended}), leading to almost sure consensus in $\zeta$.
(b) Corresponding snapshots of the angular configuration $\theta_t(i)/\pi$.
 The evolution remains organized around a twisted ramp structure.
}
\label{fig:zeta}
\end{figure}
Translating back to edge increments yields an exact decomposition, for every $i=1,\dots,N$,
\begin{equation}\label{eq:delta-decomp}
\delta_t(i)
=
\beta
+
\bigl(s_t(i+1)-s_t(i)\bigr)
+
\bigl(\zeta_t(i+1)-\zeta_t(i)\bigr).
\end{equation}
The decomposition \eqref{eq:delta-decomp} separates the dynamics
into three distinct components. The term $\beta=2\pi W_0/N$ represents the
deterministic slope fixed by the winding sector, the difference
$s_t(i+1)-s_t(i)$ describes the drift generated by the history of midpoint
updates, and the fluctuation term $\zeta_t(i+1)-\zeta_t(i)$ captures the
remaining deviations from the co–moving profile. 

By Theorem~\ref{thm:adaptive-detrended}, the fluctuation field $\zeta_t$
converges almost surely to a constant configuration. Consequently the last
term in \eqref{eq:delta-decomp} vanishes asymptotically and the
configuration approaches a twisted linear profile of slope $\beta$ that is
transported by the compensator process $s_t$.

Motivated by this decomposition, the natural notion of distance from the
twisted co–moving profile is the following:

\begin{equation}
\label{eq:dtilde}
\widetilde{\mathcal D}_t
:=
\frac{1}{N}\sum_{i=1}^N
\Big(\delta_t(i)-\beta-\big(s_t(i{+}1)-s_t(i)\big)\Big)^2.
\end{equation}
The next theorem shows that this distance vanishes almost surely (hence also in
time average, in expectation, and in probability).

\begin{theorem}[Closeness to the co--moving increment profile]
\label{thm:closeness-winding}
Assume  Assumption~\ref{asmp:tau-infty}.
Let $\eta_t\in\R^{N+1}$ be the lifted configuration, let $s_t$ be the compensator defined in
\eqref{eq:s-update1}--\eqref{eq:s-update2}, and
$\zeta_t(i)$ defined as in Definition~\ref{def:detrended}. Then
\[
\lim_{t\to\infty}
\widetilde{\mathcal D}_t=0, \qquad \mathrm{a.s.},
\]
where $\widetilde{\mathcal D}_t$ is defined in \eqref{eq:dtilde}.

\end{theorem}

\begin{proof}
By definition,
\[
\delta_t(i)=\eta_t(i+1)-\eta_t(i).
\]
Using $\eta_t(i)=\zeta_t(i)+\beta (i-1)+s_t(i)$, we obtain
\[
\delta_t(i)
=
\zeta_t(i+1)-\zeta_t(i)+\beta+\big(s_t(i+1)-s_t(i)\big),
\]
hence
\[
\delta_t(i)-\beta-\big(s_t(i+1)-s_t(i)\big)=\zeta_t(i+1)-\zeta_t(i).
\]
Therefore
\[
\widetilde{\mathcal D}_t=
\frac{1}{N}\sum_{i=1}^N\big(\zeta_t(i+1)-\zeta_t(i)\big)^2.
\]
By Theorem~\ref{thm:adaptive-detrended}, $\zeta_t$ converges almost surely to a constant:
there exists $\alpha^\star$ such that $\max_i|\zeta_t(i)-\alpha^\star|\to 0$ a.s.
Consequently,
\[
\max_i|\zeta_t(i+1)-\zeta_t(i)|
\le 2\max_j|\zeta_t(j)-\alpha^\star|
\longrightarrow 0
\qquad\text{a.s.}
\]
It follows that $\widetilde{\mathcal D}_t\to 0$ almost surely.
\end{proof}

\begin{remark}\label{rem:meaning-metastable}
Theorem~\ref{thm:adaptive-detrended} states that, in the co--moving frame defined by
subtracting $\beta i+s_t(i)$, the fluctuations $\zeta_t$ reach consensus almost surely.
Equivalently, the lifted configuration $\eta_t$ becomes asymptotically close (in sup norm)
to some affine twisted profile of the form $\alpha+\beta i+s_t(i)$, where the shift
$\alpha$ is constant and the time dependence enters only through the compensator.
Theorem~\ref{thm:closeness-winding} is the edge-increment counterpart:
it shows that, in mean square, the increment field $\delta_t(i)$ stays close to the
co--moving increment profile $\beta+(s_t(i+1)-s_t(i))$.
\end{remark}

%
% SUBSECTION: a constructive escape mechanism
%

\subsection{A constructive escape mechanism}
\label{sec:escape-heur}

The previous subsections describe the dynamics inside a fixed winding
sector. Once the system enters the no--branch--crossing corridor
(Assumption~\ref{asmp:no-antipodes}), the winding number remains constant
and the midpoint updates rapidly smooth the configuration toward a
twisted profile compatible with that winding number.
Nevertheless, for every fixed $N$ the stochastic dynamics converges almost
surely to consensus. Consequently, the process cannot remain indefinitely
in a nonzero winding sector: with probability one a branch--crossing must
eventually occur, allowing the winding number to change.

The purpose of this subsection is to exhibit an explicit mechanism that
produces such an escape. Rather than attempting a full probabilistic
analysis of the exit time, we construct a deterministic update pattern
that unwinds a twisted configuration. Since any finite update sequence has
positive probability of occurring in the stochastic dynamics, the
existence of such a sequence guarantees that escape from a nonzero winding
sector is always possible.

\medskip

After the initial transient phase—during which branch crossings may occur
frequently (Lemma~\ref{lem:first-crossing-probability}), the dynamics
typically enter a regime where the winding number remains frozen for very
long times. In this regime the configuration approaches a smooth twisted
profile of the form
\[
\theta(i)\approx \theta(1)+\frac{2\pi W_\infty}{N}(i-1),
\]
and fluctuates around it until a rare excursion produces a branch crossing and allows the winding to change.

Figures~\ref{fig:destroy}--\ref{fig:destroy2} illustrate one explicit
escape route for $N=300$, starting from a smooth winding configuration
with $W=8$. In this simulation we impose a deterministic cyclic schedule
that repeatedly updates the interior edges
\[
(1,2),(2,3),\dots,(N-1,N),
\]
while the closing edge $(N,1)$ is never selected (see Definition~\ref{def:cyclic-sweep-delta}).
During such a time window the ring is effectively treated as an open
chain (i.e. with open boundary conditions), and the \texttt{ACCA} dynamics transports the increment mass toward the
un-updated boundary.

\begin{figure}[H]
    \centering
    \subfigure[]{\includegraphics[width=0.44\textwidth]{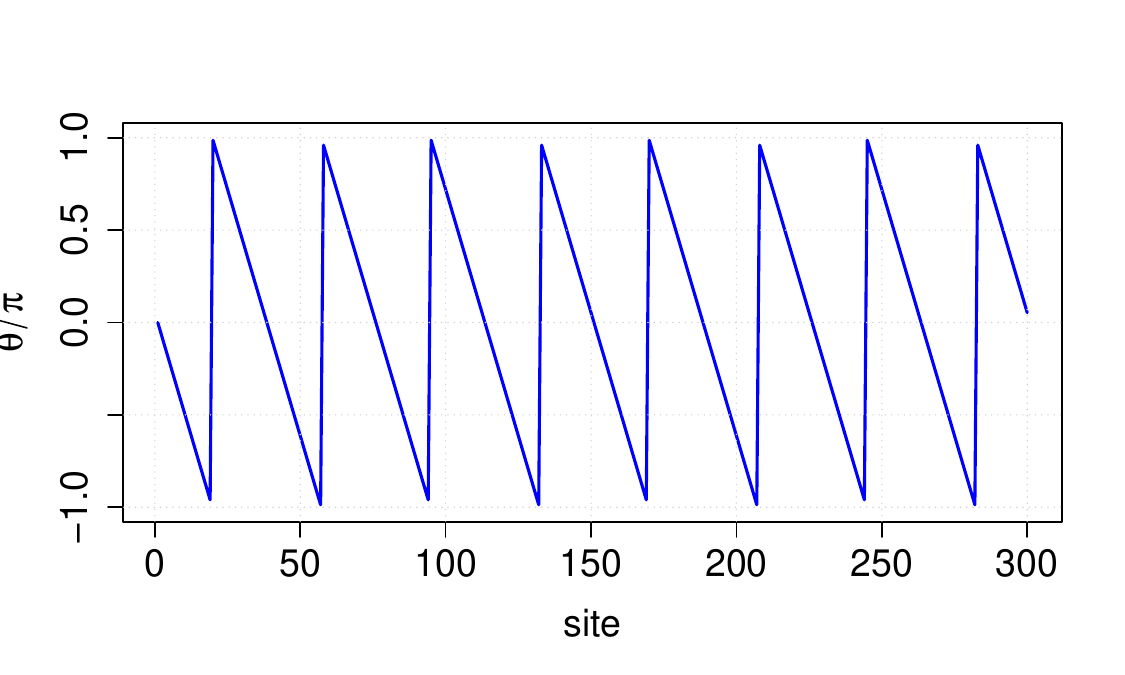}} 
    \subfigure[]{\includegraphics[width=0.44\textwidth]{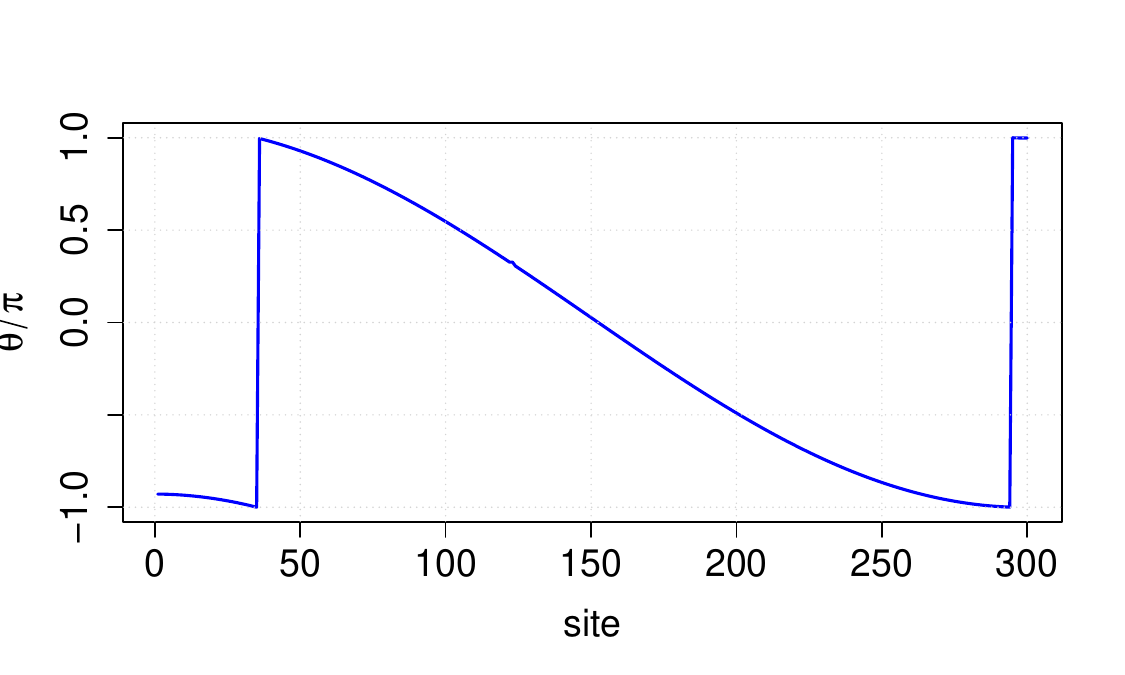}} 
    \subfigure[]{\includegraphics[width=0.44\textwidth]{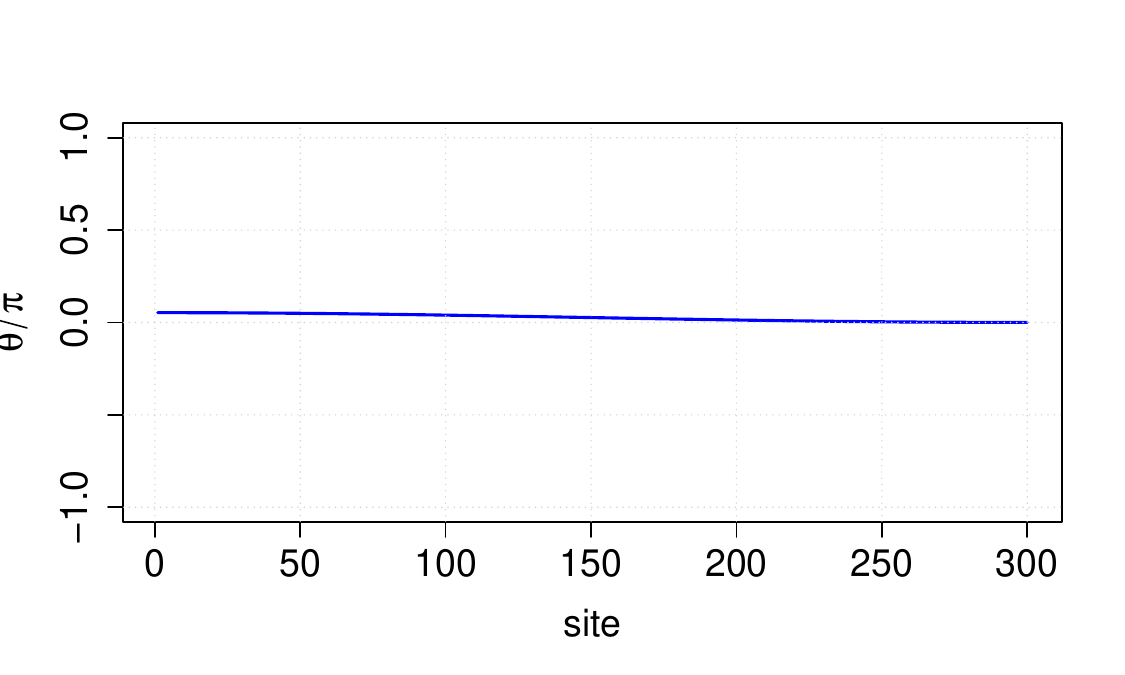}}
    \subfigure[]{\includegraphics[width=0.44\textwidth]{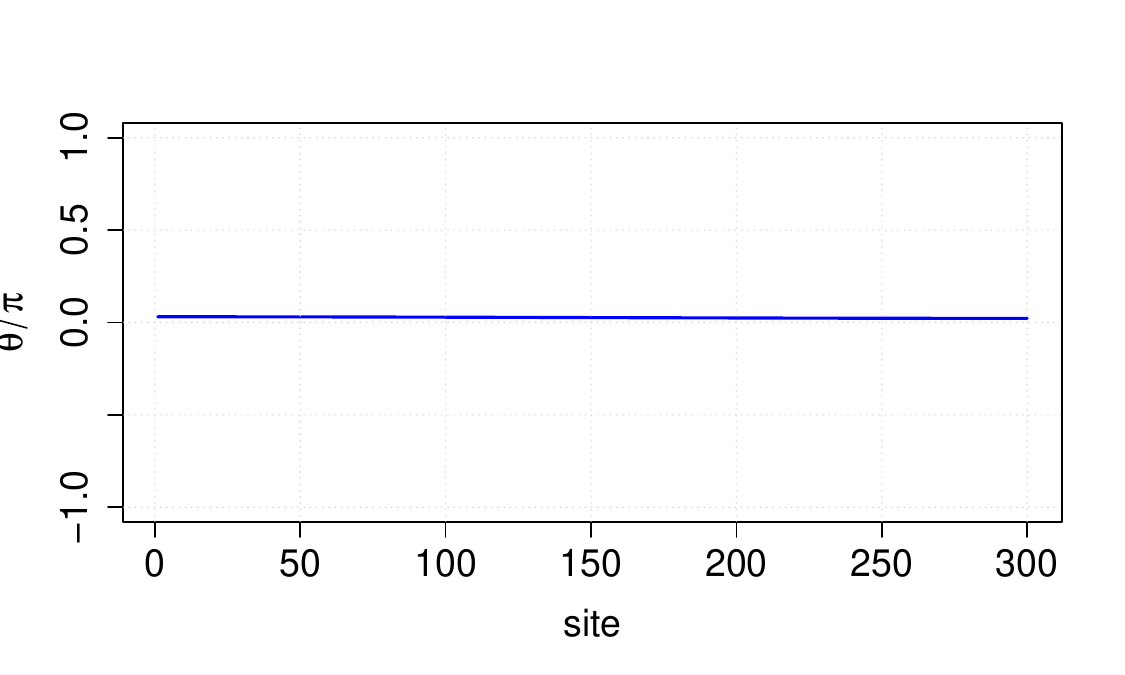}}
    \caption{
    Deterministic destruction of a winding configuration ($N=300$, $W=8$) under a cyclic schedule that updates only the interior edges $(1,2),(2,3),\dots,(N-1,N)$ and never the closing edge $(N,1)$. The $y$--axis shows $\theta_t(i)/\pi$.
    (a) $t=0$.
    (b) $t=5\times 10^6$.
    (c) $t=10^7$.
    (d) $t=2\times 10^7$ (consensus).
    }
    \label{fig:destroy}
\end{figure}

\begin{figure}[H]
    \centering
    \includegraphics[height=7cm]{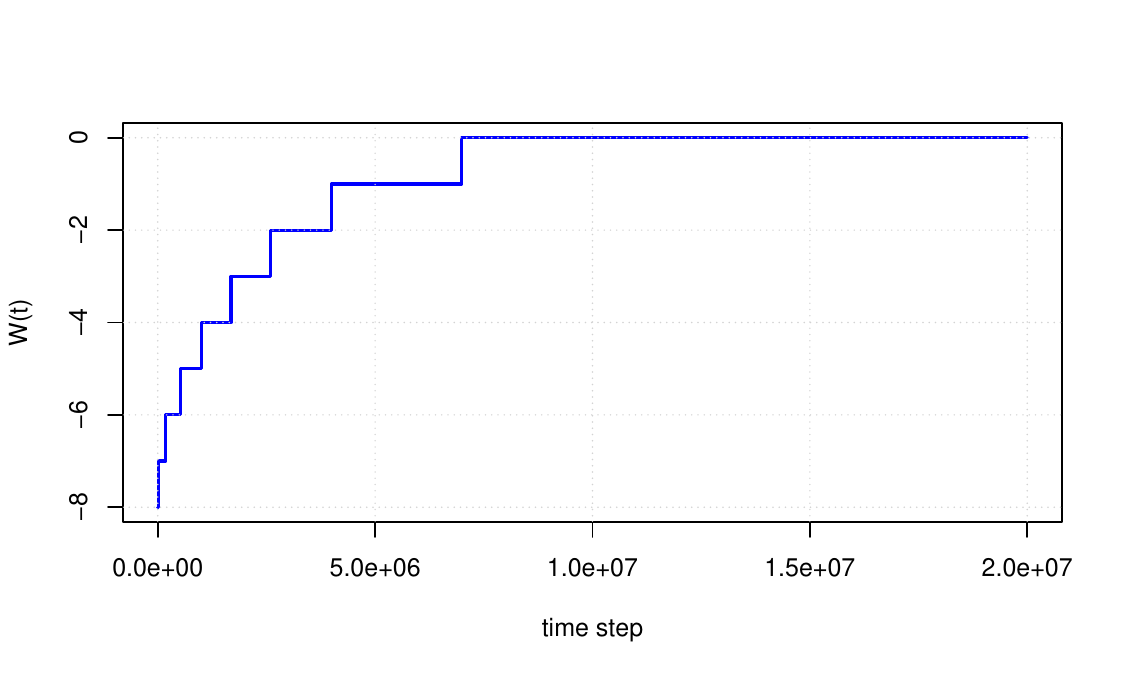}
    \caption{
    Evolution of the winding number for the simulation in Figure~\ref{fig:destroy}. The winding remains constant during the smoothing phase and decreases only when a branch crossing occurs.
    }
    \label{fig:destroy2}
\end{figure}
Panel~(a) of Figure~\ref{fig:destroy} shows the initial twisted profile, panels~(b)--(c) show progressive flattening, and panel~(d) shows collapse to consensus. The corresponding winding evolution is reported in Figure~\ref{fig:destroy2}: the winding stays constant throughout the smoothing phase and then drops at the first branch crossing; repeating the same mechanism eventually leads to winding $0$.

Let us explain heuristically the mechanism by which the dynamics escapes
from the winding manifold. Under
Assumption~\ref{asmp:no-antipodes}, that is for times $t<\tau$ before any
branch--crossing occurs, the update rule for the wrapped increments admits
a simple interpretation as a redistribution of a conserved quantity along
the edges of the ring. If edge $k$ is selected, the increment on that edge
is reset to zero while half of its value is transferred to each of its
neighbors:
\[
\delta_{t+1}(k)=0, \qquad
\delta_{t+1}(k\pm1)=\delta_t(k\pm1)+\tfrac12\,\delta_t(k),
\]
with all other increments unchanged. As long as $t<\tau$, the wrapping
operator introduces no $2\pi$ correction, so the total $W_t$ is conserved. The field $\delta_t$
therefore behaves as a conserved quantity that is transported along the
ring by the midpoint updates. This viewpoint clarifies the mechanism responsible for the eventual escape
from the winding configuration. Consider the deterministic
cyclic sweep in which the edges $(1,2),(2,3),\dots,(N\!-\!1,N)$ are updated
repeatedly while the closing edge $(N,1)$ is never selected. Each update
resets the increment on the chosen edge while redistributing its mass to
the neighboring increments. Since the total increment is conserved, the
repeated resets progressively push the increment mass toward the boundary
of the unique edge that is never updated. As a consequence, the increment
field becomes increasingly concentrated near $(N,1)$.

This mechanism is illustrated in
Figures~\ref{fig:destroy_theta}--\ref{fig:destroy_Spm}.
Figure~\ref{fig:destroy_theta} shows the evolution of the opinion
configuration $\theta(i)$ starting from a perfect winding profile.
While the global slope remains visible, the profile gradually flattens
away from the closing edge. The same evolution becomes more transparent in
terms of the increment field $\delta_t(i)$ shown in
Figure~\ref{fig:destroy_delta}. Because updated edges are repeatedly reset,
the increments away from $(N,1)$ remain small while the conserved total
increment accumulates near the closing edge, producing a growing spike in
$\delta_t(N)$. Finally,
Figure~\ref{fig:destroy_Spm} displays the associated corridor quantities
$S_-^{(t)}(k)=\delta_t(k-1)+\tfrac12\delta_t(k)$ and
$S_+^{(t)}(k)=\delta_t(k+1)+\tfrac12\delta_t(k)$.
As the concentration increases, these quantities approach the thresholds
$\pm\pi$. When one of them exits the interval $(-\pi,\pi)$, the no--branch
condition fails and a branch--crossing occurs. At that moment the winding
number can change, allowing the system to leave the winding
sector and continue its relaxation toward consensus.

% --------------------------- FIGURE 15: theta ---------------------------
\begin{figure}[H]
\centering
\subfigure[]{\scalebox{0.80}{\includegraphics[width=0.48\textwidth]{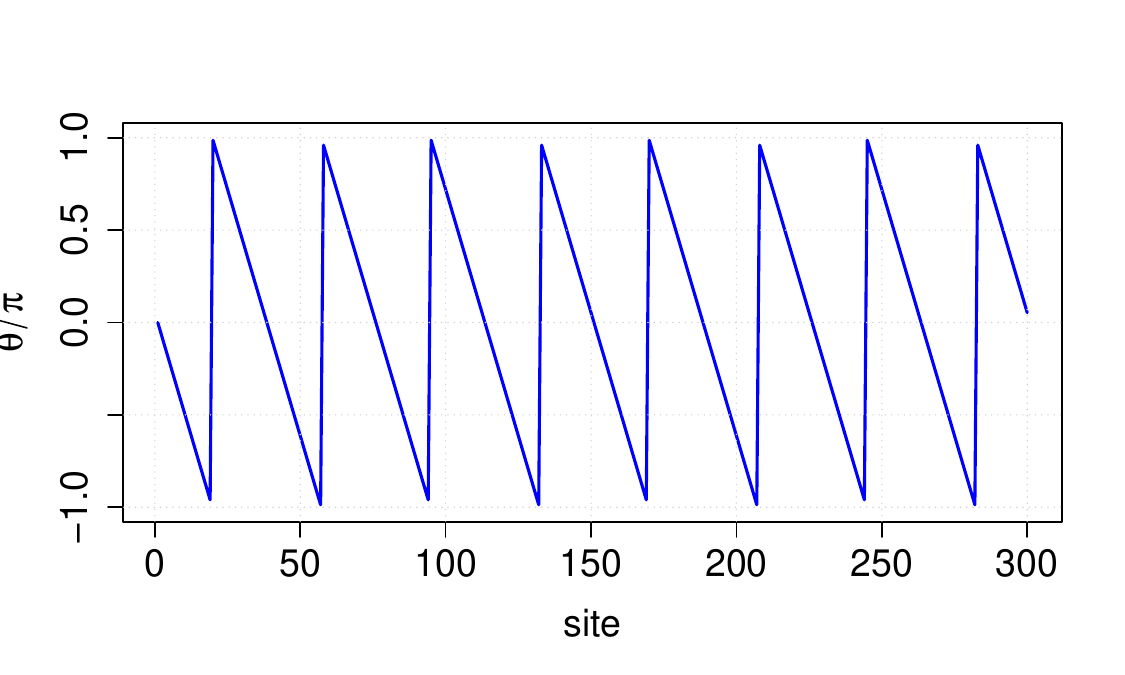}}\label{fig:destroy_theta_t0}}
\subfigure[]{\scalebox{0.80}{\includegraphics[width=0.48\textwidth]{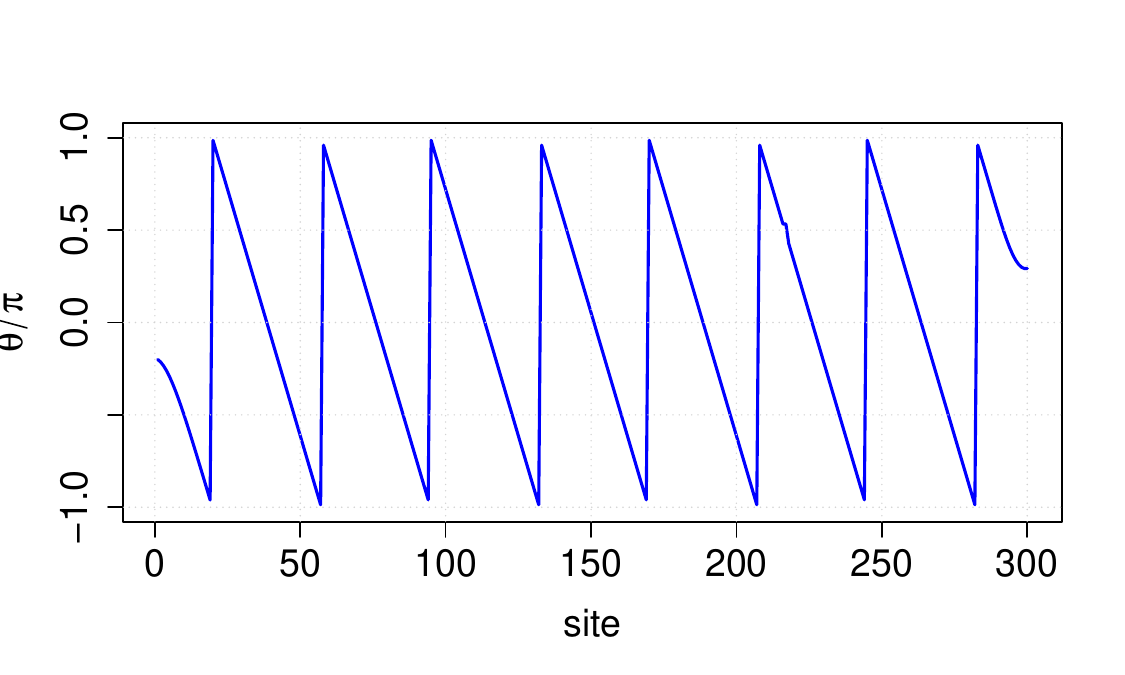}}\label{fig:destroy_theta_t1}}
\subfigure[]{\scalebox{0.80}{\includegraphics[width=0.48\textwidth]{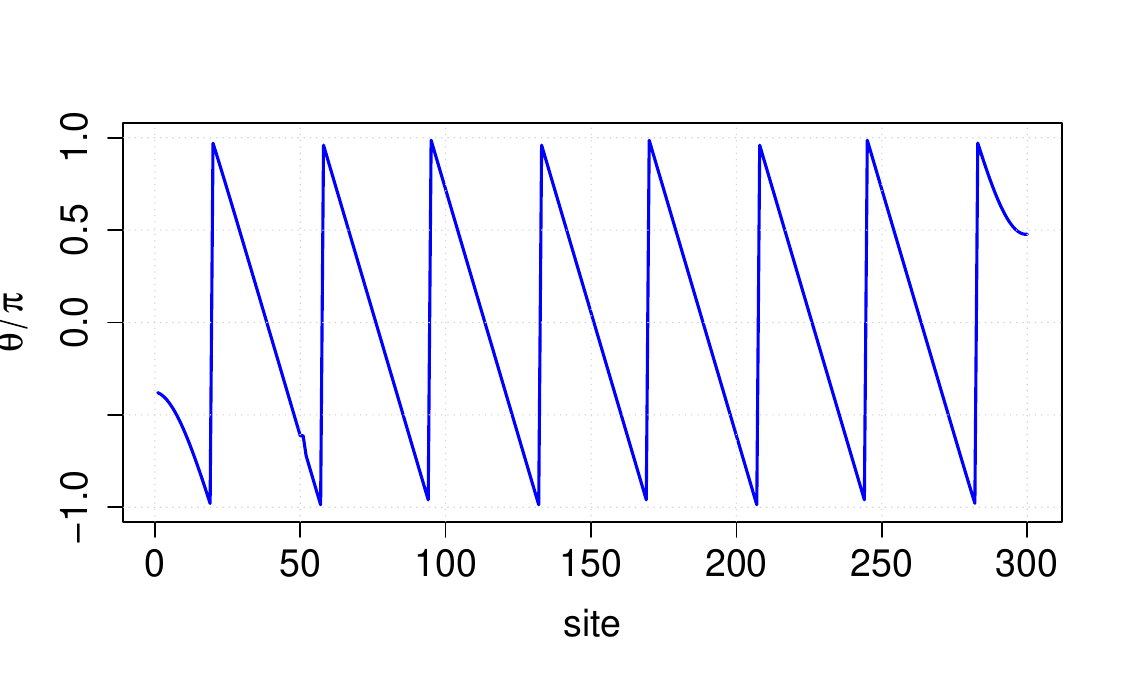}}\label{fig:destroy_theta_t3}}
\subfigure[]{\scalebox{0.80}{\includegraphics[width=0.48\textwidth]{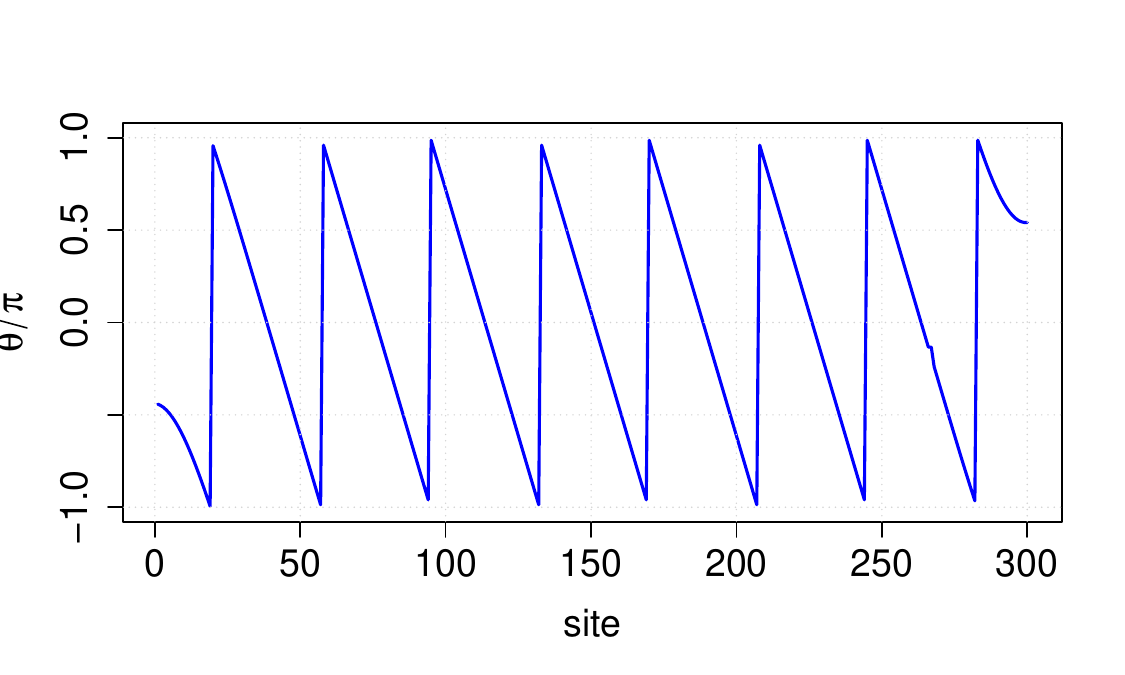}}\label{fig:destroy_theta_t4}}
\caption{Deterministic cyclic sweep on the ring with $N=300$ and prescribed winding $W_0=8$ over a horizon $T=20000$, where edges $(1,2),(2,3),\dots,(N\!-\!1,N)$ are updated cyclically and the closing edge $(N,1)$ is never selected. Shown are snapshots of $\theta(i)/\pi$ at $t_0=0$, $t_1=T/4$, $t_3=3T/4$, and $t_4=T$.}
\label{fig:destroy_theta}
\end{figure}

% --------------------------- FIGURE 16: delta ---------------------------
\begin{figure}[H]
\centering
\subfigure[]{\scalebox{0.80}{\includegraphics[width=0.48\textwidth]{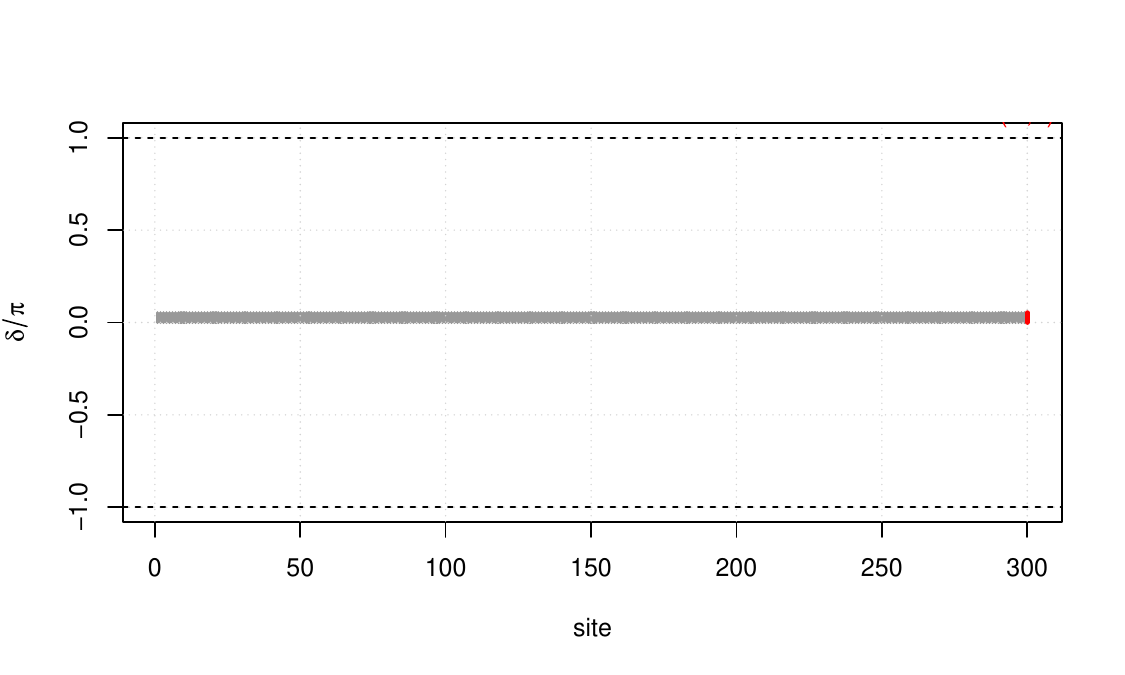}}\label{fig:destroy_delta_t0}}
\subfigure[]{\scalebox{0.80}{\includegraphics[width=0.48\textwidth]{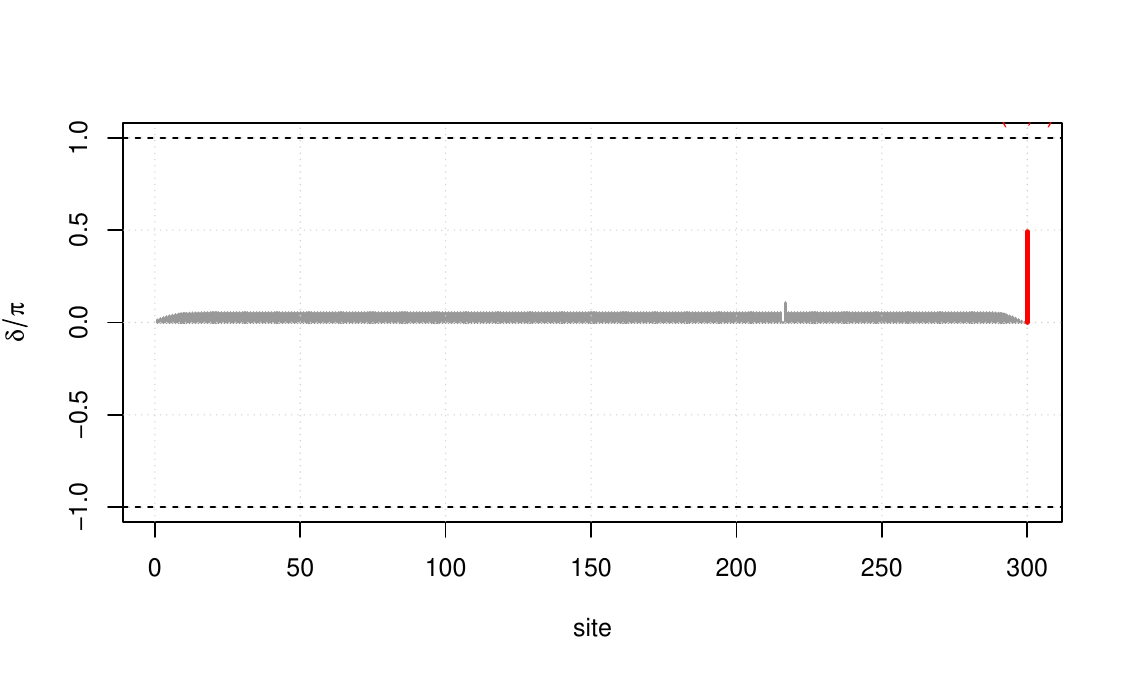}}\label{fig:destroy_delta_t1}}
\subfigure[]{\scalebox{0.80}{\includegraphics[width=0.48\textwidth]{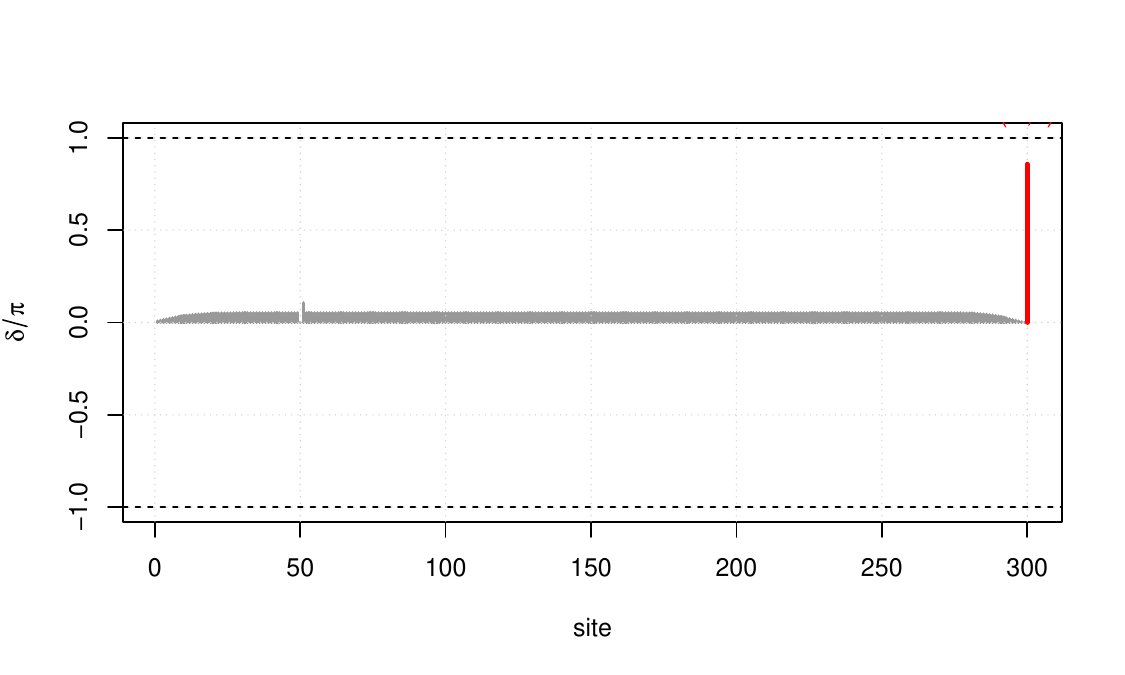}}\label{fig:destroy_delta_t3}}
\subfigure[]{\scalebox{0.80}{\includegraphics[width=0.48\textwidth]{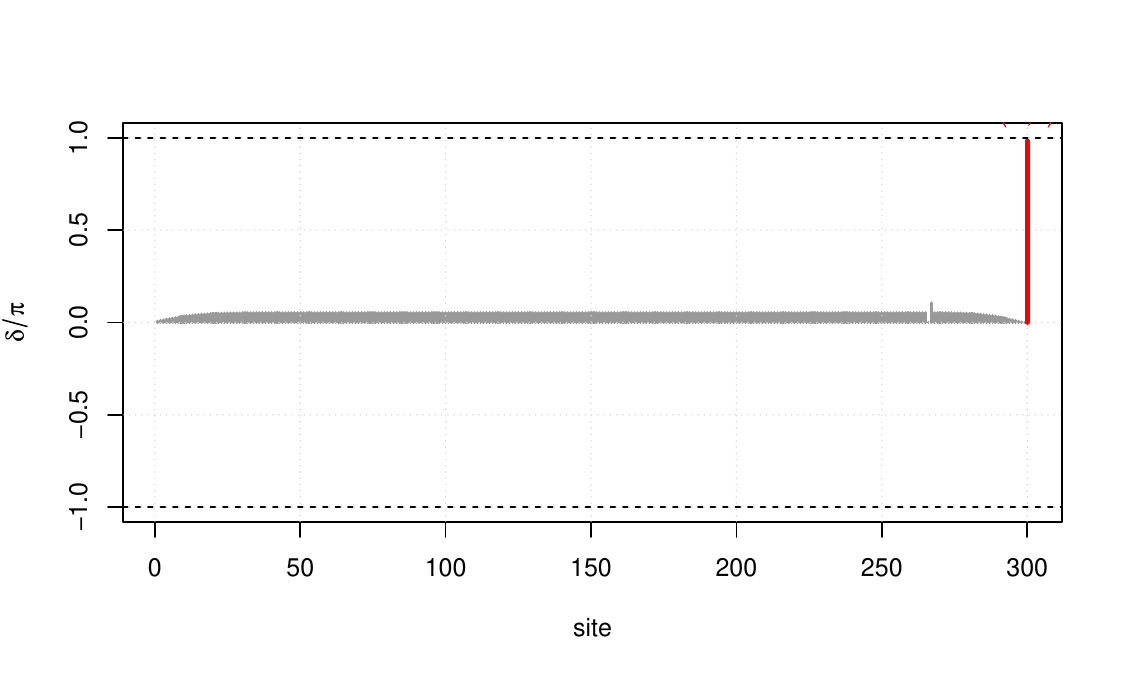}}\label{fig:destroy_delta_t4}}
\caption{Snapshots of the increment field $\delta_t(i)/\pi$ for the same deterministic sweep as in Figure~\ref{fig:destroy_theta}. Since updated edges are repeatedly reset while the total increment $\sum_{i=1}^N \delta_t(i)=2\pi W_0$ is conserved before any branch--crossing, the increment mass progressively concentrates near the un-updated closing edge $(N,1)$ (highlighted in red in the plots).}
\label{fig:destroy_delta}
\end{figure}

% --------------------------- FIGURE 17: S_pm ---------------------------
\begin{figure}[H]
\centering
\subfigure[]{\scalebox{0.95}{\includegraphics[width=0.48\textwidth]{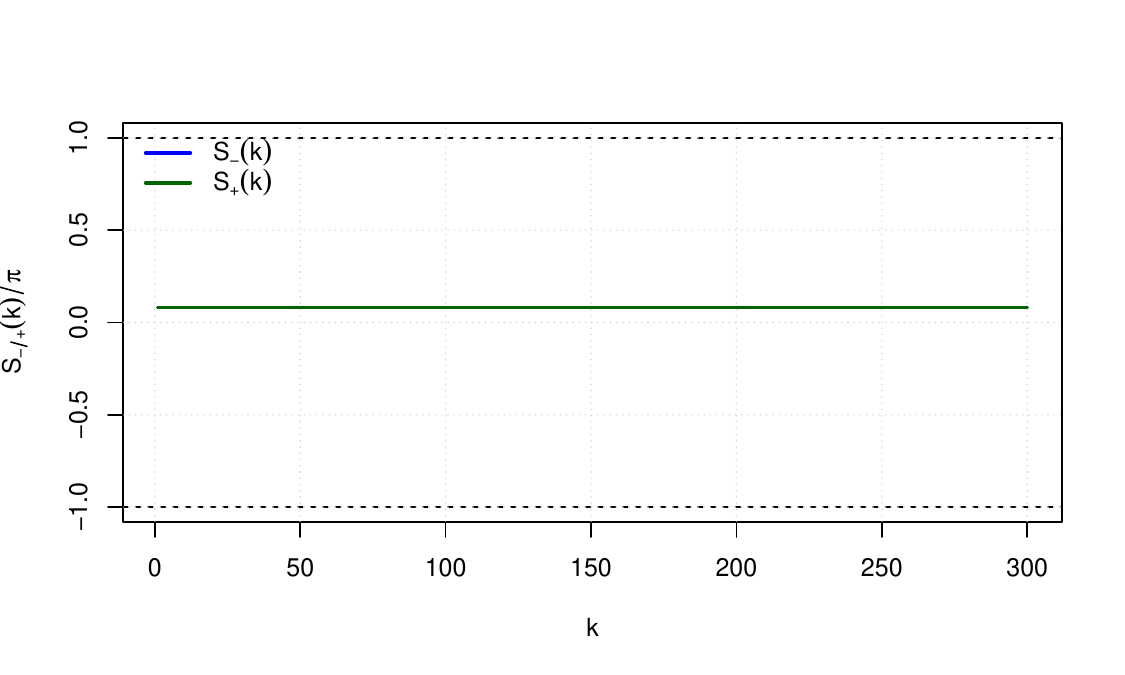}}\label{fig:destroy_Spm_t0}}
\subfigure[]{\scalebox{0.95}{\includegraphics[width=0.48\textwidth]{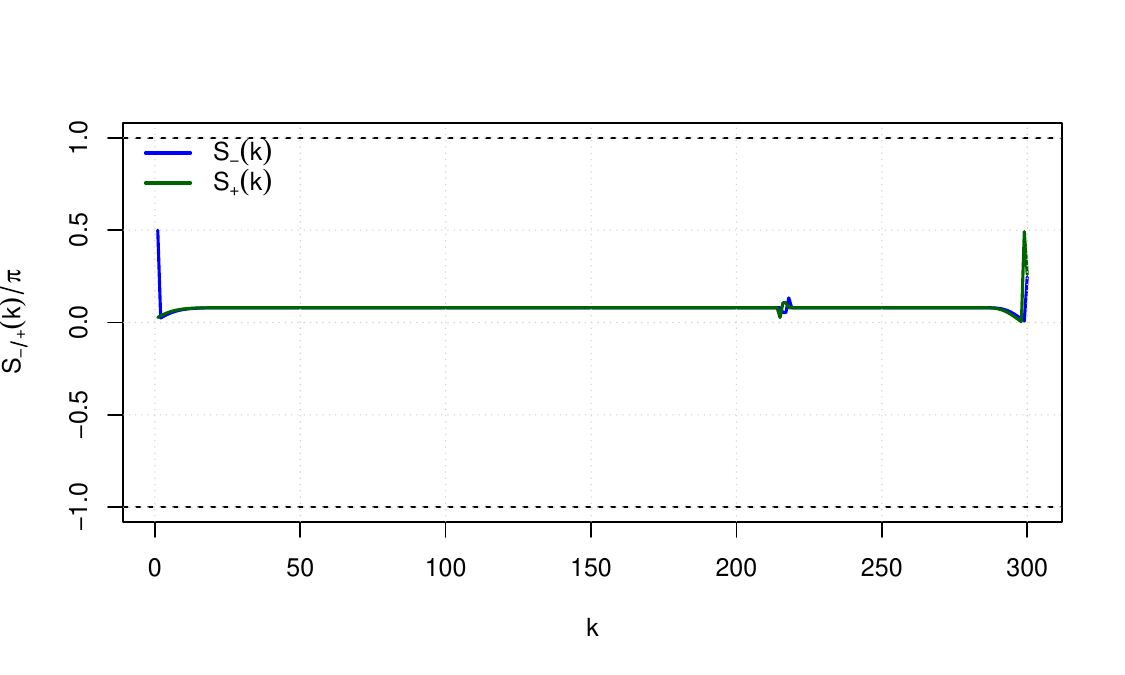}}\label{fig:destroy_Spm_t1}}
\subfigure[]{\scalebox{0.95}{\includegraphics[width=0.48\textwidth]{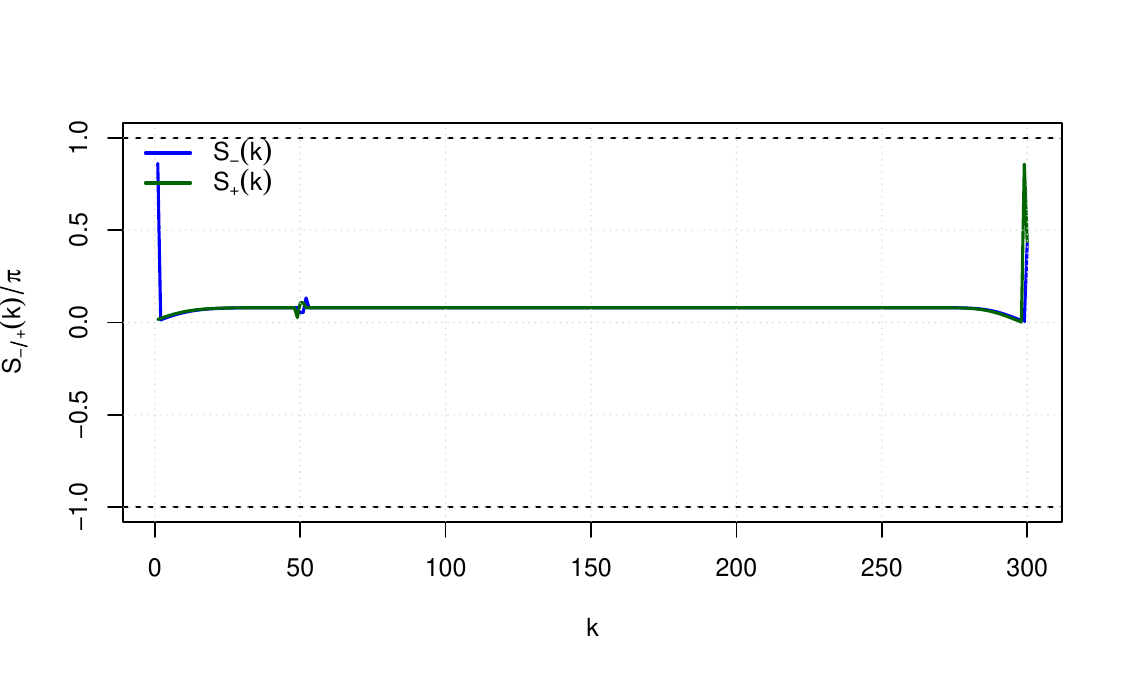}}\label{fig:destroy_Spm_t3}}
\subfigure[]{\scalebox{0.95}{\includegraphics[width=0.48\textwidth]{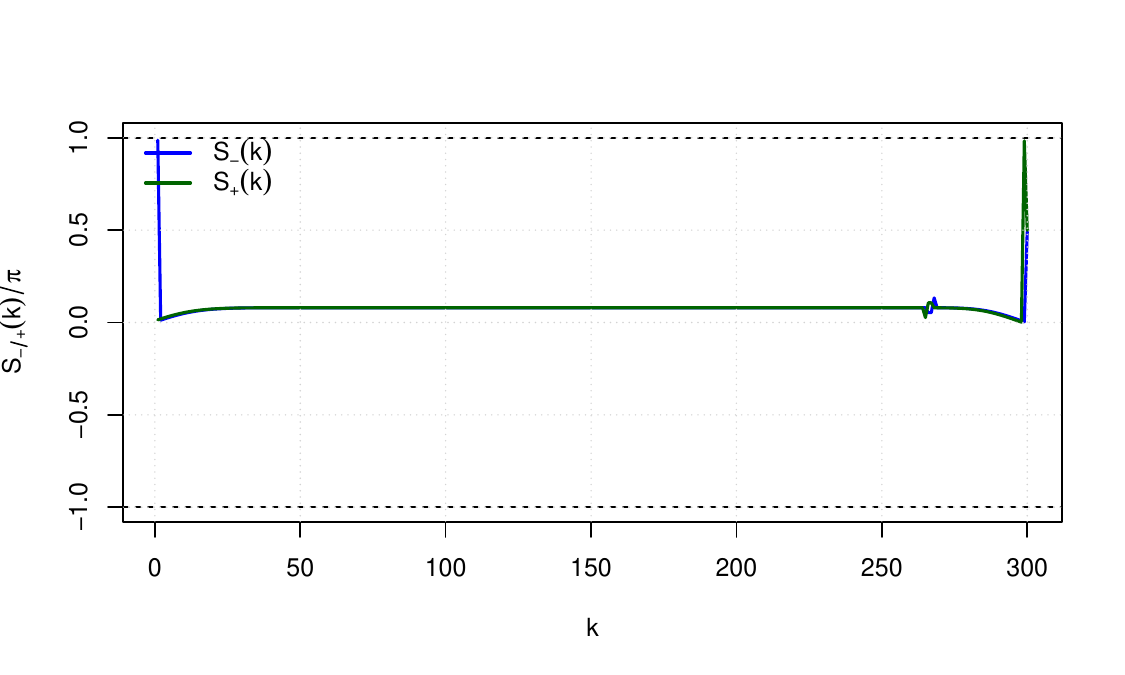}}\label{fig:destroy_Spm_t4}}
\caption{Snapshots of the corridor quantities $S_-^{(t)}(k)=\delta_t(k-1)+\tfrac12\delta_t(k)$ and $S_+^{(t)}(k)=\delta_t(k+1)+\tfrac12\delta_t(k)$ (shown rescaled by $\pi$) for the deterministic sweep in Figures~\ref{fig:destroy_theta}--\ref{fig:destroy_delta}. The dashed lines mark the thresholds $\pm 1$ (i.e.\ $\pm\pi$ before rescaling), so approaching these bounds indicates proximity to a branch--crossing event.}
\label{fig:destroy_Spm}
\end{figure}

To interpret the escape mechanism mathematically, consider a time window
during which the closing edge $\{N,1\}$ is never selected.
In lifted coordinates $\eta(i)\in\mathbb{R}$ the midpoint rule becomes
linear averaging, so over such a window the evolution coincides with
midpoint dynamics on the open path
(see Lemma~\ref{lem:lift-midpoint}).
We therefore analyze an explicit deterministic update pattern on the open path.

\begin{definition}[One cyclic sweep on the increment field]
\label{def:cyclic-sweep-delta}
Work under Assumption~\ref{asmp:no-antipodes}.
Let $\delta\in\R^N$ denote the wrapped increment field on the ring.
A \emph{cyclic sweep avoiding the closing edge} $(N,1)$ is the deterministic sequence of
updates
\[
(1,2),(2,3),\dots,(N-1,N),
\]
i.e.\ we apply the above increment-update rule successively for
$k=1,2,\dots,N-1$ and \emph{never} update the closing edge $(N,1)$.
Given an initial increment field $\delta^{[0]}\in\R^N$, we denote by
$\delta^{[1]}$ the increment field obtained after completing this sweep.
More generally, $\delta^{[m]}$ denotes the increment field after $m$ full sweeps.
\end{definition}
Under Assumption~\ref{asmp:no-antipodes}, i.e.\ for times $t<\tau$
before any branch--crossing occurs, wrapping introduces no $2\pi$
correction and the increment dynamics becomes purely linear.
Updating edge $k$ resets the corresponding increment to zero and
redistributes half of its value to each neighboring increment:
\[
\delta_{t+1}(k)=0,
\qquad
\delta_{t+1}(k\pm1)=\delta_t(k\pm1)+\tfrac12\,\delta_t(k),
\]
with all other increments unchanged.
In particular, the total increment
\[
\sum_{i=1}^N\delta_t(i)=2\pi W
\]
is conserved.
Thus the field $\delta_t$ may be viewed as a conserved \emph{increment mass}
transported along the ring by the midpoint updates.

The cyclic sweep of Definition~\ref{def:cyclic-sweep-delta} provides an explicit
instance of this transport mechanism.
Since the closing edge $(N,1)$ is never updated during the sweep,
increments on edges $1,\dots,N-1$ are repeatedly reset and their mass is
redistributed to neighboring edges.
Because the total increment is conserved, this redistribution progressively
pushes the mass toward the boundary of the unique edge that is never reset,
namely $(N,1)$.
The following lemma quantifies this accumulation after one sweep.

\begin{lemma}[Increment accumulation under one cyclic sweep]
\label{lem:cyclic-sweep-accumulation}
Assume Assumption~\ref{asmp:no-antipodes} holds.
Let us consider the \emph{cyclic sweep} defined in Definition~\ref{def:cyclic-sweep-delta}. Then the closing increment satisfies the identity:
\begin{equation}
\label{eq:deltaN-exact-correct}
\delta^{[1]}(N)
=
\delta^{[0]}(N)
+\Big(\frac12+2^{-(N-1)}\Big)\delta^{[0]}(1)
+\sum_{j=2}^{N-1}2^{-(N-j)}\,\delta^{[0]}(j).
\end{equation}
In particular, one has the lower bound
\begin{equation}
\label{eq:deltaN-lower-correct}
\delta^{[1]}(N)
\ge
\delta^{[0]}(N)+2^{-(N-2)}\sum_{j=1}^{N-1}\delta^{[0]}(j).
\end{equation}
\end{lemma}

\begin{proof}
Work under Assumption~\ref{asmp:no-antipodes}, so that during the sweep the
increment dynamics is linear and no wrapping corrections occur.
Let $\delta^{(k)}$ denote the increment field after the $k$-th update within
the sweep, where the updates are performed in the order
$(1,2),(2,3),\dots,(N-1,N)$, and $\delta^{(0)}=\delta^{[0]}$.
After the sweep we have $\delta^{[1]}=\delta^{(N-1)}$.

When edge $k$ is updated, the increment $\delta(k)$ is reset to $0$ and its
value is redistributed equally to the neighboring edges $k-1$ and $k+1$.
Since the sweep proceeds in increasing order of $k$, any increment created on
an edge to the right of $k$ may still propagate further to the right during
the same sweep, whereas increments created to the left cannot propagate
further because the corresponding edges have already been updated.

Consequently the sweep transports mass only toward increasing indices.
In particular, the value of $\delta^{[1]}(N)$ consists of the initial mass
$\delta^{[0]}(N)$ together with all contributions that are transported to
edge $N$ during the sequence of updates.

Consider an initial increment $\delta^{[0]}(j)$ with $2\le j\le N-1$.
When edge $j$ is updated, the rule places a contribution
$\tfrac12\delta^{[0]}(j)$ on edge $j+1$.
At the subsequent update of edge $j+1$, half of this quantity is transmitted
further to edge $j+2$, and this process continues deterministically along the
remaining updates of the sweep.
After $(N-j)$ such transfers the mass arriving at edge $N$ equals
$2^{-(N-j)}\delta^{[0]}(j)$.

The case $j=1$ is slightly different because edge $1$ has two neighbors.
The update of edge $1$ sends $\tfrac12\delta^{[0]}(1)$ directly to edge $N$,
while the other half is placed on edge $2$ and then propagates to the right
through the remaining updates of the sweep.
Starting from edge $2$, reaching edge $N$ requires $(N-2)$ further transfers,
so this second contribution equals $2^{-(N-1)}\delta^{[0]}(1)$.

Since the closing edge $(N,1)$ is never updated during the sweep,
any mass that reaches edge $N$ remains there until the end.
Summing the initial mass $\delta^{[0]}(N)$ with all transported contributions
yields~\eqref{eq:deltaN-exact-correct}.
The lower bound~\eqref{eq:deltaN-lower-correct} follows because the smallest
coefficient multiplying $\delta^{[0]}(j)$ for $j\in\{1,\dots,N-1\}$ is
$2^{-(N-2)}$.
\end{proof}

\begin{theorem}[Monotone accumulation at the closing edge]
\label{thm:monotone-accumulation}
Work under Assumption~\ref{asmp:no-antipodes} (no branch--crossing), and consider the \emph{cyclic sweep} defined in Definition~\ref{def:cyclic-sweep-delta}. Let $\delta^{[m]}$ be the increment
field after $m$ full sweeps. Assume that the initial increments are nonnegative,
\begin{equation}\label{eq:nonneg-init}
\delta^{[0]}(i)\ge 0\qquad \text{for all } i=1,\dots,N.
\end{equation}
Then $\delta^{[m]}(i)\ge 0$ for all $m\ge 0$ and all $i$, and the sequence
$m\mapsto \delta^{[m]}(N)$ is nondecreasing. Moreover, with
\[
c:=2^{-(N-2)},
\]
one has for every $m\ge 0$,
\begin{equation}\label{eq:monotone-gap}
m(\theta)-\delta^{[m]}(N)
\;\le\;
(1-c)^m\big(m(\theta)-\delta^{[0]}(N)\big),
\end{equation}
and equivalently,
\begin{equation}\label{eq:monotone-lower}
\delta^{[m]}(N)
\;\ge\;
m(\theta)-(1-c)^m\big(m(\theta)-\delta^{[0]}(N)\big),
\end{equation}
and $m(\theta)$ was defined in \eqref{e:w}.
\end{theorem}

\begin{proof}
Under Assumption~\ref{asmp:no-antipodes}, updating an edge $k$ replaces $\delta(k)$ by $0$
and adds $\tfrac12\delta(k)$ to each neighbor $\delta(k\pm1)$, leaving all other
increments unchanged. If $\delta(i)\ge 0$ for all $i$ before the update, then after the
update we still have nonnegativity: one coordinate is set to $0$ and the two neighbors
increase by a nonnegative amount. Therefore nonnegativity is preserved by each single
update, and \eqref{eq:nonneg-init} implies $\delta^{[m]}(i)\ge 0$ for all $m$ and $i$.

Since no wrapping corrections occur, the total increment is conserved by each update,
hence by each sweep:
\[
\sum_{i=1}^N \delta^{[m]}(i)=m(\theta)
\qquad\text{for all } m\ge 0.
\]

Apply Lemma~\ref{lem:cyclic-sweep-accumulation}  to the increment field $\delta^{[m]}$
as the input for the next sweep. The lemma gives the exact representation
\[
\delta^{[m+1]}(N)
=
\delta^{[m]}(N)
+\sum_{j=1}^{N-1} a_j\,\delta^{[m]}(j),
\]
where the coefficients $a_j$ depend only on $N$ and satisfy $a_j\ge c=2^{-(N-2)}$ for
all $j=1,\dots,N-1$. Since $\delta^{[m]}(j)\ge 0$, the sum on the right-hand side is
nonnegative, and thus $\delta^{[m+1]}(N)\ge \delta^{[m]}(N)$. This proves that
$m\mapsto \delta^{[m]}(N)$ is nondecreasing.

Using again $a_j\ge c$ yields
\[
\delta^{[m+1]}(N)
\ge
\delta^{[m]}(N)
+c\sum_{j=1}^{N-1}\delta^{[m]}(j).
\]
By conservation of the total increment,
\[
\sum_{j=1}^{N-1}\delta^{[m]}(j)=m(\theta)-\delta^{[m]}(N),
\]
so we obtain
\[
\delta^{[m+1]}(N)
\ge
\delta^{[m]}(N)+c\big(m(\theta)-\delta^{[m]}(N)\big).
\]
Rearranging gives the one-step contraction
\[
m(\theta)-\delta^{[m+1]}(N)\le (1-c)\big(m(\theta)-\delta^{[m]}(N)\big).
\]
Iterating this inequality over $m$ sweeps yields \eqref{eq:monotone-gap}, and
\eqref{eq:monotone-lower} follows immediately by rearranging.
Finally, since $\delta^{[m]}(N)$ is nondecreasing and bounded above by $m(\theta)$,
it converges, and \eqref{eq:monotone-gap} forces the limit to be $m(\theta)$.
\end{proof}
Theorem~\ref{thm:monotone-accumulation} formalizes the heuristic mass--transport
picture described above. Starting from a twisted configuration, the increments
are initially uniform and positive, $\delta^{[0]}(i)=\beta>0$, and therefore
remain nonnegative throughout the evolution as long as no branch--crossing
occurs. In this regime each sweep redistributes the increments deterministically
along the chain while preserving the total mass $m(\theta)$. Since the closing
edge $(N,1)$ is never updated, it acts as an absorbing boundary for the
transport process: increments arriving at edge $N$ are never redistributed
again during the sweep. As a consequence, successive sweeps progressively
concentrate the increment mass at the closing edge. Theorem
\ref{thm:monotone-accumulation} shows that this accumulation is monotone and
that the deficit $m(\theta)-\delta^{[m]}(N)$ decays geometrically in the number
of sweeps. In particular, repeated sweeps drive the increment at the closing
edge toward the total increment $m(\theta)=2\pi W(\theta)$, thereby creating a
large local gradient near $(N,1)$ and providing a deterministic mechanism that
can eventually trigger a branch--crossing.

A quantitative analysis of the exit time from a winding sector
would require controlling rare update sequences that produce
branch–crossings. This lies beyond the scope of the present work
and will be addressed elsewhere.
%
%
%----------------------------------------
% SECTION 5: DISCUSSION AND FUTURE DEVELOPMENTS
%----------------------------------------
%
%
%
%

\section{Discussion and future developments}
\label{sec:future}

The analysis presented here clarifies how local averaging on a circular opinion
space generates trapping winding configurations characterized by a non-zero winding
number. In these states the dynamics remains locally contractive and a
Lyapunov functional decreases monotonically, yet the global topology of the circle
prevents convergence to consensus. The system therefore relaxes toward a
twisted profile that preserves the winding number and fluctuates near this
configuration for long time intervals before a rare branch--crossing event
changes the winding sector. The constructive mechanism described in
Section~\ref{sec:escape-heur} shows explicitly how such winding states can be
destroyed through the accumulation of increment mass near an edge, eventually
triggering a branch crossing and allowing the system to unwind. These results suggest several directions for further investigation.

A first direction concerns the quantitative characterization of the
\emph{trapping winding regime}. While the winding number already provides a global
topological invariant, it is useful to complement it with an observable that
measures the proximity of the configuration to a twisted profile. One natural
candidate is the \emph{circular correlation}
\[
r(t)=\frac1N\sum_{j=1}^N e^{i\theta_t(j)},
\]
which plays a role analogous to the Kuramoto order parameter (\cite{Acebron2005}). In a consensus
state one has $r(t)\approx 1$, while in a twisted configuration the complex
phases cancel and $r(t)$ remains small. Monitoring the joint evolution of
$r(t)$ and the winding number therefore provides a practical diagnostic for
detecting winding sectors and identifying the moments when
branch--crossing events occur.

A second natural extension concerns the role of \emph{noise}. Introducing small
stochastic perturbations in the update rule, for instance by adding a random
angular displacement after each averaging step, would model communication
errors or spontaneous opinion changes. Such perturbations are expected to
induce rare transitions between winding sectors, leading to a trapping
dynamics similar to that observed in stochastic Kuramoto-type systems.
Understanding the statistics of these transitions and their dependence on the
system size could reveal large-deviation principles governing the escape time
from winding manifolds.

Another direction is to explore \emph{block updates} in which several adjacent
agents simultaneously average their opinions. This modification would provide
a natural coarse-grained analogue of the midpoint rule and could significantly
accelerate the redistribution of increment mass along the ring. Studying how
the escape mechanism identified in Section~\ref{sec:escape-heur} is modified
under such block interactions may shed light on the robustness of winding under different microscopic update schemes.

Finally, it would be interesting to introduce \emph{asymmetric noise} or biased
perturbations in the update rule. In oscillator models such as the Kuramoto
system, asymmetries and noise can produce persistent oscillatory regimes,
sometimes referred to as \emph{flip--flop} states, in which the collective
phase repeatedly switches direction. Investigating whether analogous regimes
appear in circular opinion dynamics could reveal new forms of collective
behavior arising from the interplay between topology, stochasticity, and local
consensus interactions.

More broadly, the present work highlights the delicate balance between local
consensus forces and global topological constraints. Extending these ideas to
heterogeneous interaction graphs, higher-dimensional manifolds, or continuum
limits may reveal new connections between opinion dynamics, synchronization
theory, and the topology of interacting systems. Understanding how topology,
stochasticity, and heterogeneity combine to shape collective behavior on
nontrivial manifolds remains a challenging and promising direction for future
research.

%
%
%----------------------------------------
% APPENDIX A: PROOF OF LEMMA 4.5
%----------------------------------------
%
%

\appendix
\section{Proof of Lemma~\ref{lem:first-crossing-probability}}
\label{a:first}
\begin{proof}
Set
\[
X:=\delta_0(k-1),\qquad Y:=\delta_0(k),\qquad Z:=\delta_0(k+1).
\]
A midpoint update at edge $k$ sets the updated increment to $0$ and adds half of the former increment on $(k,k+1)$ to each neighboring increment. Thus a branch-crossing at time $0$ occurs if and only if at least one of the two neighbor quantities leaves the principal interval:
\begin{equation}\label{eq:crossing-event}
\{\text{branch-crossing at time }0\}
=
\Big\{\bigl|X+\tfrac12 Y\bigr|>\pi\Big\}\ \cup\ \Big\{\bigl|Z+\tfrac12 Y\bigr|>\pi\Big\}.
\end{equation}

We begin by computing the law of a wrapped difference. Let $\Theta_1,\Theta_2$ be i.i.d.\ $\mathrm{Unif}([-\pi,\pi))$ and set
\[
D:=\Theta_2-\Theta_1.
\]
The density of $D$ is obtained by the convolution formula for differences:
\[
f_D(w)
=
\int_{\mathbb R} f_{\Theta_1}(x)\,f_{\Theta_2}(x+w)\,dx.
\]
Since $f_{\Theta_i}(x)=\frac{1}{2\pi}\mathbf 1_{[-\pi,\pi)}(x)$, this becomes
\[
f_D(w)=\frac{1}{4\pi^2}\int_{\mathbb R}
\mathbf 1_{[-\pi,\pi)}(x)\,\mathbf 1_{[-\pi,\pi)}(x+w)\,dx
=
\frac{1}{4\pi^2}\,\lambda\!\left([-\pi,\pi)\cap([-\pi,\pi)-w)\right),
\]
where $\lambda$ denotes Lebesgue measure. The intersection is empty if $|w|\ge 2\pi$, while if $|w|<2\pi$ it is an interval of length $2\pi-|w|$. Hence
\begin{equation}\label{eq:triangular}
f_D(w)=\frac{2\pi-|w|}{4\pi^2}\,\mathbf 1_{\{|w|<2\pi\}}.
\end{equation}

Now set $Y:=\wrap(D)\in[-\pi,\pi)$. For any Borel set $A\subset[-\pi,\pi)$, the event $\{Y\in A\}$ is the disjoint union over $m\in\mathbb Z$ of the events $\{D\in A+2\pi m\}$, so
\[
\mathbb P(Y\in A)=\sum_{m\in\mathbb Z}\mathbb P(D\in A+2\pi m)
=\sum_{m\in\mathbb Z}\int_{A+2\pi m} f_D(u)\,du.
\]
Changing variables $u=y+2\pi m$ in each integral gives
\[
\mathbb P(Y\in A)
=\int_A\left(\sum_{m\in\mathbb Z} f_D(y+2\pi m)\right)dy,
\]
so the density of $Y$ is
\begin{equation}\label{eq:wrapsum}
f_Y(y)=\sum_{m\in\mathbb Z} f_D(y+2\pi m),\qquad y\in[-\pi,\pi).
\end{equation}
Because $f_D$ is supported on $(-2\pi,2\pi)$, only the terms $m\in\{-1,0,1\}$ can contribute. If $y\in[0,\pi)$ then $f_D(y+2\pi)=0$ and, using \eqref{eq:triangular},
\[
f_Y(y)=f_D(y)+f_D(y-2\pi)
=\frac{2\pi-y}{4\pi^2}+\frac{2\pi-|y-2\pi|}{4\pi^2}
=\frac{2\pi-y}{4\pi^2}+\frac{y}{4\pi^2}
=\frac{1}{2\pi}.
\]
If $y\in[-\pi,0)$ then $f_D(y-2\pi)=0$ and similarly
\[
f_Y(y)=f_D(y)+f_D(y+2\pi)=\frac{1}{2\pi}.
\]
Therefore
\begin{equation}\label{eq:Yuniform}
Y\sim \mathrm{Unif}([-\pi,\pi)).
\end{equation}

We now identify the conditional law of $X$ and $Z$ given $Y$. Write
\[
\phi:=\theta_0(k),\qquad \psi:=\theta_0(k+1),\qquad U:=\theta_0(k-1),\qquad V:=\theta_0(k+2).
\]
Then
\[
X=\wrap(\phi-U),\qquad Y=\wrap(\psi-\phi),\qquad Z=\wrap(V-\psi).
\]
Because the initial angles are i.i.d., the variables $U$ and $V$ are independent $\mathrm{Unif}([-\pi,\pi))$ and are independent of $(\phi,\psi)$. Fix $(\phi,\psi)$. Since $U$ is uniform and independent of $\phi$, the difference $\phi-U$ is uniform modulo $2\pi$; applying $\wrap$ is exactly the projection modulo $2\pi$ onto $[-\pi,\pi)$, hence
\begin{equation}\label{eq:Xgivenphipsi}
X\mid(\phi,\psi)\sim \mathrm{Unif}([-\pi,\pi)).
\end{equation}
Similarly,
\begin{equation}\label{eq:Zgivenphipsi}
Z\mid(\phi,\psi)\sim \mathrm{Unif}([-\pi,\pi)).
\end{equation}
Moreover, given $(\phi,\psi)$ the random variables $X$ and $Z$ depend on $U$ and $V$ separately, so
\begin{equation}\label{eq:condindep}
X\ \perp\!\!\!\perp\ Z\mid(\phi,\psi).
\end{equation}
Since $Y$ is a measurable function of $(\phi,\psi)$, conditioning further on $Y=y$ preserves both the uniform conditional marginals and conditional independence:
\begin{equation}\label{eq:cond-on-Y}
X\mid Y=y\sim \mathrm{Unif}([-\pi,\pi)),\quad
Z\mid Y=y\sim \mathrm{Unif}([-\pi,\pi)),\quad
X\ \perp\!\!\!\perp\ Z\mid Y=y.
\end{equation}

Fix $y\in[-\pi,\pi)$. Using \eqref{eq:cond-on-Y}, we compute
\[
\mathbb P\!\left(\bigl|X+\tfrac12 y\bigr|\le\pi\,\middle|\,Y=y\right)
=
\mathbb P\!\left(-\pi-\tfrac12 y\le X\le \pi-\tfrac12 y\,\middle|\,Y=y\right).
\]
Since $X\mid Y=y$ is uniform on $[-\pi,\pi)$, this equals the normalized overlap length of the interval
$[-\pi-\tfrac12 y,\ \pi-\tfrac12 y]$ with $[-\pi,\pi)$. Because $|y|\le\pi$, the shift magnitude is at most $\pi/2$, and the overlap length is exactly $2\pi-\frac{|y|}{2}$, hence
\begin{equation}\label{eq:one-overlap}
\mathbb P\!\left(\bigl|X+\tfrac12 y\bigr|\le\pi\,\middle|\,Y=y\right)
=
1-\frac{|y|}{4\pi}.
\end{equation}
The same holds with $Z$ in place of $X$. By conditional independence in \eqref{eq:cond-on-Y},
\[
\mathbb P(\text{no crossing}\mid Y=y)
=
\left(1-\frac{|y|}{4\pi}\right)^2.
\]
Taking expectation with respect to $Y\sim\mathrm{Unif}([-\pi,\pi))$ from \eqref{eq:Yuniform},
\[
\mathbb P(\text{no crossing})
=
\frac{1}{2\pi}\int_{-\pi}^{\pi}\left(1-\frac{|y|}{4\pi}\right)^2\,dy
=
\frac{1}{\pi}\int_{0}^{\pi}\left(1-\frac{y}{4\pi}\right)^2\,dy.
\]
Expanding and integrating,
\[
\frac{1}{\pi}\int_{0}^{\pi}\left(1-\frac{y}{2\pi}+\frac{y^2}{16\pi^2}\right)dy
=
\frac{1}{\pi}\left(\pi-\frac{\pi}{4}+\frac{\pi}{48}\right)
=
\frac{37}{48}.
\]
By \eqref{eq:crossing-event}, $\mathbb P(\text{branch-crossing})=1-\mathbb P(\text{no crossing})=\frac{11}{48}$.
\end{proof}
%
%
%----------------------------------------
% APPENDIX B: PROOF OF PROPOSITION 4.15
%----------------------------------------
%
%
\section{Proof of Proposition~\ref{prop:s-uniform-bound-per-site}}
\label{sec:appB}
\begin{proof}
For an update on any edge $(k,k+1)$, one has
$s_{t+1}(k)+s_{t+1}(k+1)=s_t(k)+s_t(k+1)$ and all other coordinates are unchanged.
Hence $\sum_{i=1}^N s_t(i)$ is invariant in time, and since $s_0\equiv0$,
\begin{equation}\label{eq:sumzero}
\sum_{i=1}^N s_t(i)=0 \qquad \text{for all }t\ge0.
\end{equation}

Fix $t$ and suppose edge $(k,k+1)$ is updated.
Write $x=s_t(k)$ and $y=s_t(k+1)$, and $m=(x+y)/2$.
Then $s_{t+1}(k)=m+\beta/2$ and $s_{t+1}(k+1)=m-\beta/2$, so
\[
s_{t+1}(k)^2+s_{t+1}(k+1)^2
=
2m^2+\frac{\beta^2}{2},
\qquad
x^2+y^2
=
2m^2+\frac{(x-y)^2}{2}.
\]
Therefore
\begin{equation}\label{eq:s-incr-elementary}
\|s_{t+1}\|_2^2-\|s_t\|_2^2
=
-\frac12\big(s_t(k)-s_t(k+1)\big)^2+\frac{\beta^2}{2}.
\end{equation}

Conditionally on $\mathcal F_t$, the index $k$ is uniform in $\{1,\dots,N\}$, hence
\begin{equation}\label{eq:cond-drift}
\mathbb E\!\left[\|s_{t+1}\|_2^2 \mid \mathcal F_t\right]
=
\|s_t\|_2^2
-\frac{1}{2N}\sum_{i=1}^N\big(s_t(i)-s_t(i+1)\big)^2
+\frac{\beta^2}{2},
\end{equation}
where indices are understood modulo $N$.

Set $G_t:=\sum_{i=1}^N (s_t(i)-s_t(i+1))^2$ and
$\operatorname{diam}(s_t):=\max_i s_t(i)-\min_i s_t(i)$.
Along the ring,
\[
\operatorname{diam}(s_t)
\le \sum_{i=1}^N |s_t(i)-s_t(i+1)|
\le \sqrt{N}\,G_t^{1/2},
\]
where the second inequality is Cauchy--Schwarz.
On the other hand, since $\sum_i s_t(i)=0$ by \eqref{eq:sumzero},
there exist indices $i_+,i_-$ with $s_t(i_+)\ge0\ge s_t(i_-)$, hence
$\operatorname{diam}(s_t)\ge \max_i |s_t(i)|$ and therefore
\[
\|s_t\|_2^2 \le N\big(\max_i |s_t(i)|\big)^2 \le N\,\operatorname{diam}(s_t)^2.
\]
Combining the two relations yields
\[
\|s_t\|_2^2 \le N\big(\sqrt{N}\,G_t^{1/2}\big)^2 = N^2\,G_t,
\qquad\text{so}\qquad
G_t \ge \frac{1}{N^2}\,\|s_t\|_2^2.
\]
Substituting this into \eqref{eq:cond-drift} gives
\begin{equation}
\label{e:b_comp}
\mathbb E\!\left[\|s_{t+1}\|_2^2 \mid \mathcal F_t\right]
\le
\Big(1-\frac{1}{2N^3}\Big)\|s_t\|_2^2+\frac{\beta^2}{2}.
\end{equation}
Iterating \eqref{e:b_comp} and using $s_0\equiv0$ gives
\[
\mathbb E\|s_t\|_2^2
\le
\frac{\beta^2/2}{1/(2N^3)}
=
\beta^2\,N^3\leq CN.
\]

\end{proof}

%
%
% REFERENCES
%
%

\bibliographystyle{abbrvnat}
\bibliography{main}
\end{document}